\documentclass[review,onefignum,onetabnum]{siamart171218}

\usepackage{lipsum}
\usepackage{amsfonts}
\usepackage{graphicx}
\usepackage{epstopdf}
\usepackage{algorithmic}
\ifpdf
  \DeclareGraphicsExtensions{.eps,.pdf,.png,.jpg}
\else
  \DeclareGraphicsExtensions{.eps}
\fi

\newsiamremark{remark}{Remark}
\newsiamremark{hypothesis}{Hypothesis}
\crefname{hypothesis}{Hypothesis}{Hypotheses}
\newsiamthm{claim}{Claim}

\usepackage{amsopn}

\usepackage[percent]{overpic}
\usepackage{bm,color}
\usepackage{float}
\usepackage{subfig}

\usepackage{cleveref}

\usepackage{amssymb}        % for \varnothing
 \usepackage{tikz}
\usepackage{verbatim}
\usetikzlibrary{arrows,backgrounds,snakes,shapes}
\theoremstyle{plain}
\newtheorem{thm}{Theorem}[section]

\theoremstyle{plainNoItalics}

\newtheorem{example}[thm]{Example}

\graphicspath{{./}{./figures/}}

\ifpdf
\hypersetup{
  pdftitle={EL-WENO },
  pdfauthor={N. Zheng, X. Cai, J.-M. Qiu, J. Qiu}
}
\fi

\begin{document}

%\maketitle

\title{Non-splitting Eulerian-Lagrangian WENO schemes for two-dimensional nonlinear convection-diffusion equations%
  \thanks{%
\funding{J. Qiu was partially supported by National Key R\&D Program of China
(No. 2022YFA1004500);
 N. Zheng and J.-M. Qiu were supported by the National Science Foundation NSF-DMS-2111253, Air Force
Office of Scientific Research FA9550-22-1-0390, and the Department of Energy DE-SC0023164;
X. Cai was supported by the National Natural Science Foundation (China) (No. 12201052), the Guangdong Provincial Key Laboratory of Interdisciplinary Research and Application for Data Science, BNU-HKBU United International College, project code 2022B1212010006.
 }} }

%and the National Natural Science Foundation (China) (No.12071392)

\author{
Nanyi Zheng%
\thanks{Department of Mathematical Sciences, University of Delaware, Newark, DE, 19716, USA. (\email{nyzheng@udel.edu}).}
\and
Xiaofeng Cai%%
  \thanks{Research Center for Mathematics, Advanced Institute of Natural Sciences, Beijing Normal University, Zhuhai, 519087, China, and
  Guangdong Provincial Key Laboratory of Interdisciplinary Research and Application for Data Science, BNU-HKBU United International College, Zhuhai, 519087, China. (\email{xfcai@bnu.edu.cn}).}
  \and
  Jing-Mei Qiu%
  \thanks{Department of Mathematical Sciences, University of Delaware, Newark, DE, 19716, USA. (\email{jingqiu@udel.edu}).}
    \and
 Jianxian Qiu%
  \thanks{School of Mathematical Sciences and Fujian Provincial Key Laboratory of Mathematical Modeling and High-Performance Scientific Computing, Xiamen University, Xiamen, Fujian 361005, China. (\email{jxqiu@xmu.edu.cn}).}
}

% Custom SIAM macro to insert headers
%\headers{An Eulerian-Lagrangian discontinuous Galerkin method for transport problems\\
% and its application to nonlinear dynamics}
 \headers{Eulerian-Lagrangian WENO schemes}
{ N. Zheng, X. Cai,  J.-M. Qiu,  J. Qiu}
\maketitle

 \begin{abstract}
 In this paper, we develop high-order, conservative, non-splitting  Eulerian-Lagrangian (EL) Runge-Kutta (RK) finite volume (FV) weighted essentially non-oscillatory (WENO) schemes for convection-diffusion equations. The proposed EL-RK-FV-WENO scheme defines modified characteristic lines and evolves the solution along them, significantly relaxing the time-step constraint for the convection term.
The main algorithm design challenge arises from the complexity of constructing accurate and robust reconstructions on dynamically varying Lagrangian meshes. This reconstruction process is needed for flux evaluations on time-dependent upstream quadrilaterals and time integrations along moving characteristics.
To address this, we propose a strategy that utilizes a WENO reconstruction on a fixed Eulerian mesh for spatial reconstruction, and updates intermediate solutions on the Eulerian background mesh for implicit-explicit  RK temporal integration. This strategy leverages efficient reconstruction and remapping algorithms to manage the complexities of polynomial reconstructions on time-dependent quadrilaterals, while ensuring local mass conservation.
%Challenges in high-order polynomial reconstruction on dynamic Lagrangian upstream cells and time integration complexities are addressed. The proposed method leverages efficient polynomial reconstruction on a fixed Eulerian mesh, serving as a foundation for computations on distorted quadrilaterals. The proposed EL-RK-FV-WENO scheme is integrated with the implicit-explicit  method, allowing for the implicit evolution of the diffusion term and explicit handling of the convection term.
The proposed scheme ensures mass conservation due to the flux-form semi-discretization and the mass-conservative reconstruction on both background and upstream cells. Extensive numerical tests have been performed to verify the effectiveness of the proposed scheme.

 \end{abstract}

% REQUIRED

% REQUIRED

\section{Introduction}

%Scientific computing has evolved significantly over the years, becoming a cornerstone in advancing various scientific fields. One of the pivotal challenges in this domain is simulating incompressible

Simulating convection-diffusion phenomena has a wide range of applications, including fluid dynamics \cite{yang_generalized_2014,lee_multiscale_2016,chai_multiple-distribution-function_2022}, materials science \cite{sorensen_computational_1999,wang_impact_2019,boulais_two-dimensional_2020}, and geophysics \cite{smolarkiewicz_mpdata_1998,singh_advection_2013}. In this paper, we consider a scalar convection-diffusion equation:
\begin{equation}\label{eq:general_conv_diff}
    u_t + \nabla_{\mathbf{x}}\cdot(\mathbf{F}(u,\mathbf{x},t))=\epsilon \Delta u,
\end{equation}
where $\epsilon \geq 0$. Existing methods include the Eulerian \cite{cockburn1998local, hidalgo2011ader} and Lagrangian approaches \cite{cheng2007high,benitez2012numericalI,benitez2012numericalII,benitez2014pure}. The Eulerian approach evolves the equation upon a fixed spatial mesh (or Eulerian mesh), and such methods are usually robust and relatively easy to implement, but they suffer from time-step constraints.  The Lagrangian approach follows characteristics in time evolution by generating a Lagrangian mesh that moves with the velocity field, allowing for a larger time-step size compared with Eulerian schemes. However, the moving Lagrangian mesh can be greatly distorted, leading to significant challenges in analysis and implementation. Between the two approaches, there are the Eulerian-Lagrangian (EL) approach \cite{wang1999ellam,arbogast2010fully,nakao2022eulerian}, the semi-Lagrangian (SL) approach \cite{xiu2001semi,spiegelman2006semi,ding_semi-lagrangian_2020}, and the arbitrary Lagrangian-Eulerian (ALE) approach \cite{braescu2007arbitrary,mackenzie2012unconditionally,zhou2022arbitrary}. Both the EL and SL approaches utilize a fixed background mesh and accurately or approximately track information propagation along characteristics, which helps them ease the numerical time-step constraint. The schemes from the ALE approach consider a dynamically moving mesh. These approaches aim to balance between the Eulerian and Lagrangian approaches in various ways, tailored for better efficiency of computational algorithms in different settings.

%Eulerian-Lagrangian (EL) approach \cite{wang1999ellam,arbogast2010fully,nakao2022eulerian}, the semi-Lagrangian (SL) approach \cite{xiu2001semi,spiegelman2006semi,ding_semi-lagrangian_2020}, and the arbitrary Lagrangian-Eulerian (ALE) approach \cite{braescu2007arbitrary,mackenzie2012unconditionally,zhou2022arbitrary} have been developed to strike a fine balance between Eulerian and Lagrangian approaches.
%The former two approaches use a fixed background mesh, and accurately or approximately track information propagation along characteristics, easing the numerical time-step constraint.% The schemes from the ALE approach consider a dynamically moving mesh or a particle method based on the characteristics. All these approaches seek a good balance between the Eulerian approach and the Lagrangian approach in different ways.

In this paper, we continue our development of the EL Runge-Kutta (RK) schemes \cite{cai_eulerian-lagrangian_2021, nakao2022eulerian}, but now in a truly multi-dimensional finite volume (FV) fashion for nonlinear convection-diffusion problems. In the finite volume setting, we need to update only one degree of freedom per cell, as opposed to multiple ones, compared with the previous EL-RK discontinuous Garlerkin (DG) scheme \cite{cai_eulerian-lagrangian_2021}. Building upon the EL-RK framework, we introduce a modified velocity field as a first-order approximation of the analytic velocity field, offering triple benefits: firstly, the modified velocity field has straight characteristic lines, leading to upstream cells with straight edges that are easier to approximate than polygons with curved edges; secondly, tracking characteristics approximately allows a greatly relaxed time-stepping constraint compared with Eulerian methods; thirdly, the EL framework offers flexibility in treating nonlinearity, while integrating diffusion terms, thereby presenting a truly multi-dimensional EL finite volume scheme compared to our earlier work in \cite{ding_semi-lagrangian_2020, nakao2022eulerian}.

Yet new challenges arise in reconstructing high-order polynomials on dynamically varying Lagrangian upstream cells; robust and accurate weighted essentially non-oscillatory (WENO) reconstructions on time-dependent upstream polygons can be computationally complex and expensive. Furthermore, performing high-order time integration along moving characteristics brings new complications in algorithm design. Below, we elaborate major computational roadblocks and our proposed strategy in the following two aspects:
\begin{itemize}
    \item {\em Spatial reconstruction.} The EL RK formulation necessitates flux evaluations at the interface of upstream cells; thus, we need to reconstruct piecewise polynomials on upstream quadrilaterals. Performing WENO reconstruction of polynomials on distorted upstream quadrilaterals, e.g., see the red mesh in \Cref{fig:schematic_remap_poly_b}, can be computationally involved. Further, shapes of these upstream cells differ in every time-step, leading to expensive mesh-dependent local computations. To address such challenges, we propose to (a) perform a robust and efficient WENO reconstruction of piecewise polynomials on the background Eulerian mesh; and (b) leverage a remapping algorithm to compute cell averages on upstream cells from cell averages on the background Eulerian mesh in a mass conservative fashion \cite{lauritzen_conservative_2010,zheng_fourth-order_2022}. Finally, we perform piecewise polynomial reconstruction on upstream quadrilaterals, with preservation of cell averages computed in (b), while utilizing the piecewise polynomials on Eulerian mesh reconstructed in (a) for accuracy consideration.
    \item {\em Implicit-explicit (IMEX) RK temporal integration along linear approximation of characteristics.} A major computational challenge in performing method-of-lines time integrations along moving meshes is the complexity again in polynomial reconstructions of solutions on quadrilateral meshes that are time varying.
    To address this issue, we propose to update intermediate IMEX solutions at the background Eulerian mesh as in \cite{ding_semi-lagrangian_2020, nakao2022eulerian}, for which efficient reconstruction and the remapping algorithms can be utilized to faciliate the polynomial reconstruction on time-dependent quadrilaterals as mentioned above.
\end{itemize}
We emphasize that efficient polynomial reconstruction on a fixed Eulerian mesh serves as a cornerstone in our EL-RK algorithm, upon which polynomial reconstructions on distorted upstream quadrilaterals are performed. Indeed, WENO reconstructions on a background Eulerian mesh have been well developed in the literature \cite{levy2000compact,dumbser_arbitrary_2007,zhu_new_2016,cravero2019optimal,zheng_fourth-order_2022}. In this paper, we further improve upon our previous work \cite{zheng_fourth-order_2022} and propose a new 2D WENO reconstruction. This new approach strikes a good balance between controlling numerical oscillations and achieving optimal accuracy, by optimizing small stencil polynomial approximations and the weighting strategy.

 %This novel strategy capitalizes on the EL-RK framework in the sense that numerical solutions are always readily available with respect to the fixed Eulerian mesh, where efficient reconstruction can be consistently applied. For the time discretization method, we employ the implicit-explicit (IMEX) method \cite{ascher_implicit_explicit_1995,ascher_implicit_explicit_1997} to implicitly evolve the diffusion term and explicitly handle the convection term.

The rest of the paper is organized as follows. \Cref{sec:EL_FV_WENO_scheme} presents the proposed EL-RK-FV-WENO schemes; \Cref{sec:numerical_tests} presents extensive numerical results showcasing the scheme's effectiveness. Finally, we conclude in \Cref{sec:conclusion}.

\section{EL-RK-FV-WENO schemes}\label{sec:EL_FV_WENO_scheme}

In \Cref{sec:first_order_EL_scheme}, we introduce a first-order EL-RK-FV scheme for a linear convection-diffusion equation. Then,  building upon the basic concepts introduced in \Cref{sec:first_order_EL_scheme}, we discuss the construction of high-order EL-RK-FV-WENO schemes in \Cref{sec:high_order_EL_scheme}. Finally, the extension of the proposed EL-RK-FV-WENO scheme to a nonlinear model is presented in \Cref{sec:nonlinear}.

\subsection{First-order EL-RK-FV scheme}\label{sec:first_order_EL_scheme}
Consider
\begin{equation}\label{eq:2D_linear_convection_equation}
	u_t + (a(x,y,t)u )_x + (b(x,y,t)u)_y = \epsilon(u_{xx}+u_{yy}).
\end{equation}
We assume a rectangle computational domain denoted by $\Omega:=[x_L,x_R]\times[y_B,y_T]$ with following partitions for each dimension
$$x_L=x_{\frac12}<x_{\frac32}<\ldots<x_{i-\frac12}<x_{i+\frac12}<\ldots<x_{N+\frac12}=x_R,$$
$$y_B=y_{\frac12}<y_{\frac32}<\ldots<y_{j-\frac12}<y_{j+\frac12}<\ldots<y_{N+\frac12}=y_T$$
with $x_i:=(x_{i-\frac12}+x_{i+\frac12})/2$,  $y_j:=(y_{j-\frac12}+y_{j+\frac12})/2$, $\Delta x_i=x_{i+\frac12}-x_{i-\frac12}$, $\Delta y_j = y_{j+\frac12} - y_{j-\frac12}$, $I_{i}^x:=[x_{i-\frac12},x_{i+\frac12}]$, $I_{j}^y:=[y_{j-\frac12},y_{j+\frac12}]$ and $I_{i,j}:=I_i^x\times I_j^y,\quad \forall i,j$. We define the numerical solutions on the Eulerian mesh as $\{\overline{u}^n_{i,j}\}$, which approximate the averages of the $u(x,y,t^n)$ over the Eulerian cells $\{I_{i,j}\}$, i.e. $\{\frac{1}{|I_{i,j}|}\iint_{I_{i,j}}u(x,y,t^n)dxdy\}$.

To derive an EL-RK-FV formulation, we first define a modified velocity field $(\alpha(x,y,t),\beta(x,y,t))$. The definition of $(\alpha(x,y,t),\beta(x,y,t))$ is summarized as follows.

\begin{enumerate}
	\item At $t=t^{n+1}$, $\alpha(x,y,t^{n+1})$ and $\beta(x,y,t^{n+1})$ belong to $Q^1(I_{i,j})$ satisfying
	\begin{equation}\label{eq:alpha_beta_tnp1}
		\begin{split}
			\alpha(x_{i\pm\frac12},y_{j\pm\frac12},t^{n+1})=a(x_{i\pm\frac12},y_{j\pm\frac12},t^{n+1}),\\
			\beta(x_{i\pm\frac12},y_{j\pm\frac12},t^{n+1})=b(x_{i\pm\frac12},y_{j\pm\frac12},t^{n+1}).
		\end{split}
	\end{equation}
	\item We define a dynamic region (see \Cref{fig:dynamic_region})
	 \begin{equation}\label{eq:dynamic_region}
	 	\widetilde{I}_{i,j}(t):=\{(x,y)|(x,y)=(\widetilde{x}(t;(\xi,\eta,t^{n+1})),\widetilde{y}(t;(\xi,\eta,t^{n+1}))), (\xi,\eta)\in I_{i,j}\},
	 \end{equation}
 	where $(\widetilde{x}(t;(\xi,\eta,t^{n+1})),\widetilde{y}(t;(\xi,\eta,t^{n+1})))$ represents the straight line going through $(\xi,\eta,t^{n+1})$ satisfying
	\begin{equation}\label{eq:modified_characteristics}
		\begin{cases}
			\widetilde{x}(t;(\xi,\eta,t^{n+1})) = \xi + (t-t^{n+1})\alpha(\xi,\eta,t^{n+1}),\\
			\widetilde{y}(t;(\xi,\eta,t^{n+1})) = \eta + (t-t^{n+1})\beta(\xi,\eta,t^{n+1}).
		\end{cases}
	\end{equation}
	We call \eqref{eq:modified_characteristics} a modified characteristic line.
	\item For $t\in[t^n,t^{n+1})$ and $(\widetilde{x}(t;(\xi,\eta,t^{n+1})),\widetilde{y}(t;(\xi,\eta,t^{n+1})))\in \widetilde{I}_{i,j}(t)$,
	\begin{equation}\label{eq:modified_velocity}
		\begin{cases}
			\alpha(\widetilde{x}(t;(\xi,\eta,t^{n+1})),\widetilde{y}(t;(\xi,\eta,t^{n+1}))) = \alpha(\xi,\eta,t^{n+1}),\\
			\beta(\widetilde{x}(t;(\xi,\eta,t^{n+1})),\widetilde{y}(t;(\xi,\eta,t^{n+1}))) = \beta(\xi,\eta,t^{n+1}).
		\end{cases}
	\end{equation}
\end{enumerate}

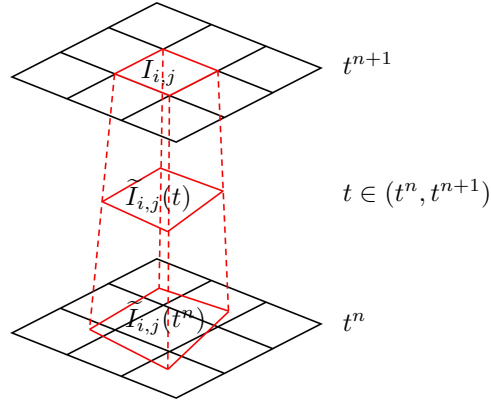
\begin{figure}[htb]
	\begin{center}
		\begin{overpic}[scale=0.3]{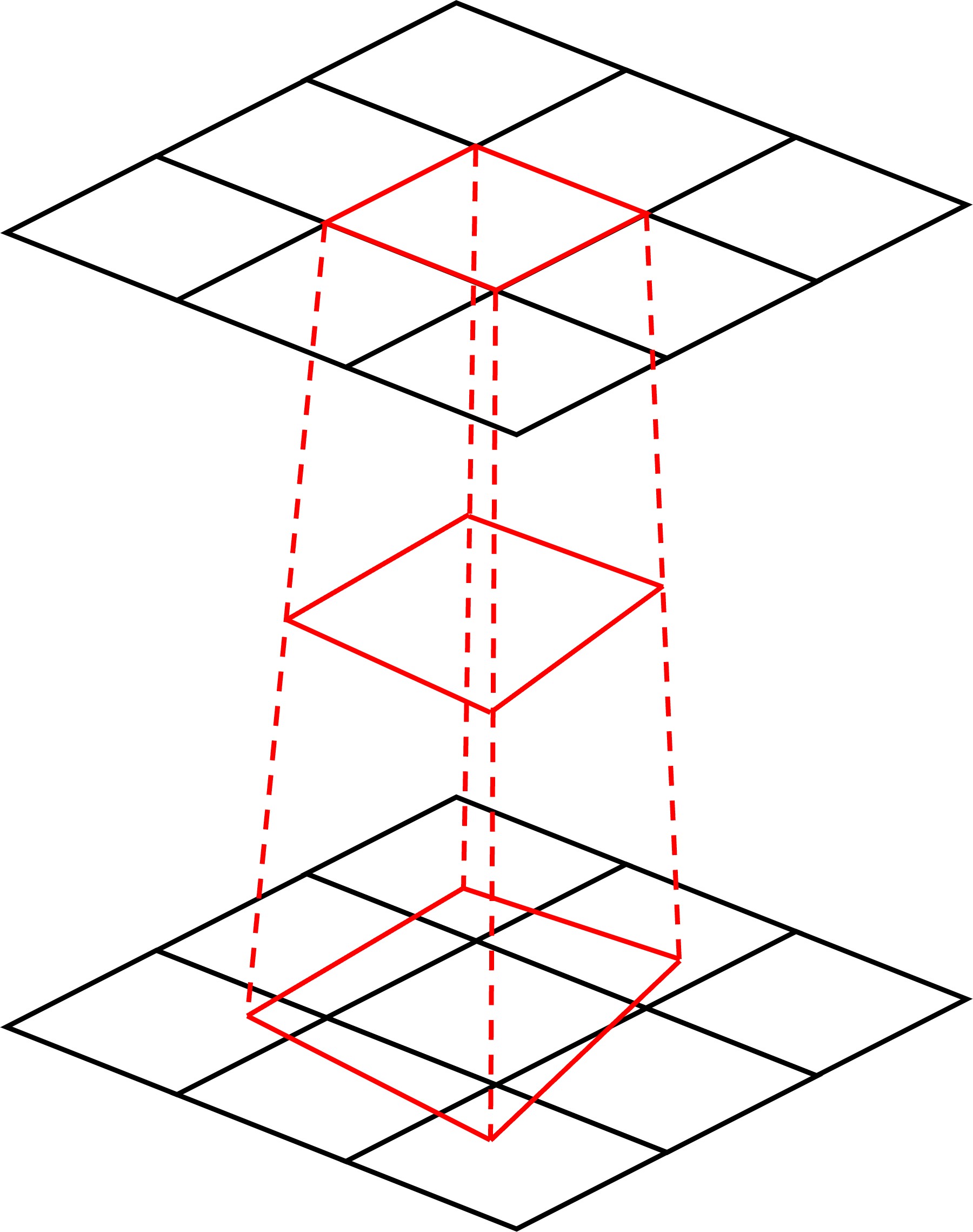}%,mesh,tics=5
			\put(33,81){$I_{i,j}$}
			\put(29,18){$\widetilde{I}_{i,j}(t^n)$}
			\put(29,49){$\widetilde{I}_{i,j}(t)$}
			\put(84,17) {$t^n$}
			\put(84,50) {$t\in(t^n,t^{n+1})$}
			\put(84,82){$t^{n+1}$}
		\end{overpic}
	\end{center}
	\caption{Schematic illustration for the dynamic region $\widetilde{I}_{i,j}(t)$.}
	\label{fig:dynamic_region}
\end{figure}

With the definitions above, we can derive
\begin{equation}\label{EL_FV_formulation_derivation}
	\begin{split}
	&\frac{d}{dt}\iint_{\widetilde{I}_{i,j}(t)}u(x,y,t)dxdy\\
   =&\iint_{\widetilde{I}_{i,j}(t)} u_t(x,y,t)dxdy + \int_{\partial \widetilde{I}_{i,j}(t)}(\alpha,\beta)u\cdot \mathbf{n}ds\\
   =&\iint_{\widetilde{I}_{i,j}(t)} u_t(x,y,t)dxdy + \iint_{\widetilde{I}_{i,j}(t)}\left[(au)_x+(bu)_y\right]dxdy\\
    &-\int_{\partial \widetilde{I}_{i,j}(t)}(a,b)u\cdot\mathbf{n}ds+\int_{\partial \widetilde{I}_{i,j}(t)}(\alpha,\beta)u\cdot \mathbf{n}ds\\
   =&-\int_{\partial \widetilde{I}_{i,j}(t)}(a-\alpha,b-\beta)u\cdot\mathbf{n}ds+\epsilon\iint_{\widetilde{I}_{i,j}(t)}\Delta udxdy.
   \end{split}
\end{equation}
We define that $\mathbf{F}(u,x,y,t):=((a-\alpha)u,(b-\beta)u)$ and provide the concise EL-FV formulation:
\begin{equation}\label{EL_FV_formulation}
    \begin{split}
	\frac{d}{dt}\iint_{\widetilde{I}_{i,j}(t)}u(x,y,t)dxdy &= -\int_{\partial \widetilde{I}_{i,j}(t)}\mathbf{F}(u,x,y,t)\cdot\mathbf{n}ds+\epsilon\iint_{\widetilde{I}_{i,j}(t)}\Delta udxdy\\
        &=:\mathcal{F}_{i,j}(u;t)+\mathcal{G}_{i,j}(u;t).
    \end{split}
\end{equation}

%\subsection{Semi-discretization}\label{sec:semi_discretization}
To evaluate the right-hand side (RHS) of \eqref{EL_FV_formulation}, we introduce the following notation for semi-discretization:
\begin{equation}\label{eq:semi_discretization}
    \frac{d\widetilde{u}_{i,j}(t)}{dt}=\widetilde{\mathcal{F}}_{i,j}(\overline{\mathbf{U}};t)+\widetilde{\mathcal{G}}_{i,j}(\overline{\mathbf{U}};t),
\end{equation}
where
\begin{itemize}
    \item the notation $\widetilde{\cdot}$ specifies that the integral value corresponds to the characteristic spatial region $\widetilde{I}_{i,j}(t)$,
    \item $\widetilde{u}_{i,j}(t)$ approximates $\iint_{\widetilde{I}_{i,j}(t)}u(x,y,t)dxdy$,
    \item $\widetilde{\mathcal{F}}_{i,j}(\overline{\mathbf{U}};t)$ approximates $\mathcal{F}_{i,j}(u;t)$,
    \item $\widetilde{\mathcal{G}}_{i,j}(\overline{\mathbf{U}};t)$ approximates $\mathcal{G}_{i,j}(u;t)$,
    \item $\overline{\mathbf{U}}:=\left(\overline{u}_{i,j}(t)\right)_{NxNy}$ represents the finite volumes such that
    \begin{equation}
        \overline{u}_{i,j}(t)\approx\frac{1}{|I_{i,j}|}\iint_{I_{i,j}}u(x,y,t)dxdy.
    \end{equation}
\end{itemize}

Similar to $\overline{\mathbf{U}}$, we can also represent \eqref{eq:semi_discretization} globally as follows:
\begin{equation}\label{eq:vec_semi_discretization}
    \frac{d\widetilde{\mathbf{U}}(t)}{dt}=\bm{\widetilde{\mathcal{F}}} (\overline{\mathbf{U}};t)+\bm{\widetilde{\mathcal{G}}}(\overline{\mathbf{U}};t),
\end{equation}
where $\bm{\widetilde{\mathcal{F}}}:=\left(\widetilde{\mathcal{F}}_{i,j}\right)_{N_xNy}$ and 
$\bm{\widetilde{\mathcal{G}}} :=\left(\widetilde{\mathcal{G}}_{i,j}\right)_{N_xNy}$. Coupling \eqref{eq:vec_semi_discretization}  with the first-order forward-backward Euler IMEX method in \cite{ascher_implicit_explicit_1997} yields the first-order EL-RK-FV scheme:
\begin{align}\label{eq:scheme_IMEX111}
    \mathbf{M}\overline{\mathbf{U}}^{n+1} = \widetilde{\mathbf{U}}^n + \Delta t\bm{\widetilde{\mathcal{F}}}(\overline{\mathbf{U}}^n,t^n)+\Delta t\left(\epsilon\mathbf{M}\mathbf{D}\overline{\mathbf{U}}^{n+1}\right),
\end{align}
where \begin{itemize}
    \item $\mathbf{M}$ is a diagonal matrix such that $\mathbf{M}\overline{\mathbf{U}}^{n+1} = \left(|I_{i,j}|\overline{u}^{n+1}_{i,j}\right)_{N_xN_y}$,
    \item $\widetilde{\mathbf{U}}^n$ approximates $\left(\iint_{\widetilde{I}_{i,j}(t^n)}u(x,y,t^n)dxdy\right)_{N_xN_y}$,
    \item $\mathbf{D}$ is a differential matrix such that
    $$\mathbf{D}\overline{\mathbf{U}}^{n+1} \text{approximates}  \left(\frac{1}{|I_{i,j}|}\iint_{I_{i,j}}\Delta udxdy\right)_{N_xN_y}.$$
\end{itemize}
As shown, the diffusion term in \eqref{eq:scheme_IMEX111} is implicit. Consequently, $\overline{\mathbf{U}}^{n+1}$ is obtained by solving the following linear system:
\begin{equation}
    \mathbf{M}\left(\mathbf{I} - \Delta t\epsilon\mathbf{D}\right)\overline{\mathbf{U}}^{n+1} = \widetilde{\mathbf{U}}^n + \Delta t\bm{\widetilde{\mathcal{F}}}(\overline{\mathbf{U}}^n,t^n).
\end{equation}

In \eqref{eq:scheme_IMEX111}, the methods for approximating the $\widetilde{\mathcal{F}}$ and $\widetilde{\mathcal{G}}$ terms with first-order accuracy are not discussed. High-order spatial approximations for these terms will be introduced in the next section. The first-order approximations can be viewed as simplified versions of their high-order counterparts.

%\textcolor{red}{Subsequently, we describe the constructions of \(\widetilde{\mathcal{F}}_{i,j}(\cdot,\cdot)\) in \Cref{sec:flux} and of \(\widetilde{\mathcal{G}}_{i,j}(\cdot,\cdot)\) in \Cref{sec:diff_source}.}

\begin{remark}(Empirical time-step constraint of the convection term for stability)\label{rem:time_step_constraint}
Similar to the flux-form finite volume method in \cite{shu_essentially_1998}, where $\alpha=\beta=0$, we require that
\begin{equation}\label{eq:empirical_time_step_estimate_1}
	\Delta t \leq \frac{1}{\frac{\max|a-\alpha|}{\Delta x} + \frac{\max|b-\beta|}{\Delta y}} = \frac{\Delta x\Delta y}{\Delta y\max|a-\alpha| + \Delta x\max|b-\beta|}
\end{equation}
with $\Delta x := \max\{\Delta x_i\}$ and $\Delta y := \max\{\Delta y_j\}.$
Furthermore, we stipulate that $\widetilde{I}_{i,j}(t^n)$ remains a convex quadrilateral. Otherwise, $\widetilde{I}_{i,j}(t^n)$ might become ill-posed in various situations. Therefore, it suffices to require that any three vertices of $\widetilde{I}_{i,j}(t^n)$ cannot be collinear. In other words,
\begin{equation}\label{eq:empirical_time_step_estimate_2}
	\iint_{\triangle_{l}(\widetilde{I}_{i,j}(t))} dxdy > 0, \quad l\in\{LT,RT,LB,RB\},
\end{equation}
where
\begin{align}
	\triangle_{LT}(\widetilde{I}_{i,j}(t)):=\{(x,y)|(x,y)=(\widetilde{x}(t;(\xi,\eta,t^{n+1})),\widetilde{y}(t;(\xi,\eta,t^{n+1}))),\\
	y_{j-\frac{1}{2}}+(\xi-x_{i-\frac{1}{2}})\Delta y_j/\Delta x_i \leq \eta \leq y_{j+\frac{1}{2}},~\xi\in [x_{i-\frac{1}{2}},x_{i+\frac{1}{2}}] \}.
\end{align}
Similar definitions hold for the other $\{\triangle_{l}(\widetilde{I}_{i,j}(t))\}$. Through tedious derivation and omitting some higher-order terms, we can establish that \eqref{eq:empirical_time_step_estimate_2} implies
\begin{equation}\label{eq:empirical_time_step_estimate_3}
	\Delta t < \frac{1}{|a_x(x_i,y_j,t^{n+1})| + |b_y(x_i,y_j,t^{n+1})|}.
\end{equation}
By combining \eqref{eq:empirical_time_step_estimate_1} with \eqref{eq:empirical_time_step_estimate_3}, and considering that $|a-\alpha|, |b-\beta| = O(\Delta t) + O(\Delta x^2) + O(\Delta y^2)$, we arrive at an approximate time-step constraint:
\begin{equation}\label{eq:time_constraint_es}
	\Delta t \sim \sqrt{\min\{\Delta x,\Delta y\}}.
\end{equation}
\end{remark}

\subsection{High-order EL-RK-FV-WENO schemes}\label{sec:high_order_EL_scheme}
In this section, we introduce the high-order EL-RK-FV-WENO schemes by first focusing on constructing the high-order approximations of the convection and the diffusion terms of the semi-discretization \eqref{eq:semi_discretization} in \Cref{sec:flux} and \Cref{sec:diff_source} respectively. Finally, the coupling of the semi-discretization with high-order IMEX RK methods is discussed in \Cref{sec:IMEX}.
\subsubsection{Flux approximation}\label{sec:flux}
As analyzed in \Cref{rem:time_step_constraint}, the design of the flux function at the boundaries of dynamically changing, nonuniform Lagrangian cells significantly relaxes the time-step constraint. However, compared to the flux function of the Eulerian approach, located at $\{\partial I_{i,j}\}$, this brings significant challenges in terms of designing an efficient spatial discretization for such a framework. To address this, our strategy contains three basic steps. First, we conduct an efficient WENO-type reconstruction on the fixed Eulerian mesh. Second, an efficient remapping procedure is designed to map the piecewise WENO reconstruction polynomial with respect to the Eulerian mesh $\{I_{i,j}\}$ to another piecewise polynomial with respect to the Lagrangian mesh $\{\widetilde{I}_{i,j}(t)\}$. Finally, we use the new piecewise polynomial to provide upwind point values at the boundaries of $\{\partial \widetilde{I}_{i,j}(t)\}$ and approximate the flux function. The details of these three steps are concluded as follows:

%The stability of a well-designed flux approximation depends on the correct ``wind direction", which determines the involvement of solutions from neighboring finite volumes. Through numerical investigations, we have observed that the numerical approximation of $u$ requires ``jumps" across $\partial \widetilde{I}_{i,j}(t)$ to ensure stability.

%To achieve this, we design our flux approximation by the following three steps.
\begin{enumerate}
    \item[\textbf{Step 1:}] \textit{Construct a piecewise polynomial with respect to the Eulerian mesh, $\{I_{i,j}\}$.}

    We construct a piecewise reconstruction polynomial $u^{\text{WENO}}(x,y)$ such that
    \begin{itemize}
        \item $u^{\text{WENO}}(x,y)|_{I_{i,j}} = u^{\text{WENO}}_{i,j}(x,y)$ with $u^{\text{WENO}}_{i,j}\in P^2(I_{i,j})$ for all $i,j,$
        \item each $u^{\text{WENO}}_{i,j}$ is constructed based on the information $\overline{u}_{i,j}(t)$ along with its eight neighbor finite volumes,
        \item $\iint_{I_{i,j}} u^{\text{WENO}}(x,y)dxdy=\Delta x_i\Delta y_j\overline{u}_{i,j}(t)$ for all $i,j,$
        \item $u^{\text{WENO}}(x,y) = u(x,y,t) + O(\Delta x^3),$ $(x,y)\in \Omega.$
    \end{itemize}
    Here, we summarize the details of constructing \(u^{\text{WENO}}(x,y)\) in \Cref{sec:WENO_ZQ_Eulerian} for conciseness. A schematic is offered in \Cref{fig:schematic_remap_poly_a} to demonstrate the discontinuity of $u^{\text{WENO}}(x,y)$. We use various colors to shade different Eulerian cells, indicating that $u^{\text{WENO}}(x,y)$ has distinct polynomial expressions in each cell. We would like to emphasize that the WENO-ZQ reconstruction method introduced in \Cref{sec:WENO_ZQ_Eulerian} is a more advanced version than the one we previously designed, as detailed in \cite{zheng_fourth-order_2022}. We redesigned the small stencils and part of the weighting strategy. The newly designed WENO-ZQ method presented in this paper significantly enhances the control of numerical oscillation, addressing the suboptimal performance of the previously designed method in managing nonphysical oscillations.

    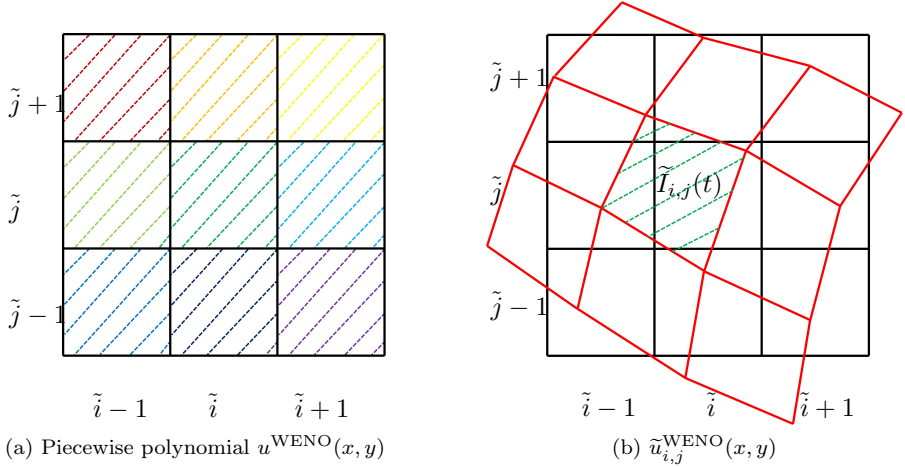
\begin{figure}[htb]
    \centering
        \subfloat[Piecewise polynomial $u^{\text{WENO}}(x,y)$]{
		\begin{overpic}[scale=0.28]{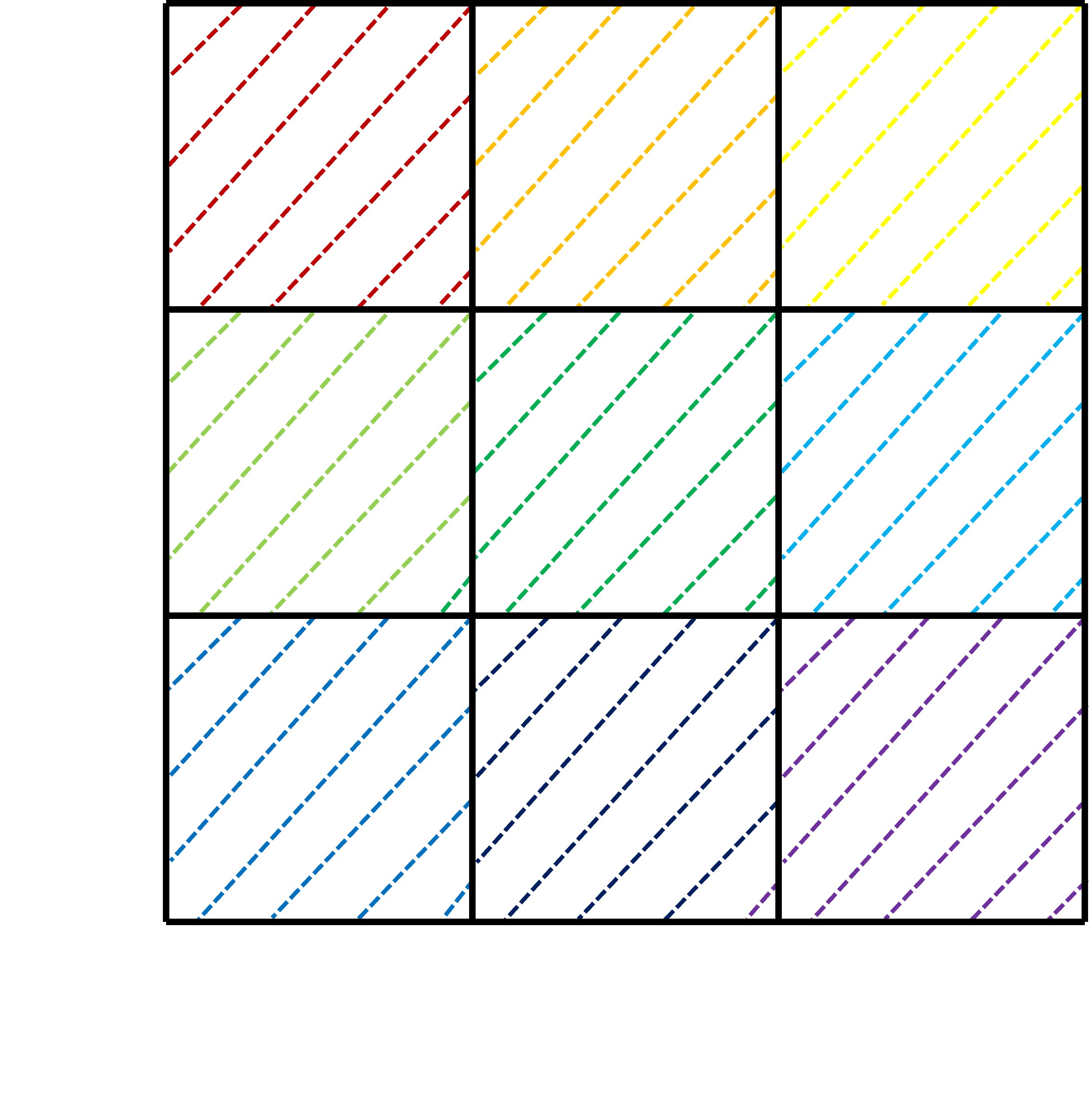}%,grid,tics=5
			\put(22.5,2){$\tilde{i}-1$}
                \put(52,2){$\tilde{i}$}
                \put(74,2){$\tilde{i}+1$}
			\put(1,26){$\tilde{j}-1$}
                \put(1,53){$\tilde{j}$}
                \put(1,80){$\tilde{j}+1$}
                %\put(40,55){$\widetilde{I}_{i,j}(t)$}
		\end{overpic}\label{fig:schematic_remap_poly_a}
	}
        \hspace{1cm}
	\subfloat[$\widetilde{u}_{i,j}^{\text{WENO}}(x,y)$]{
		\begin{overpic}[scale=0.28]{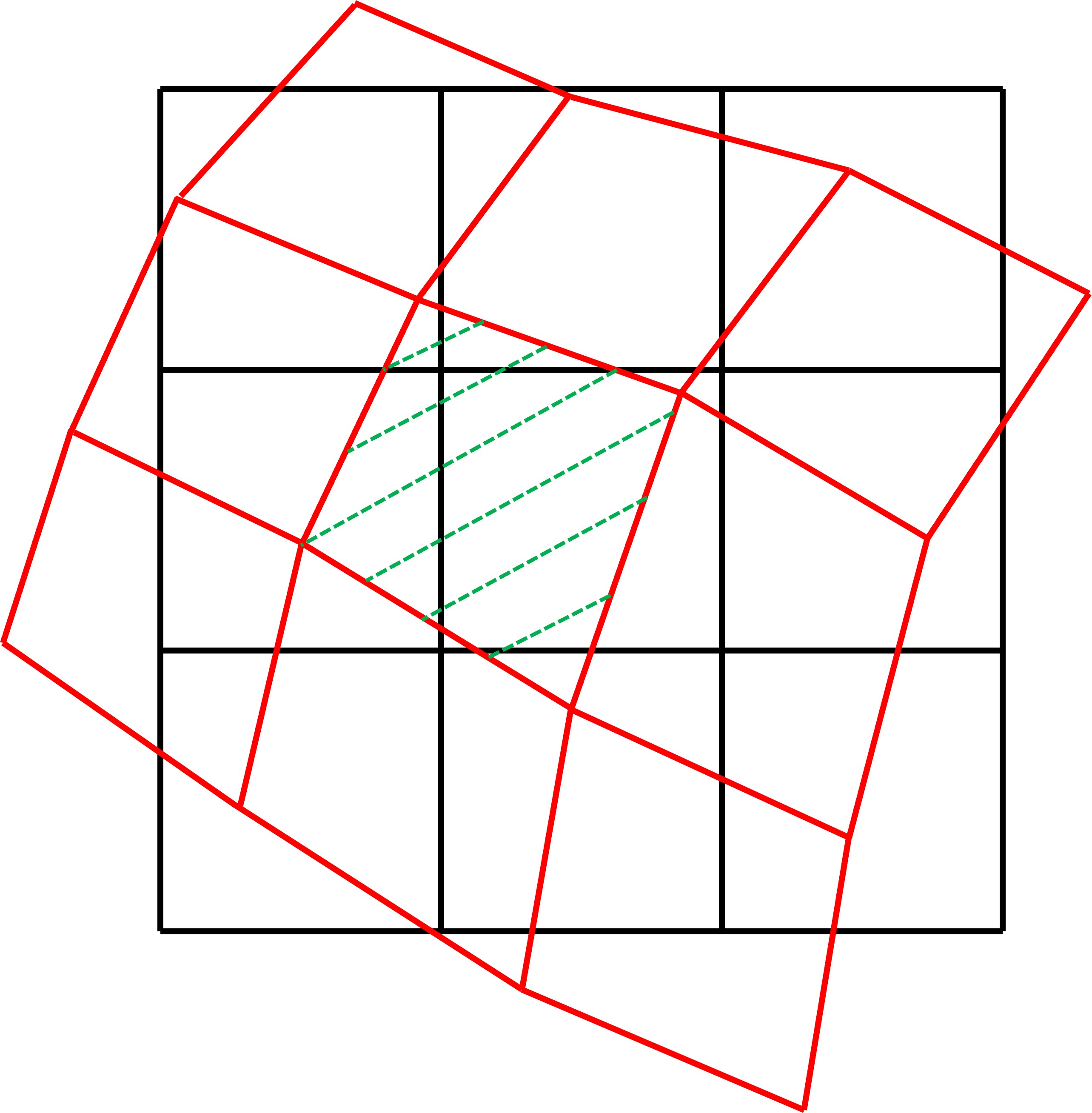}%,grid,tics=5
			\put(22.5,2){$\tilde{i}-1$}
                \put(52,2){$\tilde{i}$}
                \put(74,2){$\tilde{i}+1$}
			\put(1,26){$\tilde{j}-1$}
                \put(1,53){$\tilde{j}$}
                \put(1,80){$\tilde{j}+1$}
                \put(40,55){$\widetilde{I}_{i,j}(t)$}
		\end{overpic}\label{fig:schematic_remap_poly_b}
	}
	\caption{Schematic illustrations of $u^{\text{WENO}}(x,y)$ and $\widetilde{u}_{i,j}^{\text{WENO}}(x,y)$.}
	\label{fig:schematic_remap_poly}
    \end{figure}

    \item[\textbf{Step 2:}] \textit{Construct a Piecewise Polynomial on the Lagrangian mesh  $\{\widetilde{I}_{i,j}(t)\}$ (Remapping).}

    \textbf{Step 1} is an efficient reconstruction method and is inevitable, as will be shown in \Cref{sec:IMEX}. The basic idea of this remapping step is that we want to conduct an efficient modification to $u^{\text{WENO}}(x,y)$ instead of involving a new reconstruction on the Lagrangian mesh. The resulting new piecewise polynomial, denoted by $\widetilde{u}^{\text{WENO}}(x,y)$, satisfies the following conditions:
    \begin{itemize}
        \item $\widetilde{u}^{\text{WENO}}(x,y)|_{\widetilde{I}_{i,j}(t)} = \widetilde{u}^{\text{WENO}}_{i,j}(x,y)$ with $\widetilde{u}^{\text{WENO}}_{i,j}\in P^2(\widetilde{I}_{i,j}(t))$ for all $i,j,$
        \item $\iint_{\widetilde{I}_{i,j}(t)} \widetilde{u}^{\text{WENO}}(x,y)dxdy=\iint_{\widetilde{I}_{i,j}(t)} u^{\text{WENO}}(x,y)dxdy$ for all $i,j,$
        \item $\widetilde{u}^{\text{WENO}}(x,y) = u(x,y,t) + O(\Delta x^3),$ $(x,y)\in \Omega.$
    \end{itemize}
    Here, the first condition refers to the capacity of $\widetilde{u}^{\text{WENO}}(x,y)$ to provide upwind point value information for the flux approximation, the second condition is related to the mass conservation property, and the third one is an accuracy requirement. We summarize the procedure for constructing $\widetilde{u}^{\text{WENO}}(x,y)$ in a given Lagrangian cell, $\widetilde{I}_{i,j}(t)$, as follows (see \Cref{fig:schematic_remap_poly_b}):

    \begin{enumerate}
    	\item[\textbf{Step 2.1:}] Compute the exact ``mass" of $u^{\text{WENO}}(x,y)$ over $\widetilde{I}_{i,j}(t)$, i.e.
    	\begin{equation}\label{eq:SL_remapping}
    		\widetilde{u}_{i,j}:=\iint_{\widetilde{I}_{i,j}(t)}u^{\text{WENO}}(x,y)dxdy.
    	\end{equation}
    	The integrand in \Cref{eq:SL_remapping} is discontinuous over $\widetilde{I}_{i,j}(t)$. A numerical integral for \Cref{eq:SL_remapping} contains two basic steps. First, a clipping procedure is conducted to divide $\widetilde{I}_{i,j}(t)$ into smaller polygons such that $u^{\text{WENO}}(x,y)$ is continuous in each of them. Second, a numerical integration is conducted in each polygon, and the results are summed to obtain the final integral. Following these two basic steps, there are different implementation methods  \cite{cheng_high_2008,lauritzen_conservative_2010,zheng_fourth-order_2022}. For a detailed implementation, we refer to our previous work \cite{zheng_fourth-order_2022}.
    	\item[\textbf{Step 2.2:}] Find all Eulerian cells that intersect with $\widetilde{I}_{i,j}(t)$. We define that $\mathcal{K}:=\{(p,q)|I_{p,q}\cap \widetilde{I}_{i,j}(t)\neq \emptyset\}$. For instance, in \Cref{fig:schematic_remap_poly_b}, $\mathcal{K}=\{(\tilde{i}-1,\tilde{j}+1),(\tilde{i},\tilde{j}+1),(\tilde{i}-1,\tilde{j}),(\tilde{i},\tilde{j}),(\tilde{i},\tilde{j}-1)\}$.
    	\item[\textbf{Step 2.3:}] Compute the integrals of the candidate $P^2$ polynomials over $\widetilde{I}_{i,j}(t)$, i.e.
    	\begin{equation}
    		\widetilde{u}_{i,j}^{p,q}:=\iint_{\widetilde{I}_{i,j}(t)}u^{\text{WENO}}_{p,q}(x,y)dxdy,\quad (p,q)\in\mathcal{K}.
    	\end{equation}
    	\item[\textbf{Step 2.4:}] Choose the index $(\widetilde{p},\widetilde{q})$ such that $|\widetilde{u}^{p,q}_{i,j}-\widetilde{u}_{i,j}|$ reaches the minimum, i.e.
    	\begin{equation}
    		|\widetilde{u}^{\widetilde{p},\widetilde{q}}_{i,j}-\widetilde{u}_{i,j}| = \min_{(p,q)\in\mathcal{K}}\{|\widetilde{u}^{p,q}_{i,j}-\widetilde{u}_{i,j}|\}.
    	\end{equation}
    	\item[\textbf{Step 2.5:}] Define a $P^2$ polynomial, denoted by $\widetilde{u}^{\text{WENO}}_{i,j}(x,y)$, on $\widetilde{I}_{i,j}(t)$ such that
    	\begin{equation}\label{eq:remapping_final_formulation}
    		\widetilde{u}^{\text{WENO}}_{i,j}(x,y) := u^{\text{WENO}}_{(\widetilde{p},\widetilde{q})}(x,y)|_{\widetilde{I}_{i,j}(t)}-\frac{1}{|\widetilde{I}_{i,j}(t)|}\widetilde{u}^{(\widetilde{p},\widetilde{q})}_{i,j}+\frac{1}{|\widetilde{I}_{i,j}(t)|}\widetilde{u}_{i,j},
    	\end{equation}
    	where $\cdot|_{\widetilde{I}_{i,j}(t)}$ means that we redefine the domain of definition to be $\widetilde{I}_{i,j}(t)$ for a given function. As an example, consider the case demonstrated in \Cref{fig:schematic_remap_poly_b} where $(\tilde{p},\tilde{q}) = (\tilde{i},\tilde{j})$. In this case, we simply change the domain of definition for $u_{\tilde{i},\tilde{j}}^{\text{WENO}}(x,y)$ and adjust it according to \eqref{eq:remapping_final_formulation}.

    \end{enumerate}
    \item[\textbf{Step 3:}] \textit{Construct the final flux approximation.}

    The final flux approximation is given by the following conservative formulation
    \begin{equation}\label{eq:EL_FV_flux_approximation}
    	\widetilde{\mathcal{F}}_{i,j}(\overline{\mathbf{U}};t):=-\int_{\partial \widetilde{I}_{i,j}(t)}\hat{F}\left(\widetilde{u}^{\text{WENO}},x,y,t\right)ds,
    \end{equation}
    where
    \begin{equation}
    	\hat{F}\left(\widetilde{u}^{\text{WENO}},x,y,t\right):=W(x,y,t)\widetilde{u}_{i,j}^{\text{up}}(x,y)
    \end{equation}
    with
    \begin{equation}
    	W(x,y,t):=\left(a-\alpha,b-\beta\right)\cdot \mathbf{n}
    \end{equation}
    and
    \begin{equation}
    	\widetilde{u}_{i,j}^{\text{up}}(x,y):=
    	\begin{cases}
    		\widetilde{u}^{\text{WENO}}_{i,j}(x,y),\quad W> 0,\\
    		\widetilde{u}^{\text{WENO,ext}}_{i,j}(x,y),\quad W\leq 0.
    	\end{cases}
    \end{equation}
    Here, $\widetilde{u}^{\text{WENO,ext}}_{i,j}(x,y)$ are the exterior solution with respect to corresponding edge of $\widetilde{I}_{i,j}(t)$.
\end{enumerate}

At the end of this section, we prove two basic properties of $\widetilde{u}^{\text{WENO}}(x,y)$.
\begin{proposition}(Mass conservation for the remapping method) Local integrals of $\widetilde{u}^{\text{WENO}}$ over $\{\widetilde{I}_{i,j}(t)\}$ are consistent with the integrals of $u^{\text{WENO}}$.

\end{proposition}

\begin{proof}
    \begin{equation}
        \begin{split}
            &\iint_{\widetilde{I}_{i,j}(t)}\widetilde{u}^{\text{WENO}}(x,y)dxdy\\
            =&\iint_{\widetilde{I}_{i,j}(t)}\left(u^{\text{WENO}}_{(\widetilde{p},\widetilde{q})}(x,y)|_{\widetilde{I}_{i,j}(t)}-\frac{1}{|\widetilde{I}_{i,j}(t)|}\widetilde{u}^{(\widetilde{p},\widetilde{q})}_{i,j}+\frac{1}{|\widetilde{I}_{i,j}(t)|}\widetilde{u}_{i,j}\right)dxdy\\
            =&~~\widetilde{u}^{(\widetilde{p},\widetilde{q})}_{i,j}-\widetilde{u}^{(\widetilde{p},\widetilde{q})}_{i,j}+\widetilde{u}_{i,j}\\
            =&\iint_{\widetilde{I}_{i,j}(t)}u^{\text{WENO}}(x,y)dxdy.
        \end{split}
    \end{equation}
\end{proof}

\begin{proposition}(Accuracy of the remapping method) Assuming $\Delta x \sim \Delta y$, we have the following estimate for $\widetilde{u}^{\text{WENO}}$:
\begin{equation}
    \widetilde{u}^{\text{WENO}}(x,y)=u(x,y,t)+O(\Delta x^3).
\end{equation}
\begin{proof}
    For $(\widetilde{x},\widetilde{y})\in \widetilde{I}_{i,j}(t)$,
    \begin{equation}
    \begin{split}
        &\widetilde{u}^{\text{WENO}}(\widetilde{x},\widetilde{y})-u(\widetilde{x},\widetilde{y},t)\\
        =&\left(u^{\text{WENO}}_{(\widetilde{p},\widetilde{q})}(\widetilde{x},\widetilde{y})|_{\widetilde{I}_{i,j}(t)}-\frac{1}{|\widetilde{I}_{i,j}(t)|}\widetilde{u}^{(\widetilde{p},\widetilde{q})}_{i,j}+\frac{1}{|\widetilde{I}_{i,j}(t)|}\widetilde{u}_{i,j}\right)-u(\widetilde{x},\widetilde{y},t)\\
        =&\left(u^{\text{WENO}}_{(\widetilde{p},\widetilde{q})}(\widetilde{x},\widetilde{y})|_{\widetilde{I}_{i,j}(t)}-u(\widetilde{x},\widetilde{y},t)\right)-\frac{1}{|\widetilde{I}_{i,j}(t)|}\left(\widetilde{u}^{(\widetilde{p},\widetilde{q})}_{i,j}-\widetilde{u}_{i,j}\right)
    \end{split}
    \end{equation}
    For $u^{\text{WENO}}_{(\widetilde{p},\widetilde{q})}(\widetilde{x},\widetilde{y})|_{\widetilde{I}_{i,j}(t)}-u(\widetilde{x},\widetilde{y},t)$, we have
    \begin{equation*}
        \begin{split}
            &u^{\text{WENO}}_{(\widetilde{p},\widetilde{q})}(\widetilde{x},\widetilde{y})|_{\widetilde{I}_{i,j}(t)}-u(\widetilde{x},\widetilde{y},t)\\
            =&\sum\limits_{k=0}^2\left[\left(\widetilde{x}-x_{\widetilde{p}}\right)\frac{\partial}{\partial x}+\left(\widetilde{y}-y_{\widetilde{q}}\right)\frac{\partial}{\partial y}\right]^k\left(u^{\text{WENO}}_{(\widetilde{p},\widetilde{q})}(\cdot,\cdot)-u(\cdot,\cdot,t)\right)|_{(x_{\widetilde{p}},y_{\widetilde{q}})}+O(\Delta x^3).
        \end{split}
    \end{equation*}
    Following the same procedure in \Cref{remark:WENO_ZQ_accuracy}, we can easily prove that
    \begin{equation}\label{eq:estimate_remapping_1}
        \|D^{\alpha}\left(u^{\text{WENO}}_{(\widetilde{p},\widetilde{q})}(\cdot,\cdot)-u(\cdot,\cdot,t)\right)\|_{\infty}=O(\Delta x^{3-\alpha})
    \end{equation}
    Combining \eqref{eq:estimate_remapping_1} with the fact that $\widetilde{x}-x_{\widetilde{p}}=O(\Delta t)$, $\widetilde{y}-y_{\widetilde{q}}=O(\Delta t)$, and $\Delta t \sim \Delta x \sim \Delta y$, we have
    \begin{equation}\label{{eq:estimate_remapping_2}}
        u^{\text{WENO}}_{(\widetilde{p},\widetilde{q})}(\widetilde{x},\widetilde{y})|_{\widetilde{I}_{i,j}(t)}-u(\widetilde{x},\widetilde{y},t) = O(\Delta x^3).
    \end{equation}
    For $\frac{1}{|\widetilde{I}_{i,j}(t)|}\left(\widetilde{u}^{(\widetilde{p},\widetilde{q})}_{i,j}-\widetilde{u}_{i,j}\right)$,
    \begin{equation}
        \begin{split}
            &\frac{1}{|\widetilde{I}_{i,j}(t)|}\left(\widetilde{u}^{(\widetilde{p},\widetilde{q})}_{i,j}-\widetilde{u}_{i,j}\right)\\
            =&\frac{1}{|\widetilde{I}_{i,j}(t)|}\left(\iint_{\widetilde{I}_{i,j}(t)}u^{\text{WENO}}_{(\widetilde{p},\widetilde{q})}(x,y)dxdy-\iint_{\widetilde{I}_{i,j}(t)}u^{\text{WENO}}(x,y)dxdy\right)\\
            =&\frac{1}{|\widetilde{I}_{i,j}(t)|}\left(\iint_{\widetilde{I}_{i,j}(t)}\left(u^{\text{WENO}}_{(\widetilde{p},\widetilde{q})}(x,y)-u^{\text{WENO}}(x,y)\right)dxdy\right)\\
        \end{split}
    \end{equation}
    Since \eqref{{eq:estimate_remapping_2}} is true for all $(\widetilde{x},\widetilde{y})\in \widetilde{I}_{i,j}(t)$ and
    \begin{equation}
        u^{\text{WENO}}(x,y) - u(x,y,t) = O(\Delta x^3)\quad \text{for all } (x,y)\in\Omega
    \end{equation}
    from \Cref{remark:WENO_ZQ_accuracy}, we immediately have
    \begin{equation}
        \frac{1}{|\widetilde{I}_{i,j}(t)|}\left(\widetilde{u}^{(\widetilde{p},\widetilde{q})}_{i,j}-\widetilde{u}_{i,j}\right) = O(\Delta x^3).
    \end{equation}
\end{proof}
	
\end{proposition}

\subsubsection{Approximation of the diffusion term}\label{sec:diff_source}

The strategy of constructing $\widetilde{\mathcal{G}}_{i,j}(\overline{\mathbf{U}};t)$ contains three steps. First, based on the finite volume information of $u$, we recover high-order finite volume information of $\Delta u$  (assuming uniform mesh) by introducing a differential matrix. Second, we reconstruct a high-order piecewise polynomial to approximate $\Delta u$. Finally, we evaluate the corresponding integral over $\widetilde{I}_{i,j}(t)$ and obtain $\widetilde{\mathcal{G}}_{i,j}(\overline{\mathbf{U}};t)$. We emphasize here that the implicit treatment of the diffusion term is performed on the Eulerian mesh. This means that the differential matrix introduced in the following \textbf{Step 1} is necessary, as will be shown in \Cref{sec:IMEX}. The details of these three steps are summarized as follows:

\begin{enumerate}
    \item[\textbf{Step 1:}] \textit{Recover high-order finite volume information of $\Delta u$.}

    We recover
    \begin{equation}
        \Delta \overline{\mathbf{U}} := \mathbf{D}\overline{\mathbf{U}},
    \end{equation}
    where $\mathbf{D}$ is the differential matrix assembled by the following local operators
    \begin{equation*}
        \overline{\Delta u}_{i,j} := \left[-\frac{1}{12}~~\frac{4}{3}~~-\frac{5}{2}~~\frac{4}{3}~~ -\frac{1}{12}\right]\left(\frac{1}{\Delta x^2}\left[\begin{array}{c}
            \overline{u}_{i-2,j}\\
            \overline{u}_{i-1,j}\\
            \overline{u}_{i,j}\\
            \overline{u}_{i+1,j}\\
            \overline{u}_{i+2,j}
        \end{array}\right]+\frac{1}{\Delta y^2}\left[\begin{array}{c}
            \overline{u}_{i,j-2}\\
            \overline{u}_{i,j-1}\\
            \overline{u}_{i,j}\\
            \overline{u}_{i,j+1}\\
            \overline{u}_{i,j+2}
        \end{array}\right]\right),
    \end{equation*}
    for all $i,j$, and corresponding boundary conditions.

    \item[\textbf{Step 2:}] \textit{Recover a piecewise reconstruction polynomial of $\Delta u$.}

    We utilize the same polynomial $q_0(x,y)$ in \Cref{sec:WENO_ZQ_Eulerian} as the reconstruction formula and denote the final piecewise polynomial by $\Delta u^{\text{rec}}(x,y)$.

    \item[\textbf{Step 3:}] \textit{Compute the integral of $\Delta u^{\text{rec}}(x,y)$ as the final approximation.}

    \begin{align}
          &\widetilde{\mathcal{G}}_{i,j}(\overline{\mathbf{U}};t) \nonumber \\
         =&\epsilon\iint_{\widetilde{I}_{i,j}(t)}\Delta udxdy\approx \epsilon\iint_{\widetilde{I}_{i,j}(t)}\Delta u^{\text{rec}}(x,y)dxdy :=\epsilon\widetilde{\Delta u}_{i,j}\left(\overline{\mathbf{U}};t\right)\quad \text{for all}~~i,j,\label{eq:EL_FV_diffussion_approx}
    \end{align}
    where the integral of piesewise polynomial is accomplished by the method of our previous work \cite{zheng_fourth-order_2022}.
\end{enumerate}
%For the source term, we simply apply four-point Gauss quadrature over $\widetilde{I}_{i,j}(t)$ to $g(x,y,t)$. Hence, the final formula for $\widetilde{\mathcal{G}}_{i,j}$ is
%\begin{equation}\label{eq:EL_FV_diffussion_source_approx}
%    \widetilde{\mathcal{G}}_{i,j}(\overline{\mathbf{U}};t) = \epsilon\iint_{\widetilde{I}_{i,j}(t)}\Delta u^{\text{rec}}(x,y)dxdy + \iint_{\widetilde{I}_{i,j}(t)}g(x,y,t)dxdy\quad\text{for all}~~i,j,
%\end{equation}
%where we use the same notation $\iint$ but refer to corresponding numerical integral methods for conciseness. For the global notation $\bm{\widetilde{\mathcal{G}}}$, we denote that:
%\begin{equation}
%\bm{\widetilde{\mathcal{G}}}\left(\overline{\mathbf{U}};t\right)=\epsilon\widetilde{\Delta\mathbf{U}}\left(\overline{\mathbf{U}};t\right)+\widetilde{\mathbf{S}}(t),
%\end{equation}
%where $\epsilon\widetilde{\Delta\mathbf{U}}=:\left(\epsilon\widetilde{\Delta u}_{i,j}\right)_{N_xN_y}$ and $\widetilde{\mathbf{S}}(t)=:\left(\iint_{\widetilde{I}_{i,j}(t)}g(x,y,t)dxdy\right)_{N_xN_y}$.

\subsubsection{High-order IMEX RK temporal discretization}\label{sec:IMEX}

In light of the high-order spatial discretization presented in \eqref{eq:semi_discretization}, \eqref{eq:EL_FV_flux_approximation}, and \eqref{eq:EL_FV_diffussion_approx}, we introduce the IMEX RK temporal discretizations. Under the IMEX setting, we evolve the convection flux term explicitly while evolving the diffusion term implicitly. In addition, the time-step constraint is controlled by the explicit convection part, with $\Delta t\sim \sqrt{\min\{\Delta x, \Delta y\}}$.

An IMEX RK scheme can be represented by the following two butcher tables \cite{ascher_implicit_explicit_1997}:
\[
\begin{array}{cc}
\textbf{Implicit Scheme} & \hspace{1.cm}\textbf{Explicit Scheme} \\
\begin{array}{c|ccccc}
0 & 0 & 0 &  0 & \cdots & 0 \\
c_1 & 0  & a_{11} & 0 & \cdots & 0 \\
c_2 & 0  & a_{21} & a_{22} & \cdots & 0 \\
\vdots & \vdots &  \vdots & \vdots & \ddots & \vdots \\
c_s & 0  & a_{s1} & a_{s2} & \cdots & a_{ss} \\
\hline
& 0 & b_1 & b_2 & \cdots & b_s
\end{array}
& \hspace{1.cm}
\begin{array}{c|ccccc}
0 & 0 & 0 & 0 & \cdots & 0 \\
c_1 & \hat{a}_{21} & 0 &  0 & \cdots & 0 \\
c_2 & \hat{a}_{31} & \hat{a}_{32} & 0 &  \cdots & 0 \\
\vdots & \vdots & \vdots & \vdots & \ddots & \vdots \\
c_{\sigma-1} & \hat{a}_{\sigma,1} & \hat{a}_{\sigma,2} & \hat{a}_{\sigma,3} & \cdots & 0 \\
\hline
& \hat{b}_1 & \hat{b}_2 & \hat{b}_3 & \cdots & \hat{b}_{\sigma}
\end{array}
\end{array}
\]
A triplet $(s,\sigma,p)$ is used to demonstrate that the IMEX scheme uses an s-stage implicit scheme and a $\sigma$-stage explicit scheme achieving $p$th-order accuracy. %Here, $\sigma = s+1$ unless $c_s=1$ and $\hat{a}_{s+1,j} = \hat{b}_j$ for all $j=1,2,\ldots,s+1$, in which case $\sigma = s$.

The butcher table of first-order IMEX scheme in \eqref{eq:scheme_IMEX111}, which is also called IMEX(1,1,1) in \cite{ascher_implicit_explicit_1997}, is
\[
\begin{array}{cc}
\textbf{Implicit Scheme} & \hspace{1cm}\textbf{Explicit Scheme} \\
\begin{array}{c|cc}
0 & 0 & 0 \\
1 & 0 & 1 \\
\hline
& 0 & 1
\end{array}
& \hspace{1cm}
\begin{array}{c|cc}
0 & 0 & 0 \\
1 & 1 & 0 \\
\hline
& 1 & 0
\end{array}
\end{array}
\]
The second-order IMEX(1,2,2) scheme in \cite{ascher_implicit_explicit_1997} is represented by the following butcher tables:
\[
\begin{array}{cc}
\textbf{Implicit Scheme} & \hspace{1cm}\textbf{Explicit Scheme} \\
\begin{array}{c|cc}
0 & 0 & 0 \\
\frac{1}{2} & 0 & \frac{1}{2} \\
\hline
& 0 & 1
\end{array}
& \hspace{1cm}
\begin{array}{c|cc}
0 & 0 & 0 \\
\frac{1}{2} & \frac{1}{2} & 0 \\
\hline
& 0 & 1
\end{array}
\end{array}
\]

\begin{figure}[!htbp]
    \centering
    \subfloat[]{
        \begin{overpic}[scale=0.4]{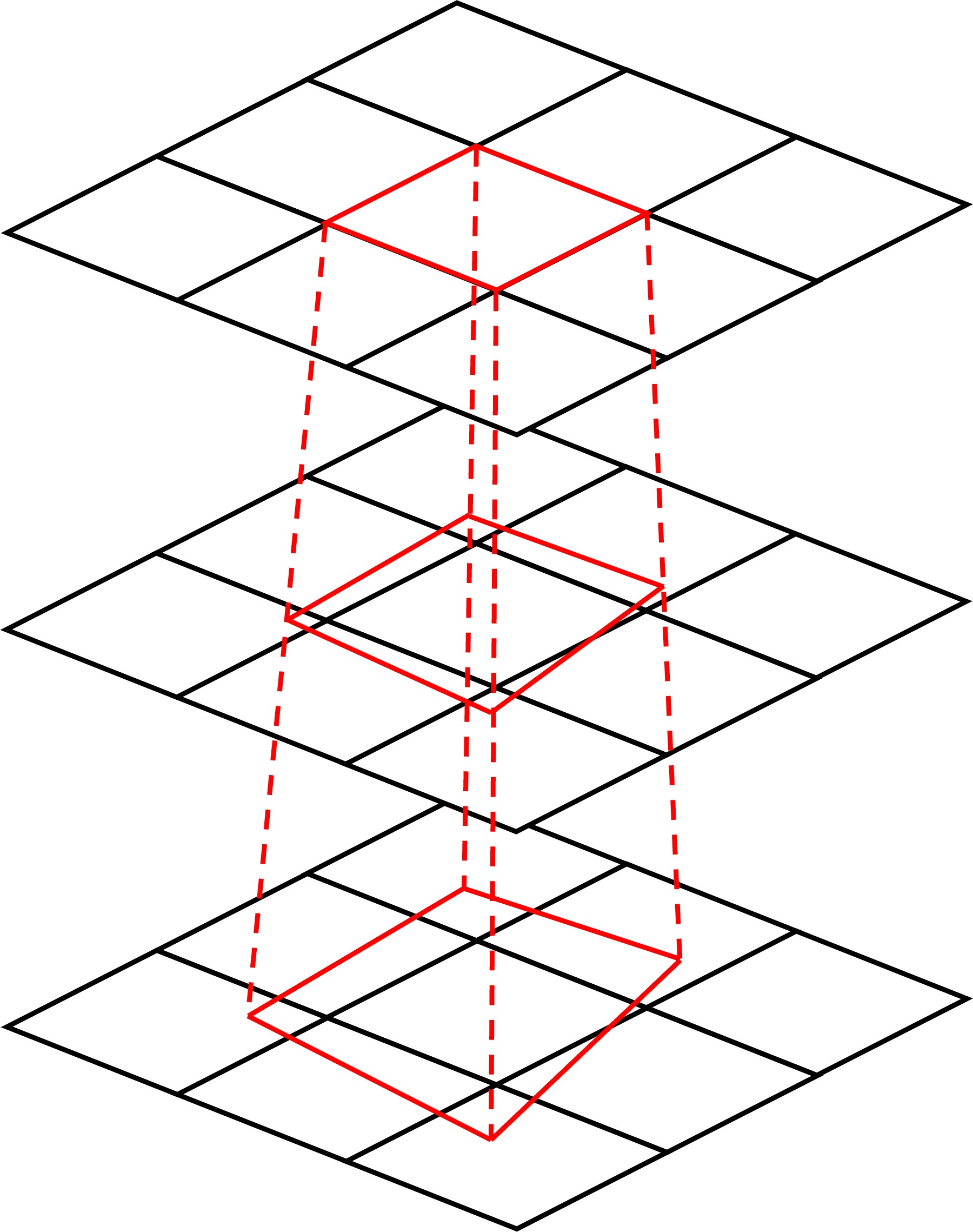}
            \put(33,81){$I_{i,j}$}
            \put(29,18){$\widetilde{I}_{i,j}(t^n)$}
            \put(28,50){$\widetilde{I}_{i,j}(t^{n+\frac{1}{2}})$}
            \put(84,17) {$t^n$}
            \put(84,49.5) {$t^{n+\frac{1}{2}}$}
            \put(84,82){$t^{n+1}$}
        \end{overpic}\label{fig:schematic_IMEX122_a}
    }
    \hspace{1cm}
    \subfloat[]{
        \begin{overpic}[scale=0.4]{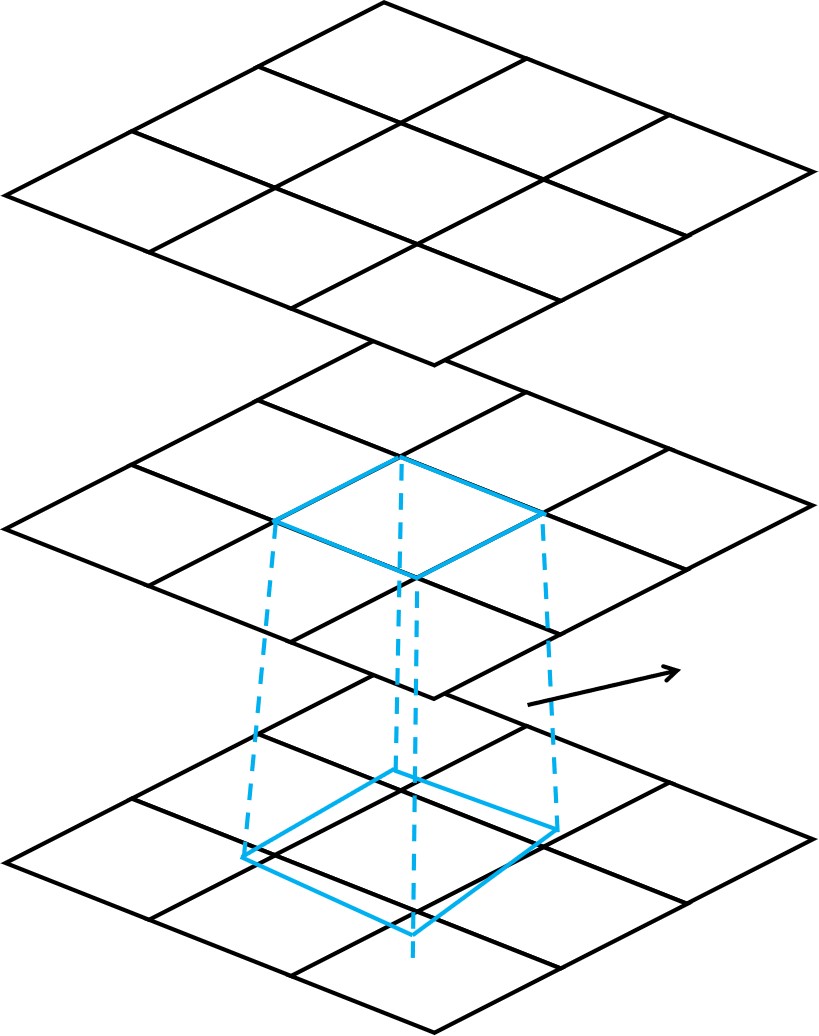}
            \put(33,81){$I_{i,j}$}
            \put(33,49){$I_{i,j}$}
            \put(29,18){$\widetilde{I}_{i,j}(t^{n+\frac{1}{2}})$}
            \put(84,17) {$t^n$}
            \put(84,49.5) {$t^{n+\frac{1}{2}}$}
            \put(84,82){$t^{n+1}$}
            \put(68,34){$\widetilde{\widetilde{\Omega}}_{i,j}$}
        \end{overpic}\label{fig:schematic_IMEX122_b}
    }
    \caption{Schematic illustration of full discretization with IMEX(1,2,2).}\label{fig:schematic_IMEX122}
\end{figure}
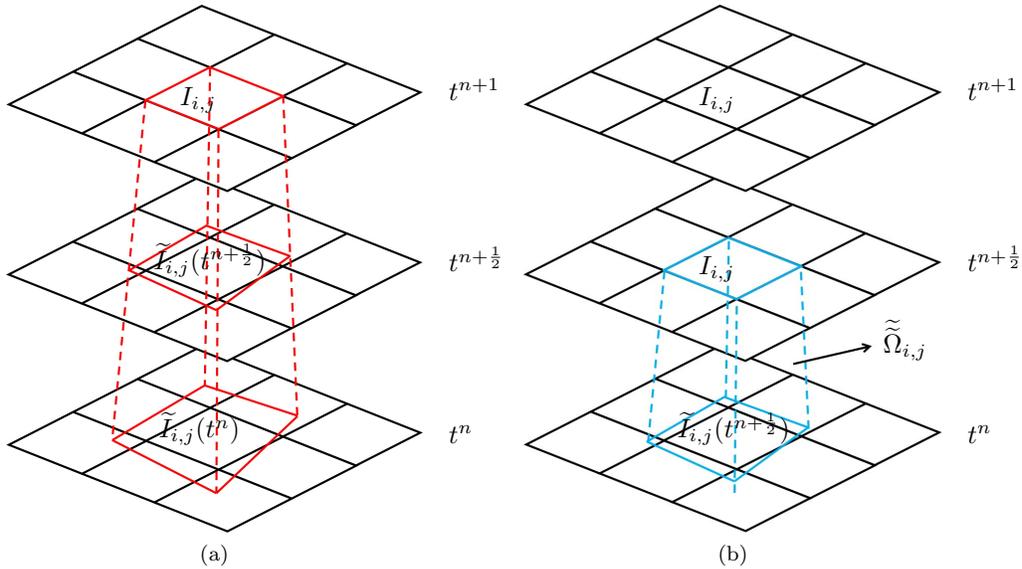

One noteworthy complexity of IMEX(1,2,2), in contrast to IMEX(1,1,1), is its involvement of the intermediate time level $t^{n+\frac{1}{2}}$ (see \Cref{fig:schematic_IMEX122_a}). To construct a full discretization based on IMEX(1,2,2), we define a new characteristic region $\widetilde{\widetilde{\Omega}}_{i,j}:=\{(x,y,t)|(x,y)\in \widetilde{I}_{i,j}(t+\frac{1}{2}\Delta t),~~t\in[t^n,t^{n+\frac{1}{2}}]\}$ for each mesh index $(i,j)$ (see \Cref{fig:schematic_IMEX122_b}), following a similar strategy as proposed in \cite{nakao2022eulerian}. For $\{\widetilde{\widetilde{\Omega}}_{i,j}\}$, we define an operator $\widetilde{\widetilde{\mathcal{W}}}$ such that $\widetilde{\widetilde{\mathcal{W}}}(\overline{\mathbf{U}},t^{n+\alpha})$ represents the corresponding WENO piecewise polynomial with respect to $\{\widetilde{I}_{i,j}(t^{n+\frac{1}{2}+\alpha})\}$, which is constructed based on $\overline{\mathbf{U}}$. Furthermore, we define that $\widetilde{\widetilde{\bm{\mathcal{F}}}}(\overline{\mathbf{U}},t^{n+\alpha}) :=\left(\widetilde{\widetilde{\mathcal{F}}}_{i,j}(\overline{\mathbf{U}},t^{n+\alpha}) \right)_{N_xN_y}$ with
\begin{equation}
\widetilde{\widetilde{\mathcal{F}}}_{i,j}(\overline{\mathbf{U}},t^{n+\alpha}) := -\int_{\partial \widetilde{I}_{i,j}(t^{n+\frac{1}{2}+\alpha})}\hat{F}(\widetilde{\widetilde{W}}(\overline{\mathbf{U}},t^{n+\alpha}),x,y,t^{n+\alpha})ds,\quad\text{for all}~~i,j.
\end{equation}
%and that $\widetilde{\widetilde{\mathbf{S}}}(t)=:\left(\widetilde{\widetilde{S}}_{i,j}(t)\right)_{N_xN_y}$ with
%\begin{equation}
%\widetilde{\widetilde{S}}_{i,j}(t):=\iint_{\widetilde{I}_{i,j}(t+\frac{1}{2}\Delta t)}g(x,y,t)dxdy\quad\text{for all}~~i,j.
%\end{equation}
%Similar to the definition of $\mathcal{W}$, we define that $\mathcal{W}^{\text{rec}}\left(\mathbf{D}\overline{\mathbf{U}}\right)$ represents the non-WENO piecewise polynomial with respect to the Eulerian mesh, which is constructed based on $\mathbf{D}\overline{\mathbf{U}}$
Then, the second-order fully discrete scheme is provided as follows (see \Cref{fig:schematic_IMEX122}):
\begin{align}
    \mathbf{M}\overline{\mathbf{U}}^{(1)} &= \widetilde{\widetilde{\mathbf{U}}}^n+\frac{1}{2}\Delta t\widetilde{\widetilde{\bm{\mathcal{F}}}}(\overline{\mathbf{U}}^n,t^n) + \frac{1}{2}\Delta t\left(\epsilon\mathbf{M}\mathbf{D}\overline{\mathbf{U}}^{(1)}\right),\label{eq:scheme_IMEX122_1}\\
    \mathbf{M}\overline{\mathbf{U}}^{n+1} &= \widetilde{\mathbf{U}}^n+\Delta t\widetilde{\bm{\mathcal{F}}}(\overline{\mathbf{U}}^{(1)},t^{n+\frac{1}{2}}) + \Delta t\widetilde{\bm{\mathbf{\mathcal{G}}}}(\overline{\mathbf{U}}^{(1)},t^{n+\frac{1}{2}}),\label{eq:scheme_IMEX122_2}
\end{align}
where \begin{itemize}
    \item the notation $\widetilde{\widetilde{\cdot}}$ is used to indicate that the corresponding values or operations refer to specific slices of the characteristic region $\widetilde{\widetilde{\Omega}}_{i,j}$,
    \item $\widetilde{\widetilde{\mathbf{U}}}^n=:\left(\iint_{\widetilde{I}_{i,j}(t^{n+\frac{1}{2}})}\mathcal{W}\left(\overline{\mathbf{U}}^n\right)dxdy\right)_{N_xN_y}$ (see \Cref{fig:schematic_IMEX122_b}).
\end{itemize}

The design of the additional characteristic region $\widetilde{\widetilde{\Omega}}_{i,j}$ helps us avoid the expansive reconstruction procedure associated with nonuniform Lagrangian meshes at the intermediate time level, which dynamically change as time evolves. For the numerical tests, we employ IMEX(2,3,3) as described in \cite{ascher_implicit_explicit_1997} to construct a third-order fully discretized EL-RK-FV-WENO-IMEX scheme. We omit the details for IMEX(2,3,3) since the key techniques have already been covered in the IMEX(1,2,2) case.

%\begin{rem}
%    The design of the extra characteristic region $\widetilde{\widetilde{\Omega}}_{i,j}$ along with the polynomial remapping method helps us escape the expansive reconstruction procedure based on non-uniform Lagrangian meshes.
%\end{rem}

\subsection{EL-RK-FV-WENO scheme for the nonlinear convection-diffusion equation}\label{sec:nonlinear}
Consider
\begin{equation}\label{eq:convec_diffu_reac_nonlinear}
    u_t+(f_1(u))_x+(f_2(u))_y=\epsilon(u_{xx}+u_{yy}).
\end{equation}
To extend the proposed scheme for \eqref{eq:convec_diffu_reac_nonlinear}, we further design  two modifications with very limited extra cost. The two modifications utilize the EL-RK-FV framework designed in the previous sections and no longer require the adoption of one temporal method within another as in \cite{cai_eulerian-lagrangian_2021}. We summarize the details of these two modifications as follows:
\begin{enumerate}
    \item[\textbf{1.}] \textit{Redesign the modified velocity field.}

    The construction of the original modified velocity field requires the exact velocity field at $t = t^{n+1}$, which is unknown for nonlinear models. We redesign that $(\alpha(x,y,t),\beta(x,y,t))$ is defined by first applying interpolation at $t=t^{n+1}$ such that
	\begin{equation}\label{eq:alpha_beta_tnp1_2}
		\begin{split}
			\alpha(x_{i\pm\frac12},y_{j\pm\frac12},t^{n+1})=f_1^\prime(\mathcal{W}(\overline{\mathbf{U}}^n))|_{(x_{i\pm\frac12},y_{j\pm\frac12})},\\
			\beta(x_{i\pm\frac12},y_{j\pm\frac12},t^{n+1})=f_2^\prime(\mathcal{W}(\overline{\mathbf{U}}^n))|_{(x_{i\pm\frac12},y_{j\pm\frac12})},
		\end{split}
	\end{equation}
and then following the same procedure as in \eqref{eq:dynamic_region}-\eqref{eq:modified_velocity}. This modified velocity field can still give us the same time-step constraint $\Delta t \sim \sqrt{\min\{\Delta x, \Delta y\}}$ if we apply similar analysis as in \Cref{rem:time_step_constraint}. This flexibility of defining the modified velocity field is of vital importance. In our previous semi-Lagrangian work for convection-diffusion equations \cite{ding_semi-lagrangian_2020}, the necessity of an exact velocity field presents challenges in adapting the approach for nonlinear models.

    \item[\textbf{2.}] \textit{Recover the velocity fields at intermediate time levels.}

    The velocity field for $t\in(t^n,t^{n+1}]$, which is required for evaluating numerical fluxs, is unknown for nonlinear models. For this issue, we simply use the predicted solutions at the intermediate time levels to recover corresponding velocity fields. For example, in \eqref{eq:scheme_IMEX122_2}, the exact velocity field $(a(x,y,t^{n+\frac{1}{2}}),b(x,y,t^{n+\frac{1}{2}}))$ is replaced with $\left(f_1^\prime(\mathcal{W}(\overline{\mathbf{U}}^{(1)})),f_2^\prime(\mathcal{W}(\overline{\mathbf{U}}^{(1)}))\right)$ for $\widetilde{\bm{\mathcal{F}}}\left(\overline{\mathbf{U}}^{(1)},t^{n+\frac{1}{2}}\right)$.

    %\item[\textbf{3.}] \textit{Recover the source term.}

    %We use $\widetilde{\mathbf{S}}(t)$ as an example to illustrate the idea. For $t=t^{n+1}$, We require that
    %\begin{equation}
    %    \widetilde{\mathbf{S}}(t^{n+1}) = \mathbf{S}\overline{\mathbf{U}}^{n+1} + \mathbf{b},
    %\end{equation}
    %Where $\mathbf{S}$ is a matrix and $\mathbf{b}$ is a vector. In this case, The linear system of \Cref{eq:scheme_IMEX111} becomes:
    %\begin{equation}
     %   \left(\mathbf{M}-\epsilon\Delta t\mathbf{D}-\Delta t\mathbf{S}\right)\overline{\mathbf{U}}^{n+1} = \Delta t\bm{\widetilde{\mathcal{F}}}(\overline{\mathbf{U}}^n,t^n) + \Delta t\mathbf{b}.
    %\end{equation}
    %For $t=t^n$, we simply apply that
    %\begin{equation}
     %   \widetilde{S}_{i,j}(t^{n}) = \iint_{\widetilde{I}_{i,j}(t^n)}g(\mathcal{W}(\overline{\mathbf{U}}^n),x,y,t)dxdy\quad \text{for all}~~i,j.
    %\end{equation}
\end{enumerate}

\section{Numerical tests}\label{sec:numerical_tests}
In this section, we apply the proposed EL-RK-FV-WENO scheme to four challenging problems. The first two problems involve linear equations: the swirling deformation flow, characterized by pure convection terms, and the 0D2V Leonard-Bernstein linearized Fokker-Planck equation. The latter two are nonlinear models: the Kelvin-Helmholtz instability problem, again with pure convection terms, and the incompressible Navier-Stokes equations. We use these four cases to demonstrate the effectiveness and the designed properties of the proposed scheme. The time-steps in the following are defined by:
\begin{equation}\label{eq:time_steps_linear}
    \Delta t = \frac{\text{CFL}}{\frac{\max\{|f^\prime(u)|\}}{\Delta x}+\frac{\max\{|g^\prime(u)|\}}{\Delta y}},
\end{equation}
where $\left(f^\prime(u),g^\prime(u)\right)$ represents the corresponding velocity field. For pure convection simulation, we use the third-order Runge-Kutta temporal discretization with the following butcher table:
\[
\begin{array}{c|ccc}
0 & 0 & 0 & 0 \\
\frac{1}{2} & \frac{1}{2} & 0 & 0 \\
1 & -1 & 2 & 0\\
\hline
& \frac{1}{6} & \frac{2}{3} & \frac{1}{6}
\end{array}
\]
For convection-diffusion simulation, we apply IMEX(2,3,3) in \cite{ascher_implicit_explicit_1997} with the following butcher tables:
\[
\begin{array}{cc}
\textbf{Implicit Scheme} & \hspace{1.5cm}\textbf{Explicit Scheme} \\
\begin{array}{c|ccc}
0 & 0 & 0 & 0 \\
\gamma & 0 & \gamma & 0 \\
1-\gamma & 0 & 1-2\gamma & \gamma\\
\hline
& 0 & \frac{1}{2} & \frac{1}{2}
\end{array}
& \hspace{1.5cm}
\begin{array}{c|ccc}
0 & 0 & 0 & 0 \\
\gamma & \gamma & 0 & 0 \\
1-\gamma & \gamma-1 & 2\left(1-\gamma\right) & 0\\
\hline
& 0 & \frac{1}{2} & \frac{1}{2}
\end{array}
\end{array}
\]
where $\gamma = (3+\sqrt{3})/6$.

\subsection{Linear models}

\begin{example}(Swirling deformation flow). Consider  the following equation:
	\begin{equation}\label{2_D_SDF}
		\begin{split}
			u_t-(2\pi\text{cos}^2(\frac x2)\text{sin}(y)g(t)u)_x + (2\pi\text{sin}(x)\text{cos}^2(\frac y2)g(t)u)_y=0,~~x,~y\in[-\pi,\pi],
		\end{split}
	\end{equation}
	where $g(t) = \text{cos}(\pi t/T)$ with $T = 1.5$. We first consider \eqref{2_D_SDF} with the following smooth initial condition:
	\begin{equation}\label{eq:SDF_smooth_initial_condtion}
		u(x,y,0) = \begin{cases}
			r^b_0\text{cos}(\frac{r^b(\mathbf{x})\pi}{2r^b_0})^6~~&\text{if}~ r^b(\mathbf{x})<r^b_0,\\
			0,~~&\text{otherwise},
		\end{cases}
	\end{equation}
	where $r^b_0=0.3\pi$, $r^b(\mathbf{x})=\sqrt{ (x-x_0^b)^2+(y-y_0^b)^2 }$ and the center of the cosine bell $(x_0^b,y_0^b) = (0.3\pi,0)$. Table \ref{tab_2_D_SDF} shows the $L^1$, $L^2$, and $L^{\infty}$ errors and corresponding orders of accuracy for the proposed scheme. As indicated, the expected 3rd-order spatial accuracy is achieved through mesh refinement.
	
\begin{table}[!htbp]
	\centering
	\caption{(Swirling deformation flow)  $L^1$, $L^2$, and $L^{\infty}$ errors and corresponding orders of accuracy of the EL-RK-FV-WENO scheme for \eqref{2_D_SDF} with initial condition \eqref{eq:SDF_smooth_initial_condtion} at $t = 1.5$ with CFL = 1.}\label{tab_2_D_SDF}
	\centering
	\begin{tabular}{|c|cc|cc|cc|}
            \hline
            mesh&$L^1$ error&order&$L^2$ error&order&$L^{\infty}$ error&order\\
            \hline
            20$\times$  20&    8.37E-03&   ---&    4.51E-02&   ---&    7.40E-01&   ---\\
            40$\times$  40&    3.85E-03&   1.12&    2.52E-02&   0.84&    4.74E-01&   0.64\\
            80$\times$  80&    1.16E-03&   1.72&    8.06E-03&   1.64&    1.66E-01&   1.52\\
            160$\times$ 160&    2.22E-04&   2.39&    1.50E-03&   2.43&    3.24E-02&   2.35\\
            320$\times$ 320&    3.01E-05&   2.88&    2.03E-04&   2.88&    4.86E-03&   2.74\\	
             \hline
	\end{tabular}
\end{table}

In Figure \ref{fig:SDF_CFL_vs_L2error}, by varying the CFL number while fixing the spatial mesh, we investigate the temporal order of accuracy. For the results using mesh $160\times160$, the proposed scheme demonstrates 3rd-order temporal accuracy and is stable when CFL is less than 21. For the mesh $320\times320$, stability is observed at least up to CFL = 30, corroborating our time-step constraint estimate \eqref{eq:time_constraint_es}.

\begin{figure}[!htbp]
	\centering
	\includegraphics[width=0.45\textwidth]{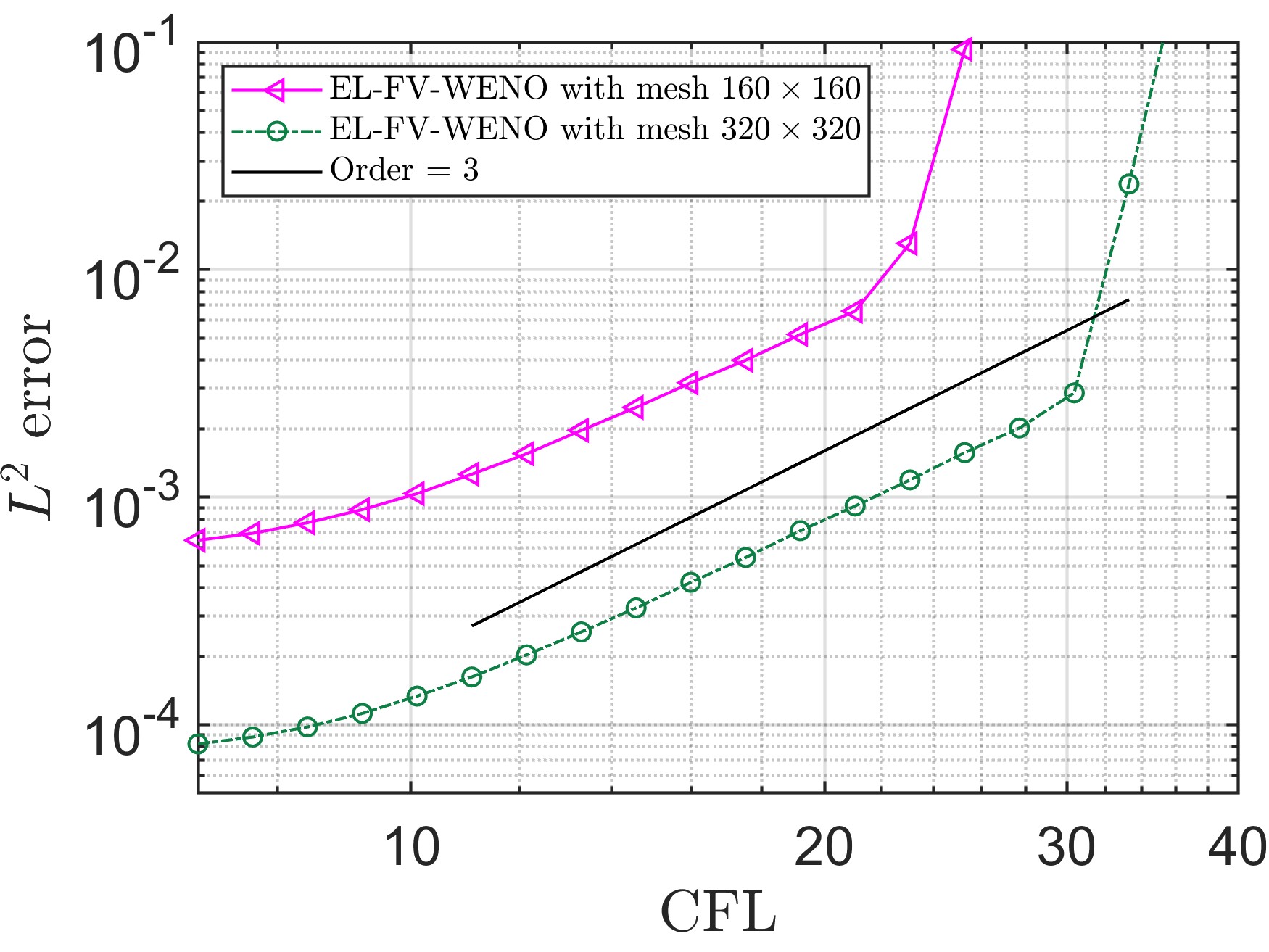}
	\caption{(Swirling deformation flow) Log-log plot of CFL numbers versus $L^2$ errors with fixed meshes $160\times160$ and $320\times320$ at $t = 1.5$ of the EL-RK-FV-WENO scheme.}\label{fig:SDF_CFL_vs_L2error}
\end{figure}	

To validate the non-oscillatory nature of the proposed WENO reconstruction, we consider a discontinuous initial condition featuring a cylinder with a notch, a cone, and a smooth bell as shown in Figure \ref{fig:SDF_initial}. Figures \ref{fig:SDF_0_75} and \ref{fig:SDF_1_5} display the mesh plots and contour plots of the numerical solutions at $t = 0.75$ and $t=1.5$. The swirling deformation flow significantly deforms the solution at half the period ($t=0.75$) and then reforms it back to its initial state at $t=1.5$. The WENO reconstruction method effectively controls numerical oscillations as demonstrated.

\begin{figure}[!htbp]
	\centering
	\subfloat[Mesh plot]{
		\includegraphics[width=0.45\textwidth]{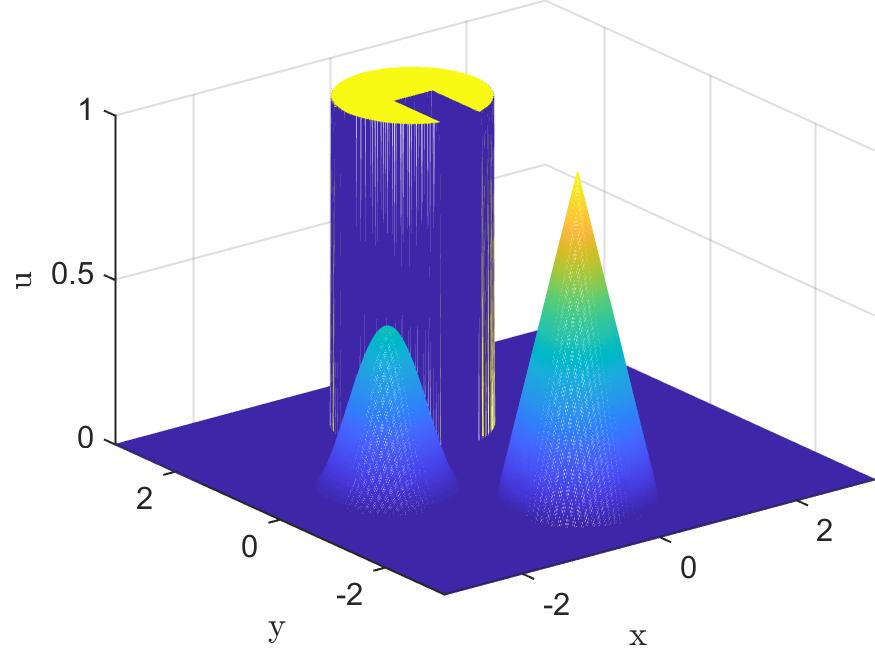}
	}
 	\subfloat[Contour plot]{
		\includegraphics[width=0.45\textwidth]{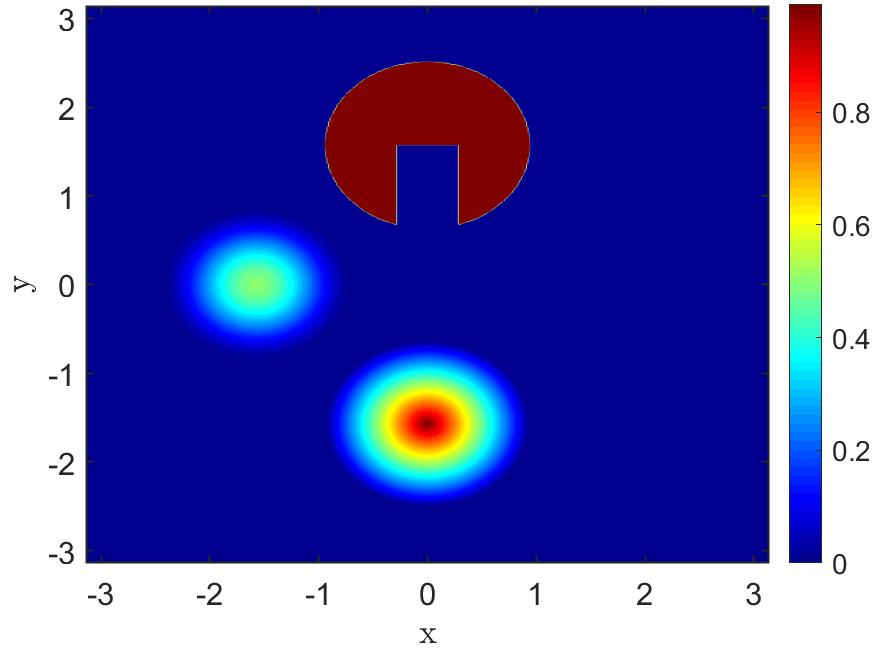}
	}
	\caption{(Swirling deformation flow) Mesh plot and contour plot of a discontinuous initial condition.}\label{fig:SDF_initial}
\end{figure}	

\begin{figure}[!htbp]
	\centering
	\subfloat[Mesh plot]{
		\includegraphics[width=0.45\textwidth]{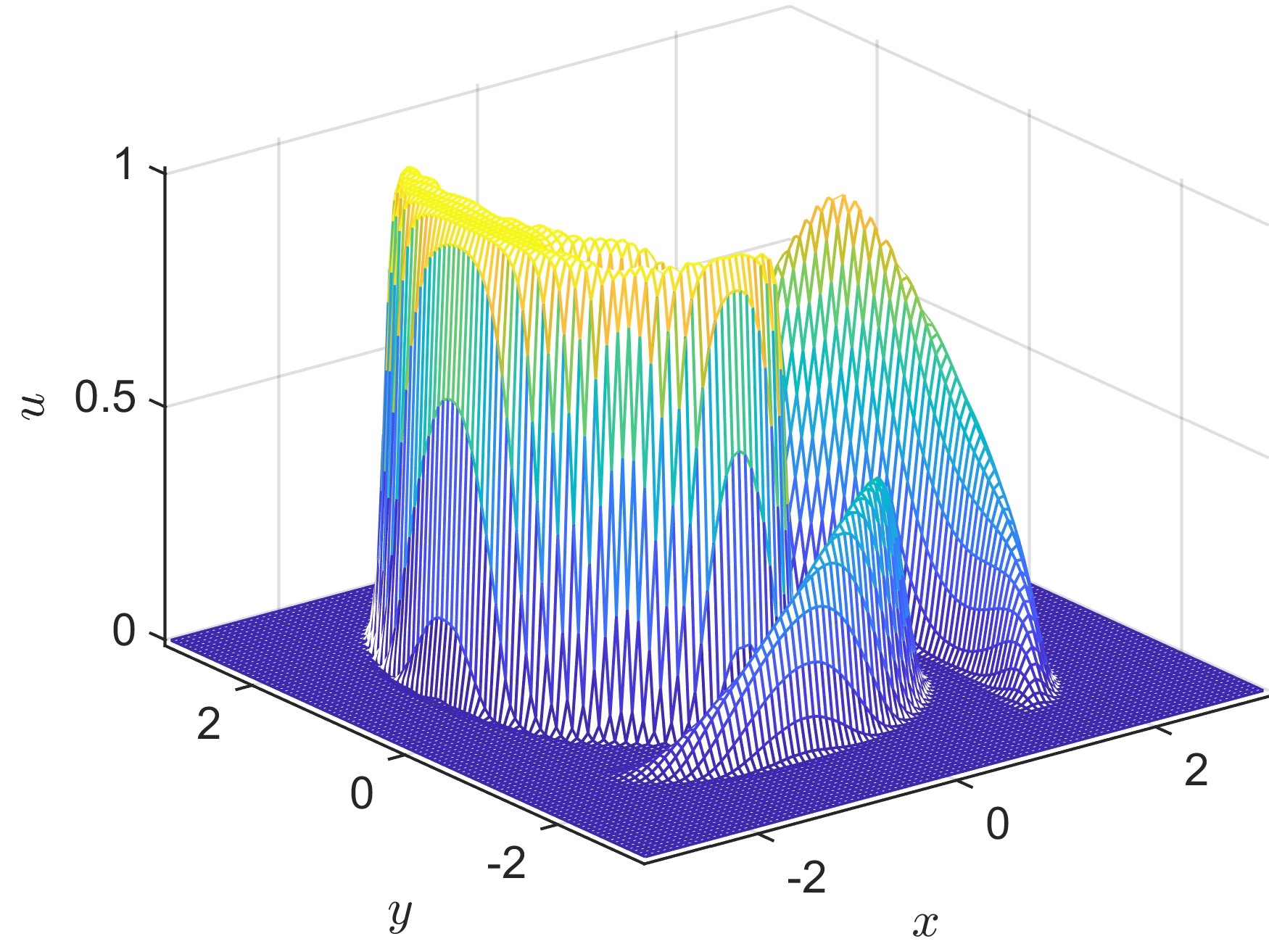}
	}
 	\subfloat[Contour plot]{
		\includegraphics[width=0.45\textwidth]{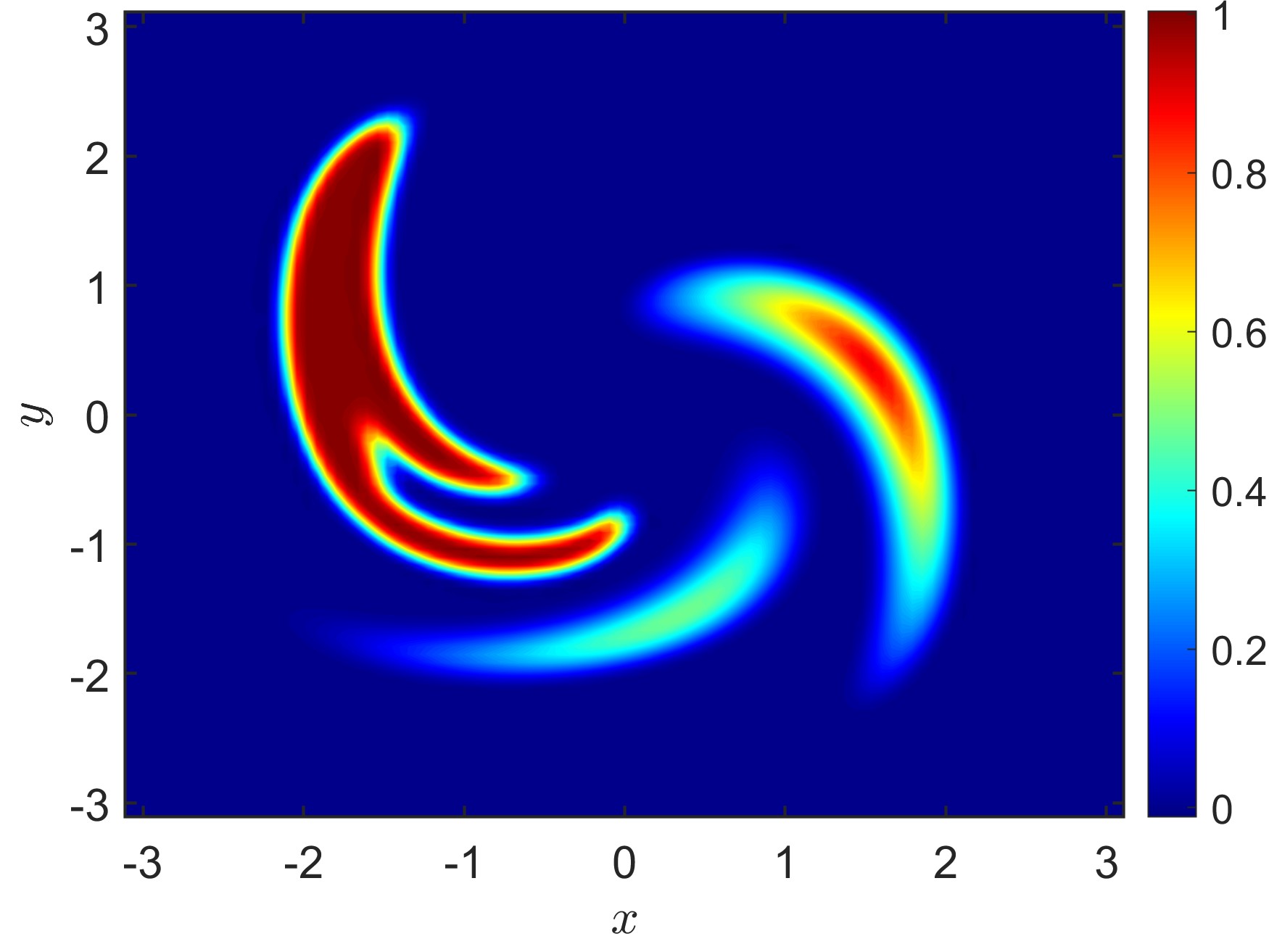}
	}
	\caption{(Swirling deformation flow) Mesh plot and contour plot of the numerical solution of the EL-RK-FV-WENO scheme with CFL = 10.2 and mesh size $100\times100$ at $t=0.75$.}\label{fig:SDF_0_75}
\end{figure}	

\begin{figure}[!htbp]
	\centering
	\subfloat[Mesh plot]{
		\includegraphics[width=0.45\textwidth]{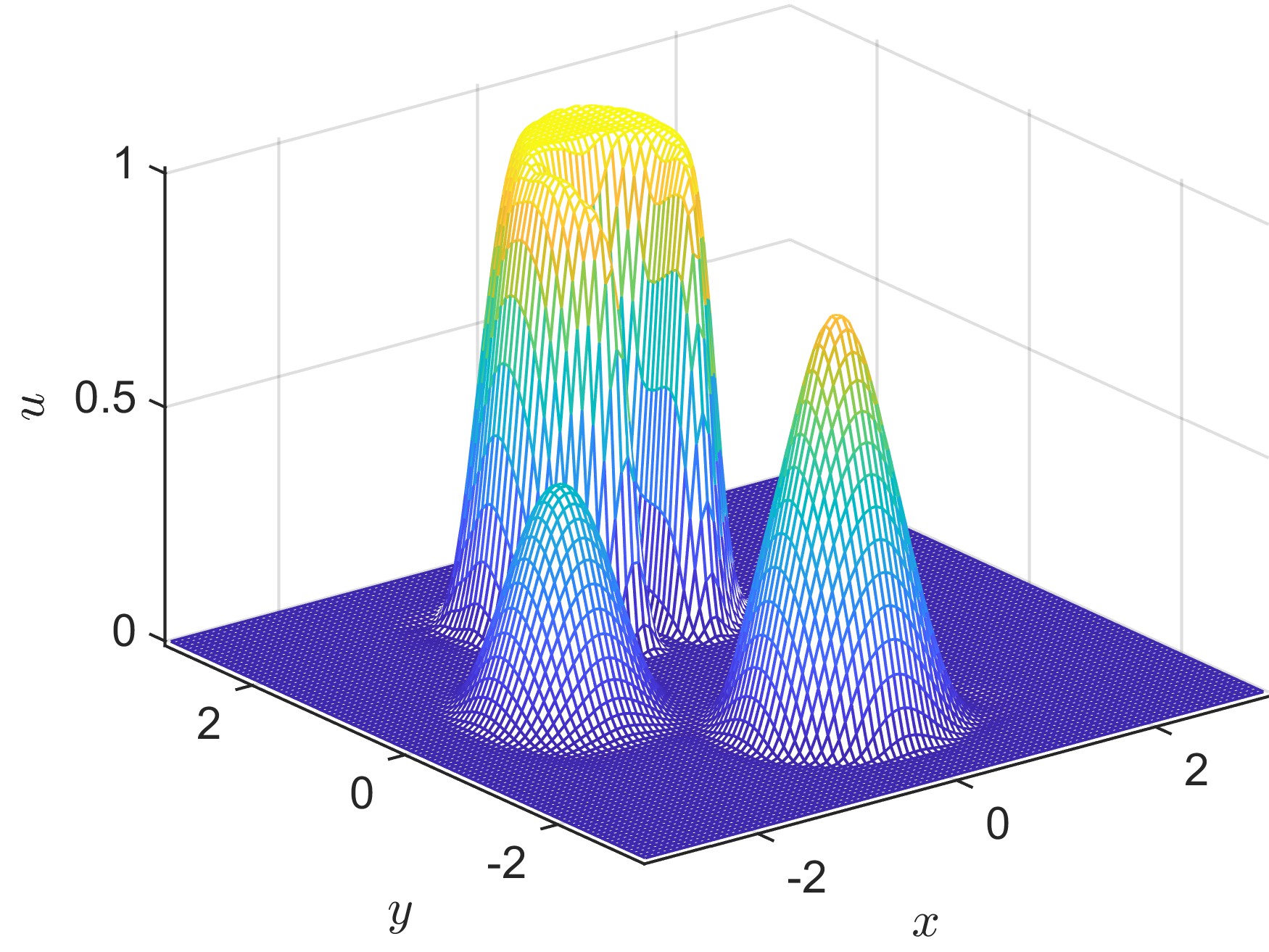}
	}
 	\subfloat[Contour plot]{
		\includegraphics[width=0.45\textwidth]{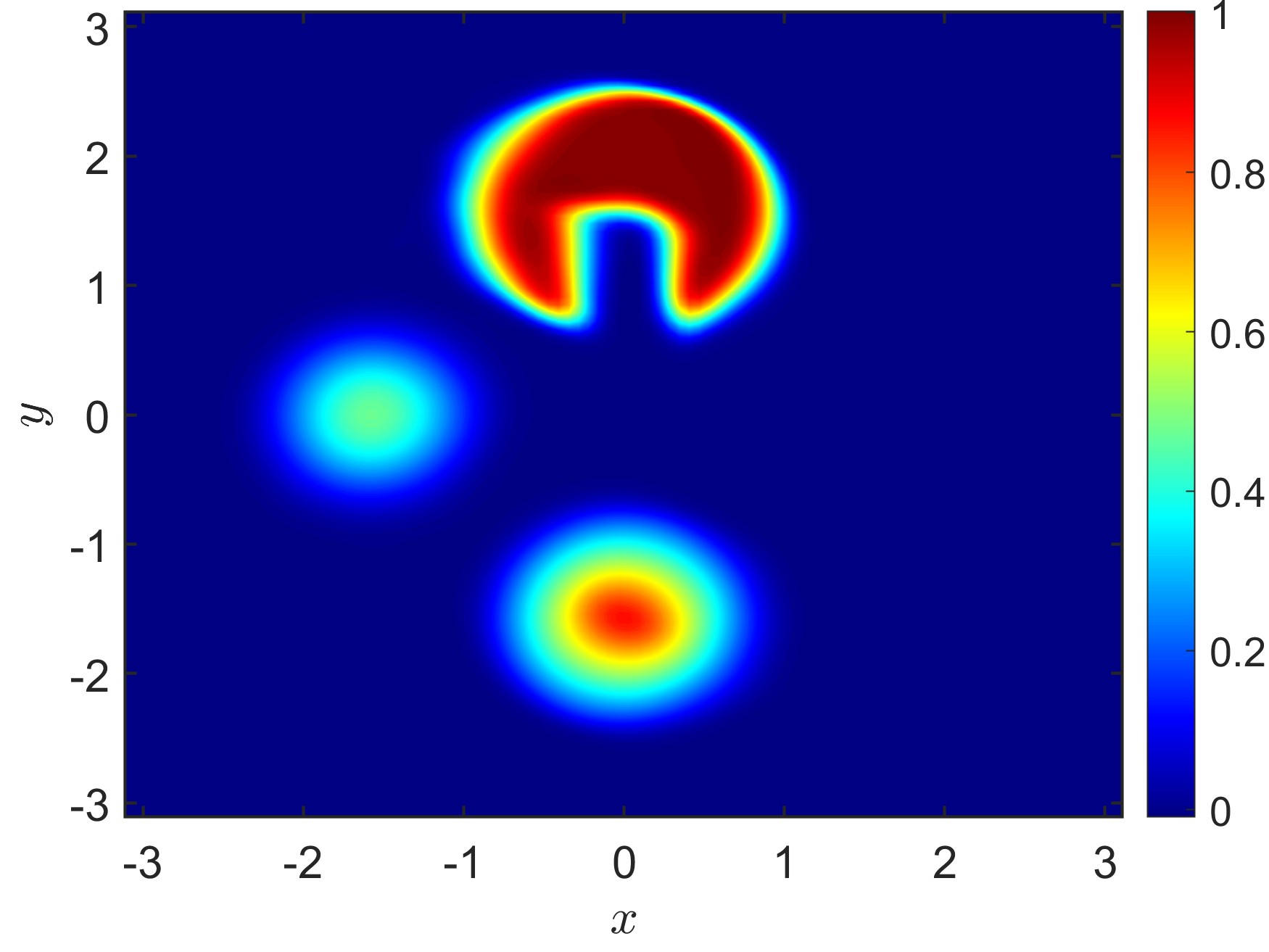}
	}
	\caption{(Swirling deformation flow) Mesh plot and contour plot of the numerical solution of the EL-RK-FV-WENO scheme with CFL = 10.2 and mesh size $100\times100$ at $t=1.5$.}\label{fig:SDF_1_5}
\end{figure}		

In summary, this example illustrates the effectiveness of the proposed EL-RK-FV evolving strategy for the convection terms. This strategy, which involves remapping from Eulerian to Lagrangian spatial approximations, not only achieves high-order spatial and temporal accuracy but also allows for large time-steps while maintaining a non-oscillatory property.

\end{example}

%%%%%%%%%%%%%%5

\begin{example}(The 0D2V Leonard-Bernstein linearized Fokker-Planck equation) Consider the following equation:
	\begin{equation}\label{eq:LBFP}
		f_t+\frac{1}{\epsilon}\left((v_x-\overline{v}_x)f\right)_{v_x}-\frac{1}{\epsilon}\left((v_y-\overline{v}_y)f\right)_{v_y} = \frac{1}{\epsilon}D\left(f_{v_xv_x}+f_{v_yv_y}\right),~~~~v_x,v_y\in[-2\pi,2\pi]
	\end{equation}
	with zero boundary conditions and an equilibrium solution by the given Maxwellian:
	\begin{equation}\label{eq:SBFP_initial}
		f_M(v_x,v_y) = \frac{n}{2\pi RT}\exp\left(\frac{(v_x-\overline{v}_x)^2+(v_y-\overline{v}_y)^2}{2RT}\right),
	\end{equation}
	where, $\epsilon = 1$, gas constant $R=1/6$, temperature $T=3$, thermal velocity $v_{th}=\sqrt{2RT}=\sqrt{2D}=1$, number density $n=\pi$, and bulk velocities $\overline{v}_x=\overline{v}_y=1$. For spatial and temporal accuracy tests, we choose $f(v_x,v_y,0)=f_M(v_x,v_y)$. In Table \ref{tab_2_D_LBFP}, we present the $L^1$, $L^2$, and $L^{\infty}$ errors and the corresponding orders of accuracy for the proposed scheme. The results demonstrate a consistent 3rd-order spatial accuracy.
	
	\begin{table}[!htbp]
		\centering
		\caption{(The 0D2V Leonard-Bernstein linearized Fokker-Planck equation)  $L^1$, $L^2$, and $L^{\infty}$ errors and corresponding orders of accuracy of the EL-RK-FV-WENO scheme for \eqref{eq:LBFP} with initial condition \eqref{eq:SBFP_initial} at $t = 0.5$ with CFL = 1.}\label{tab_2_D_LBFP}
		\centering
		\begin{tabular}{|c|cc|cc|cc|}
			\hline
			mesh&$L^1$ error&order&$L^2$ error&order&$L^{\infty}$ error& order\\
			\hline
			20$\times$  20&    2.18E-03&   ---&    1.38E-02&   ---&    2.95E-01&   ---\\
			40$\times$  40&    3.01E-04&   2.85&    2.05E-03&   2.75&    9.50E-02&   1.63\\
			80$\times$  80&    3.33E-05&   3.18&    1.58E-04&   3.70&    6.40E-03&   3.89\\
			160$\times$ 160&    4.10E-06&   3.02&    1.80E-05&   3.13&    2.88E-04&   4.47\\
			320$\times$ 320&    4.77E-07&   3.10&    2.10E-06&   3.11&    3.20E-05&   3.17\\
			\hline											
		\end{tabular}
	\end{table}

By fixing the spatial mesh and varying the CFL number, the temporal order of accuracy is investigated in Figure \ref{fig:LBFP_CFL_vs_L2error}. The proposed scheme exhibits 3rd-order temporal accuracy and allows the use of large time-steps, as evidenced by the stability for high CFL numbers.

	\begin{figure}[!htbp]
		\centering
		\includegraphics[width=0.45\textwidth]{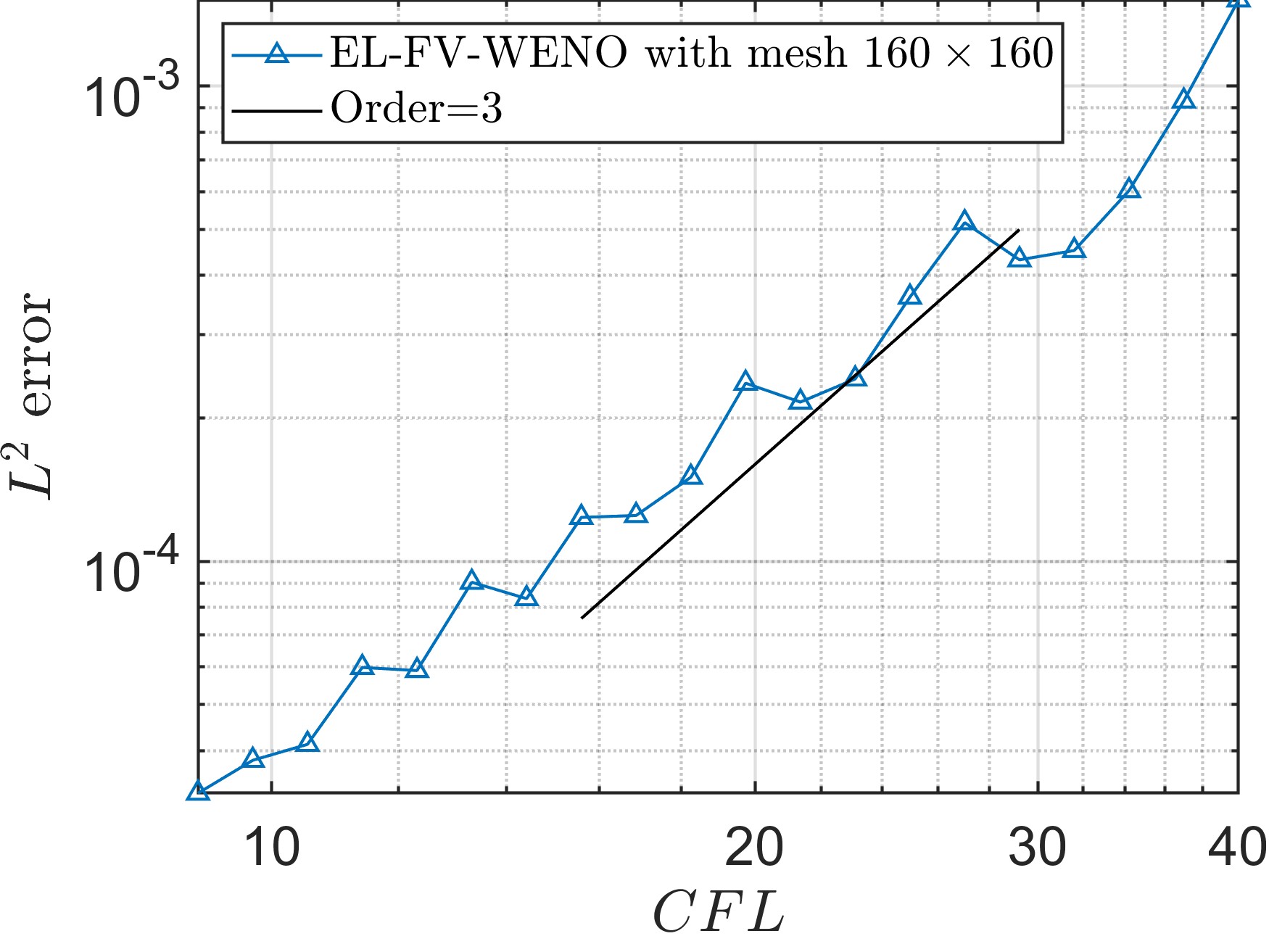}
		\caption{(The 0D2V Leonard-Bernstein linearized Fokker-Planck equation) Log-log plot of CFL numbers versus $L^2$ errors with fixed mesh $160\times160$ at $t = 0.5$ for the EL-RK-FV-WENO scheme.}\label{fig:LBFP_CFL_vs_L2error}
	\end{figure}	

 To test the relaxation of the system, we choose the initial condition $f(v_x, v_y, t=0) = f_{M1}(v_x, v_y) + f_{M2}(v_x, v_y)$, where the parameters of each Maxwellian, $f_{M1}$ and $f_{M2}$, are detailed in Table \ref{tab:Maxwellians}. The two Maxwellians are shifted along the $v_x$ direction with $\overline{v}_y=0$. The evolution of the numerical results, illustrated in Figure \ref{fig:LBFP_contour}, shows that after $t \geq 3$, there is no discernible difference between the numerical solution and the solution at $t=3$, validating the efficacy of the proposed scheme.

 \begin{table}[h!]
\centering
\begin{tabular}{|c|c|c|}
\hline
    & \( f_{M1} \) & \( f_{M2} \) \\
\hline
\( n \) & 1.990964530353041 & 1.150628123236752 \\
\hline
\( \bar{v}_x \) & 0.4979792385268875 & -0.8616676237412346 \\
\hline
\( \bar{v}_y \) & 0 & 0 \\
\hline
\( T \) & 2.46518981703837 & 0.4107062104302872 \\
\hline
\end{tabular}

\caption{(The 0D2V Leonard-Bernstein linearized Fokker-Planck equation) Settings of the Maxwellians $f_{M1}$ and $f_{M2}$.}
\label{tab:Maxwellians}
\end{table}

	\begin{figure}[!htbp]
		\centering
		\subfloat[$t=0$]{
			\includegraphics[width=0.3\textwidth]{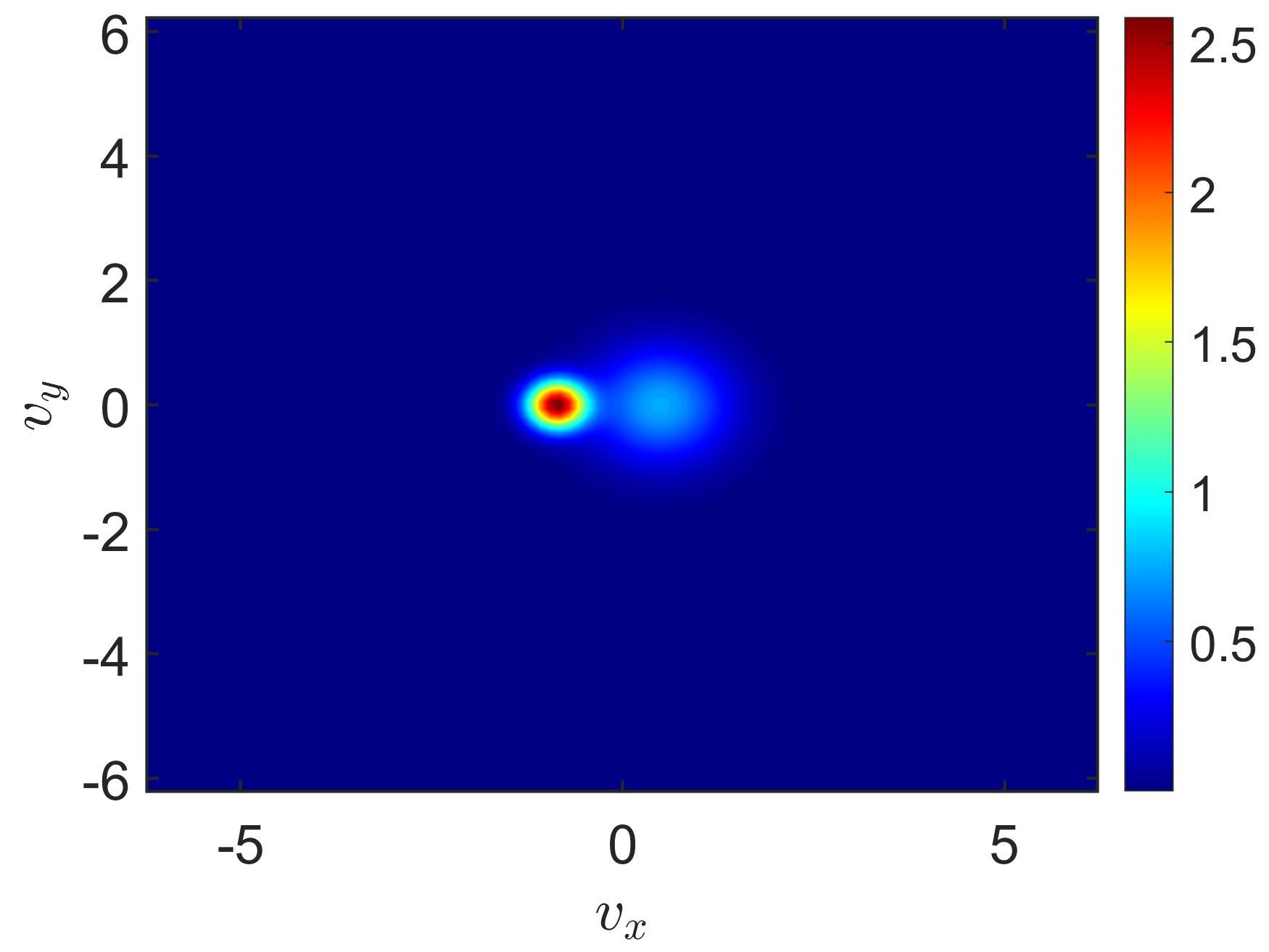}
		}
		\subfloat[$t=0.2$]{
			\includegraphics[width=0.3\textwidth]{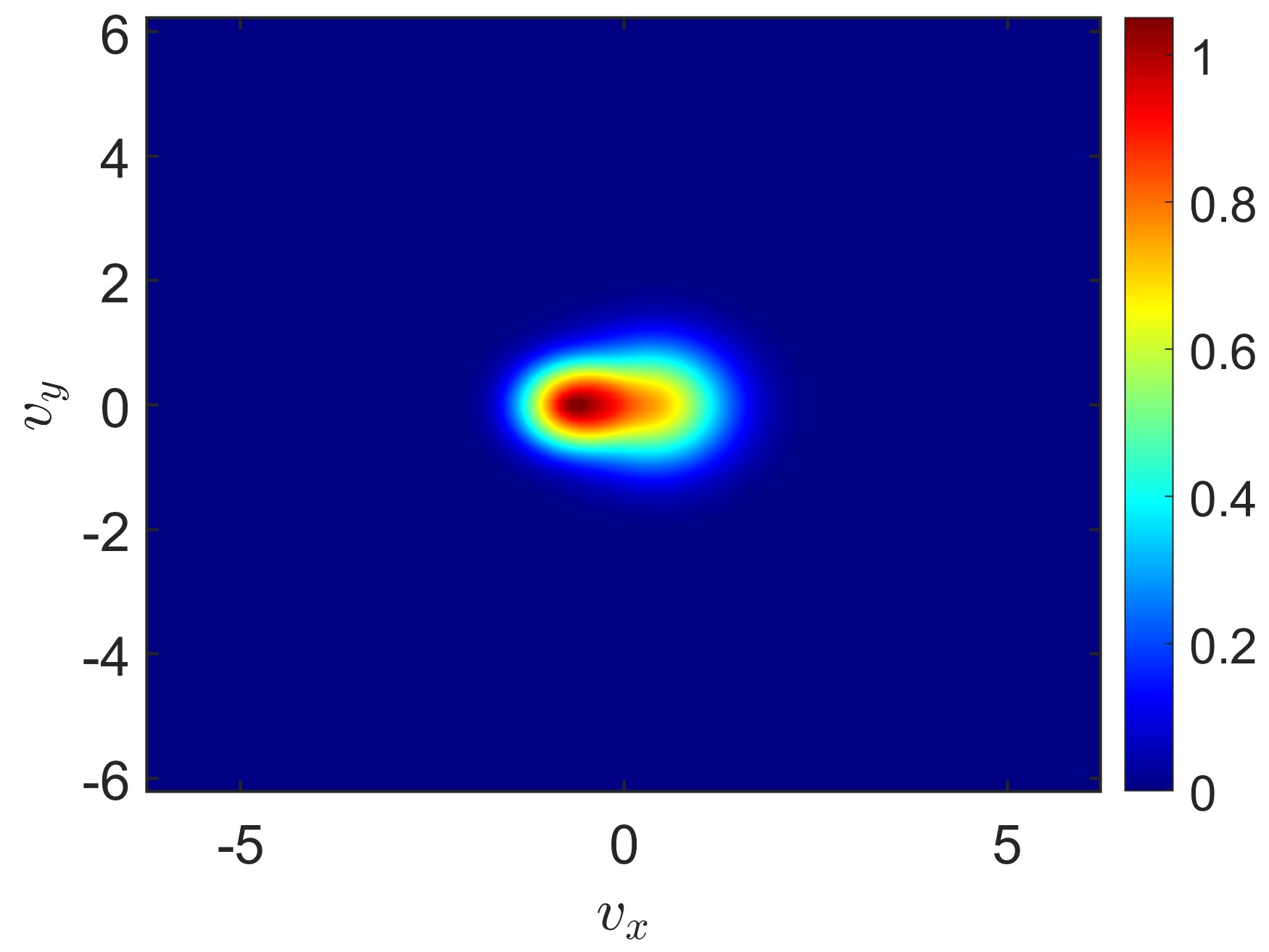}
		}
		\subfloat[$t=0.4$]{
			\includegraphics[width=0.3\textwidth]{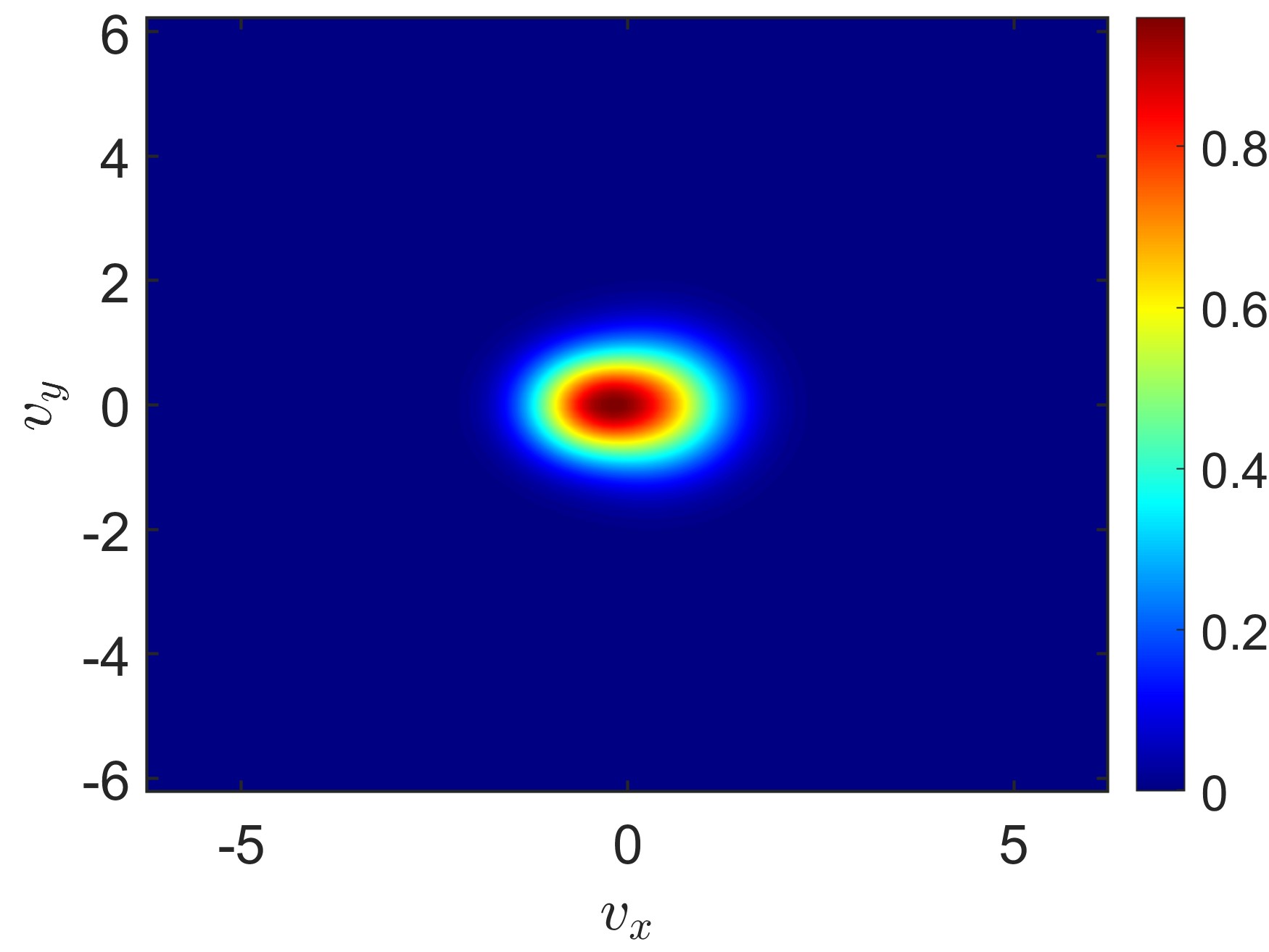}
		}
            \vspace{-0.4cm}
		\subfloat[$t=0.8$]{
			\includegraphics[width=0.3\textwidth]{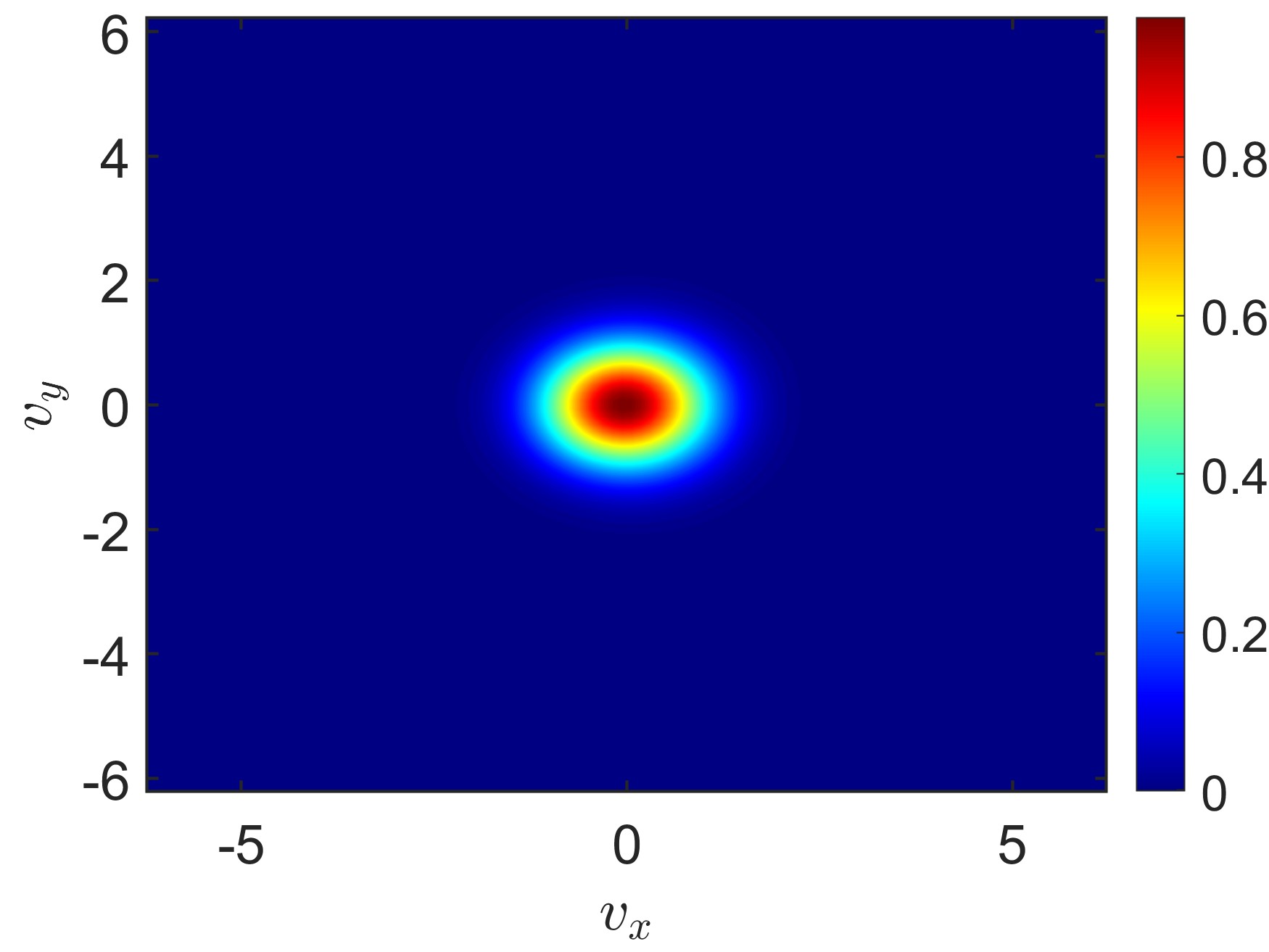}
		}
		\subfloat[$t=3$]{
			\includegraphics[width=0.3\textwidth]{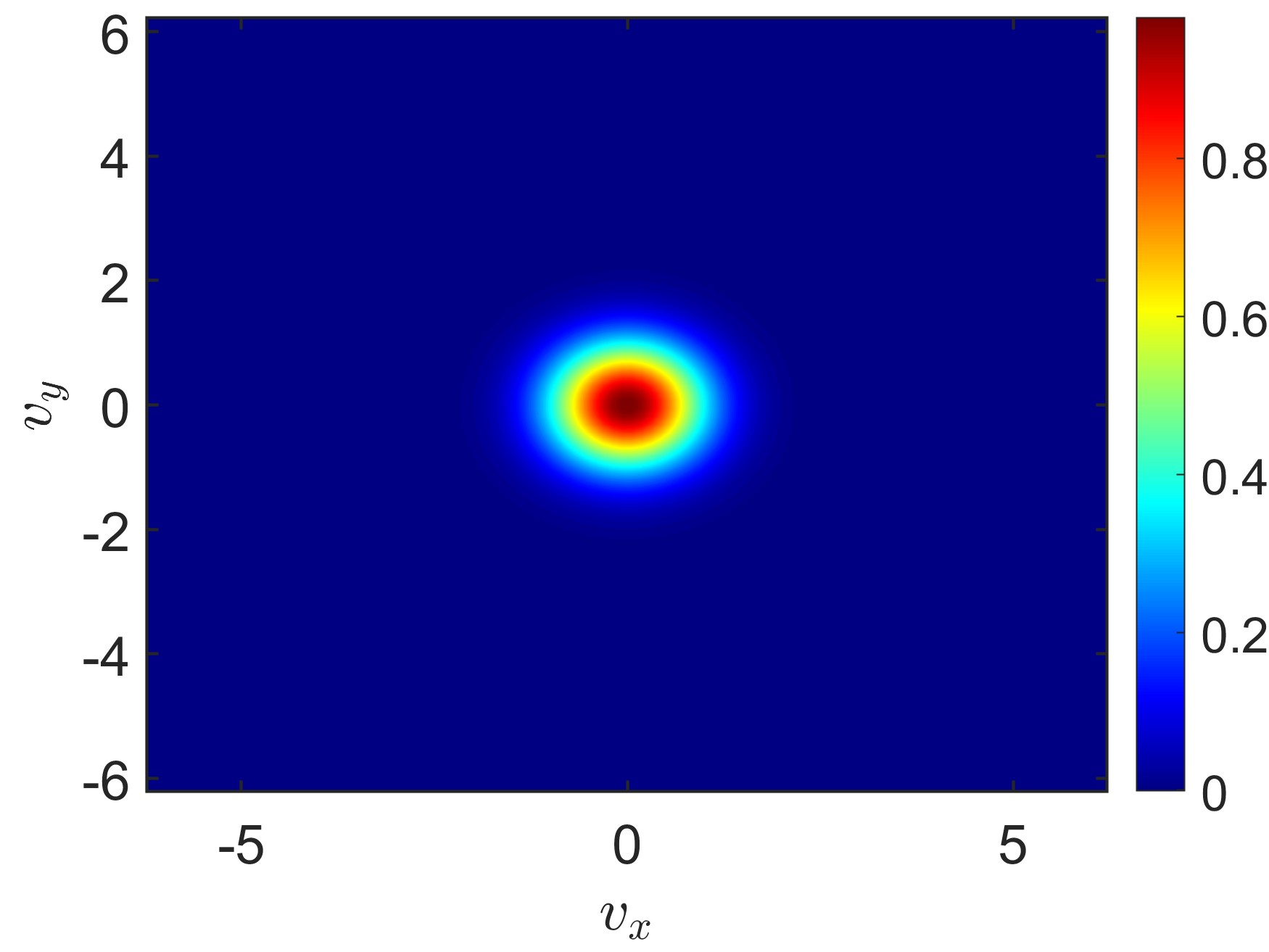}
		}
		\subfloat[$t=20$]{
			\includegraphics[width=0.3\textwidth]{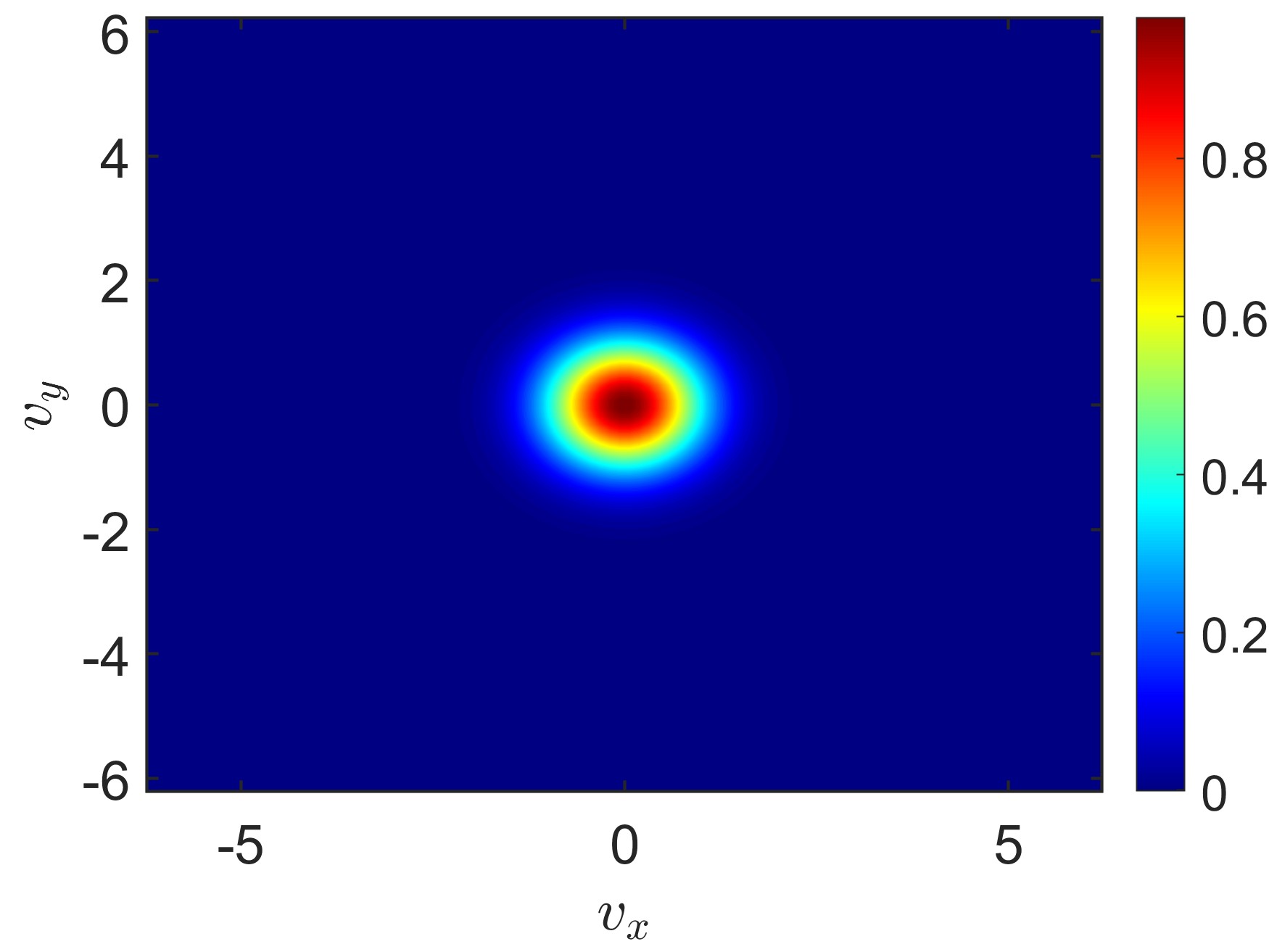}
		}
		\caption{(The 0D2V Leonard-Bernstein linearized Fokker-Planck equation) Contour plots of the numerical results of the EL-RK-FV-WENO scheme with CFL = 3 and mesh size  $100\times100$ at $t=0$ (initial condition), $t=0.2$, $t=0.4$, $t=0.6$, $t=3$, and $t=20$.}\label{fig:LBFP_contour}
	\end{figure}	

In \Cref{fig:LBFP_conservedquanties}, we present the capability of the proposed scheme in preserving various physical conservative quantities. The results indicate that while the scheme is effective in conserving mass, it does not maintain the other quantities to the machine precision.

	\begin{figure}[!htbp]
	\centering
	\subfloat{
		\includegraphics[width=0.4\textwidth]{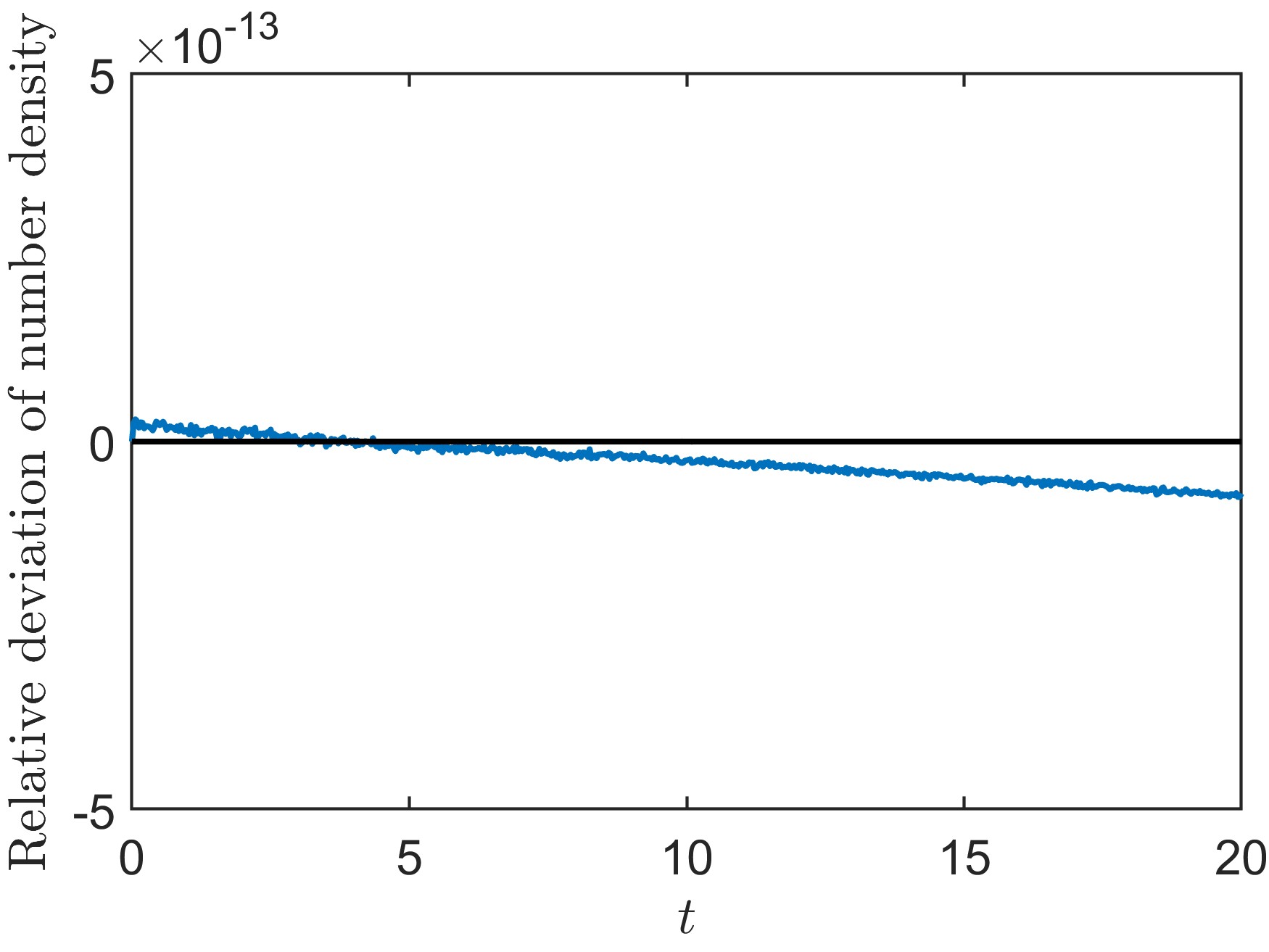}
	}
        \subfloat{
		\includegraphics[width=0.4\textwidth]{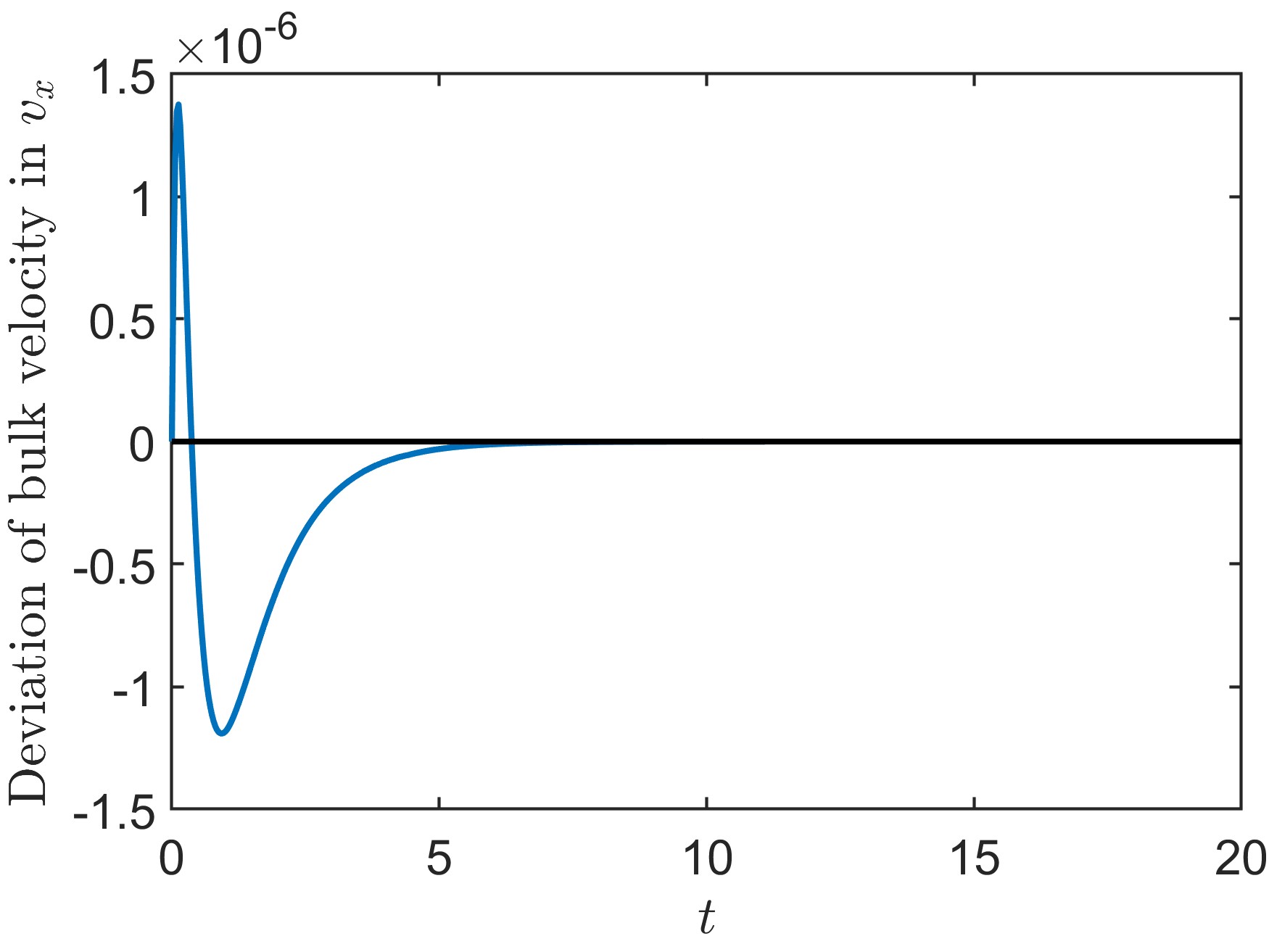}
	}
        \vspace{-0.5cm}
        \subfloat{
		\includegraphics[width=0.4\textwidth]{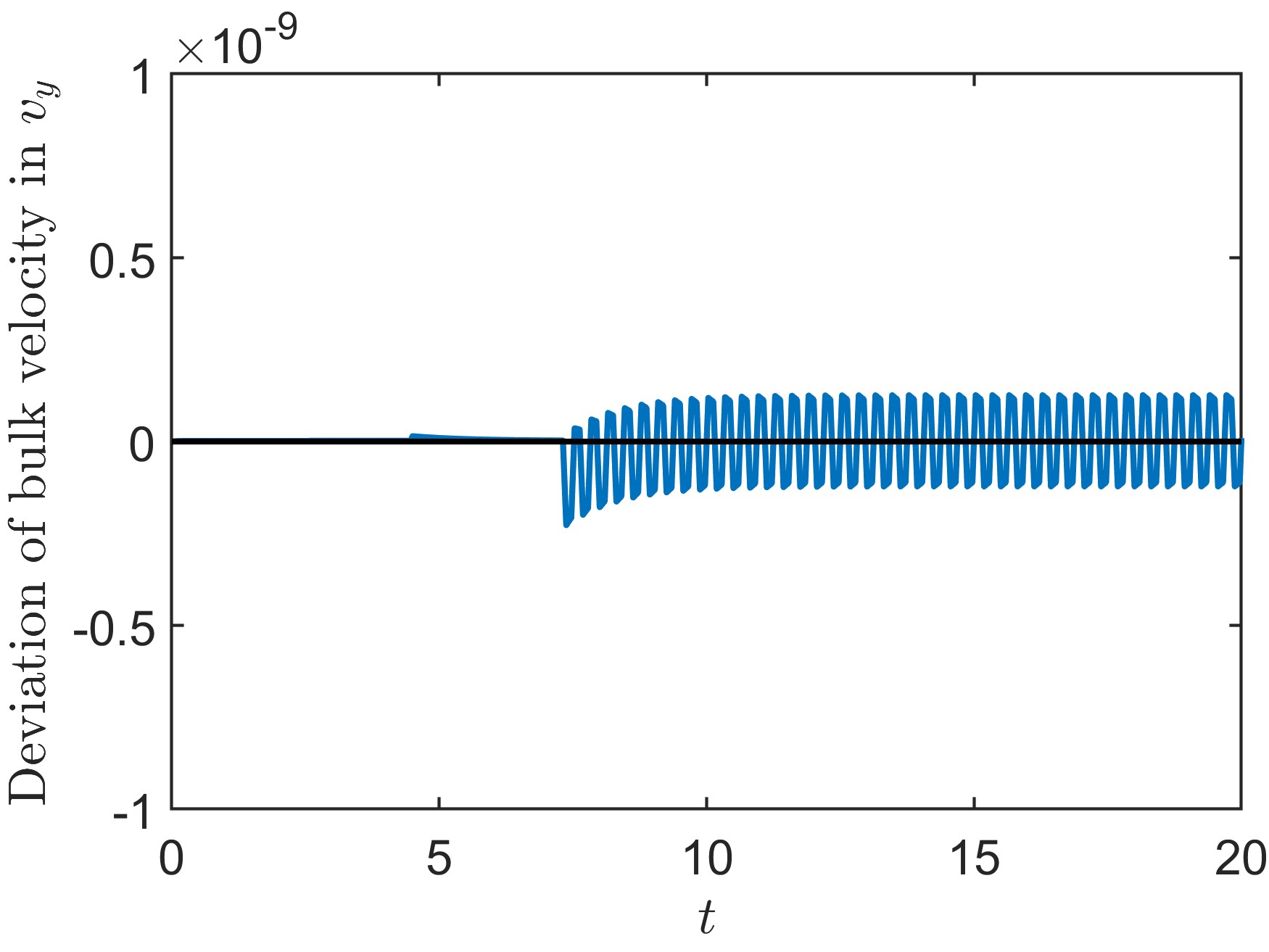}
	}
        \subfloat{
		\includegraphics[width=0.4\textwidth]{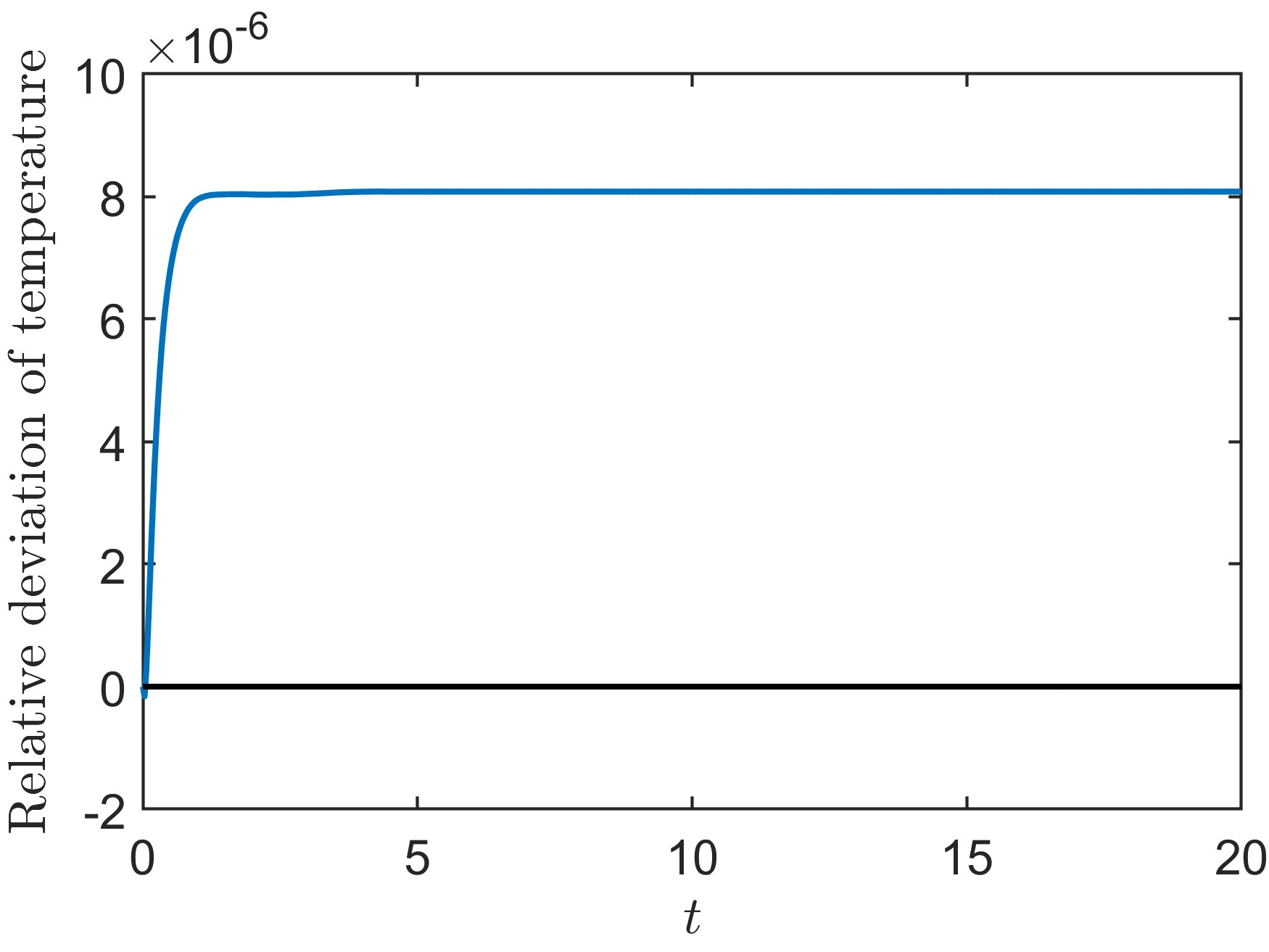}
	}
	\caption{(The 0D2V Leonard-Bernstein linearized Fokker-Planck equation) Relative deviation (or deviation) of  number density, bulk velocity in $v_x$, bulk velocity in $v_y$, and temperature for the EL-RK-FV-WENO scheme with
CFL = 10.2 and with mesh $160\times160$ from $t=0$ to $t=20$ .}\label{fig:LBFP_conservedquanties}
	\end{figure}	

To encapsulate, this example primarily demonstrates the effectiveness of the IMEX temporal discretization, which combines the implicit solver for the diffussion term and the EL evolving strategy for the convection terms, enabling the use of large time-steps.
\end{example}

\subsection{Nonlinear models}

\begin{example}(Kelvin-Helmholtz instability problem) Consider the guiding center Vlasov model \cite{shoucri_two-level_1981,crouseilles_conservative_2010}:
\begin{equation}\label{eq:GC}
    \begin{split}
        \rho_t+\nabla\cdot(\mathbf{E}^{\perp}\rho)=0,\\
        -\Delta \Phi=\rho,~~ \mathbf{E}^{\perp}=\left(-\Phi_y,\Phi_x\right),
    \end{split}
\end{equation}
with the periodic boundary condition and the following initial condition:
	\begin{equation}\label{GC_KHI}
		\begin{split}
			\rho(x,y,0) = \text{sin}(y) + 0.015\text{cos}(0.5x),~~x\in[0,4\pi],~~y\in[0,2\pi],
		\end{split}
	\end{equation}
 where $\rho$ is the charge density and $\mathbf{E}$ is the electric field. We validate the 3rd-order spatial and temporal accuracy of the proposed scheme in \Cref{tab_2_D_KHI} and \Cref{fig:KHI_CFL_vs_L2error}. Additionally, the stability of the scheme with large time-steps up to CFL = 50 is verified.

\begin{table}[!htbp]
	\centering
	\caption{ (Kelvin-Helmholtz instability problem)  $L^1$, $L^2$, and $L^{\infty}$ errors and corresponding orders of accuracy of the EL-RK-FV-WENO scheme for \eqref{eq:GC} with initial condition \eqref{GC_KHI} at $t = 5$ with CFL = 1.}\label{tab_2_D_KHI}
	\centering
	\begin{tabular}{|c|cc|cc|cc|}
		\hline
		mesh&$L^1$ error&order&$L^2$ error&order&$L^{\infty}$ error&order\\
		\hline
            16$\times$  16&    6.79E-03&   ---&    1.11E-02&   ---&    6.68E-02&   ---\\
            32$\times$  32&    5.48E-04&   3.63&    9.46E-04&   3.55&    1.18E-02&   2.50\\
            64$\times$  64&    4.35E-05&   3.66&    7.15E-05&   3.73&    1.43E-03&   3.05\\
            128$\times$ 128&    4.37E-06&   3.31&    6.45E-06&   3.47&    1.98E-04&   2.85\\
            256$\times$ 256&    5.23E-07&   3.06&    6.72E-07&   3.26&    1.73E-05&   3.51\\
		\hline											
	\end{tabular}
\end{table}

\begin{figure}[!htbp]
	\centering
	\includegraphics[width=0.45\textwidth]{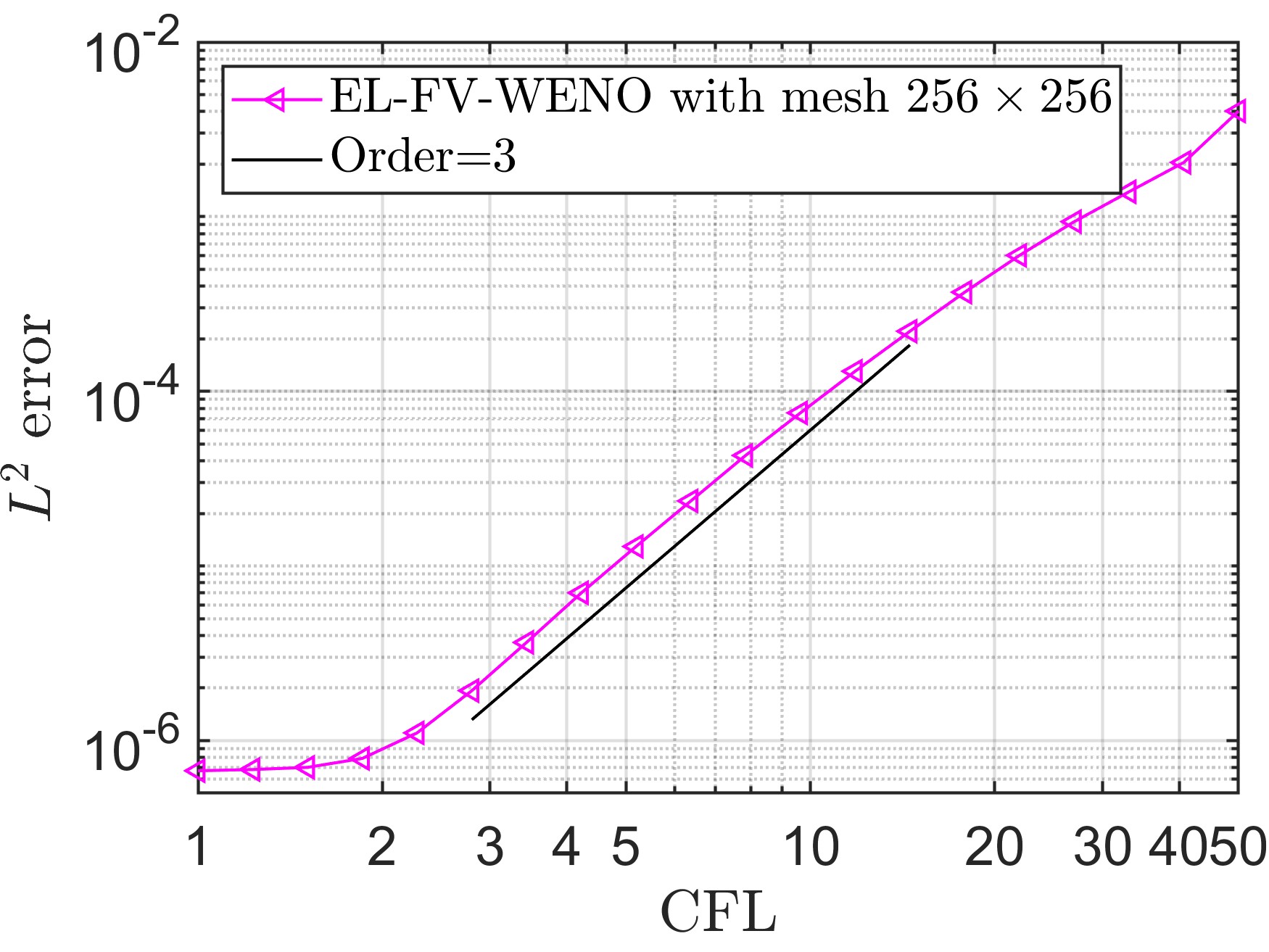}
	\caption{(Kelvin-Helmholtz instability problem) Log-log plot of CFL numbers versus $L^2$ errors with fixed mesh $256\times256$ at $t = 5$ of the EL-RK-FV-WENO scheme.}\label{fig:KHI_CFL_vs_L2error}
\end{figure}	

 In \Cref{fig:KHI_meshplot}, we display the mesh plot and contour plot of the numerical solution of the EL-RK-FV-WENO scheme at $t=40$. The result is comparable with the one of our fourth-order SL-FV-WENO scheme in \cite{zheng_fourth-order_2022}.
\begin{figure}[!htbp]
	\centering
	\subfloat[Mesh plot]{
		\includegraphics[width=0.45\textwidth]{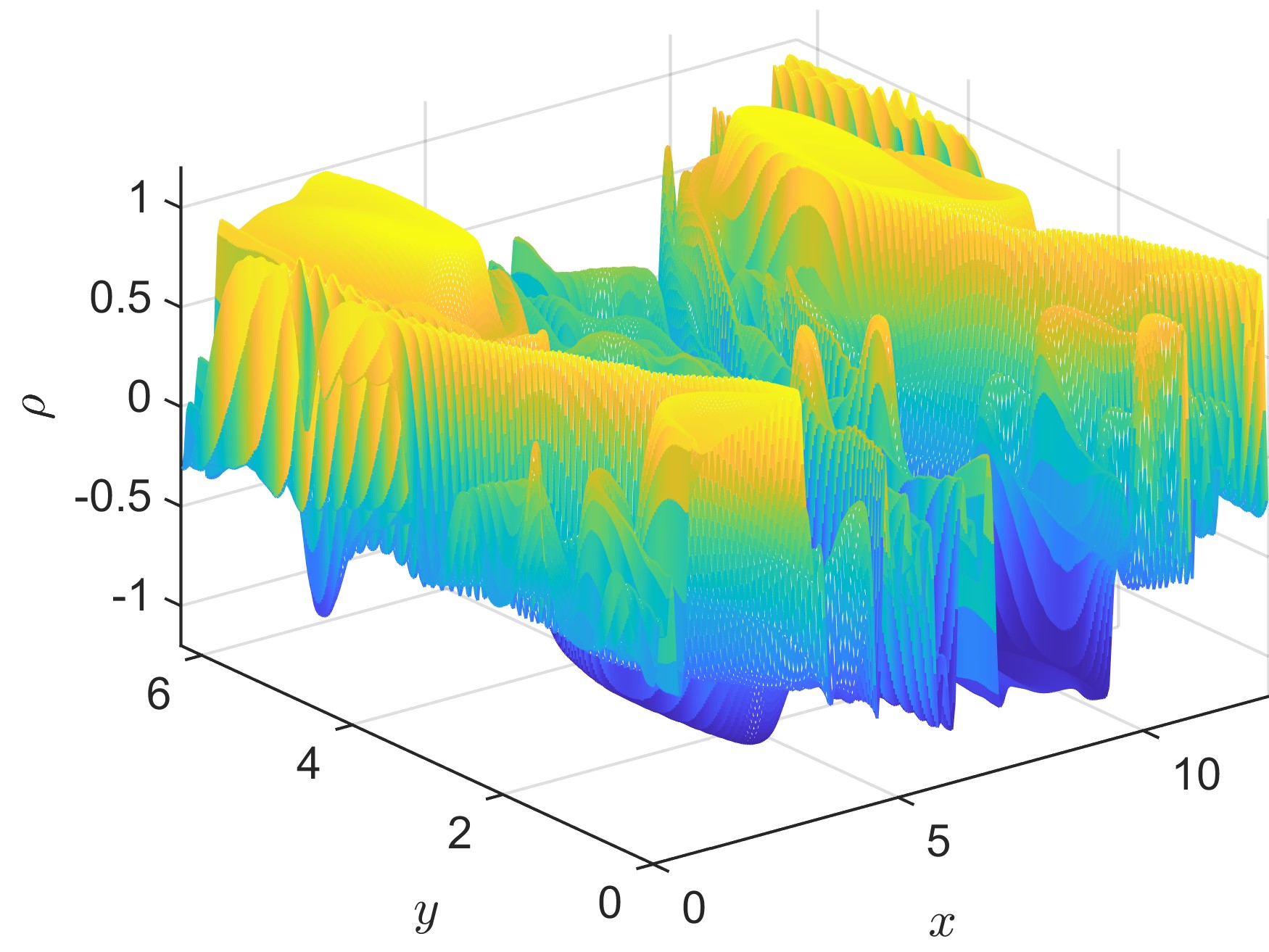}
	}
 	\subfloat[Contour plot]{
		\includegraphics[width=0.45\textwidth]{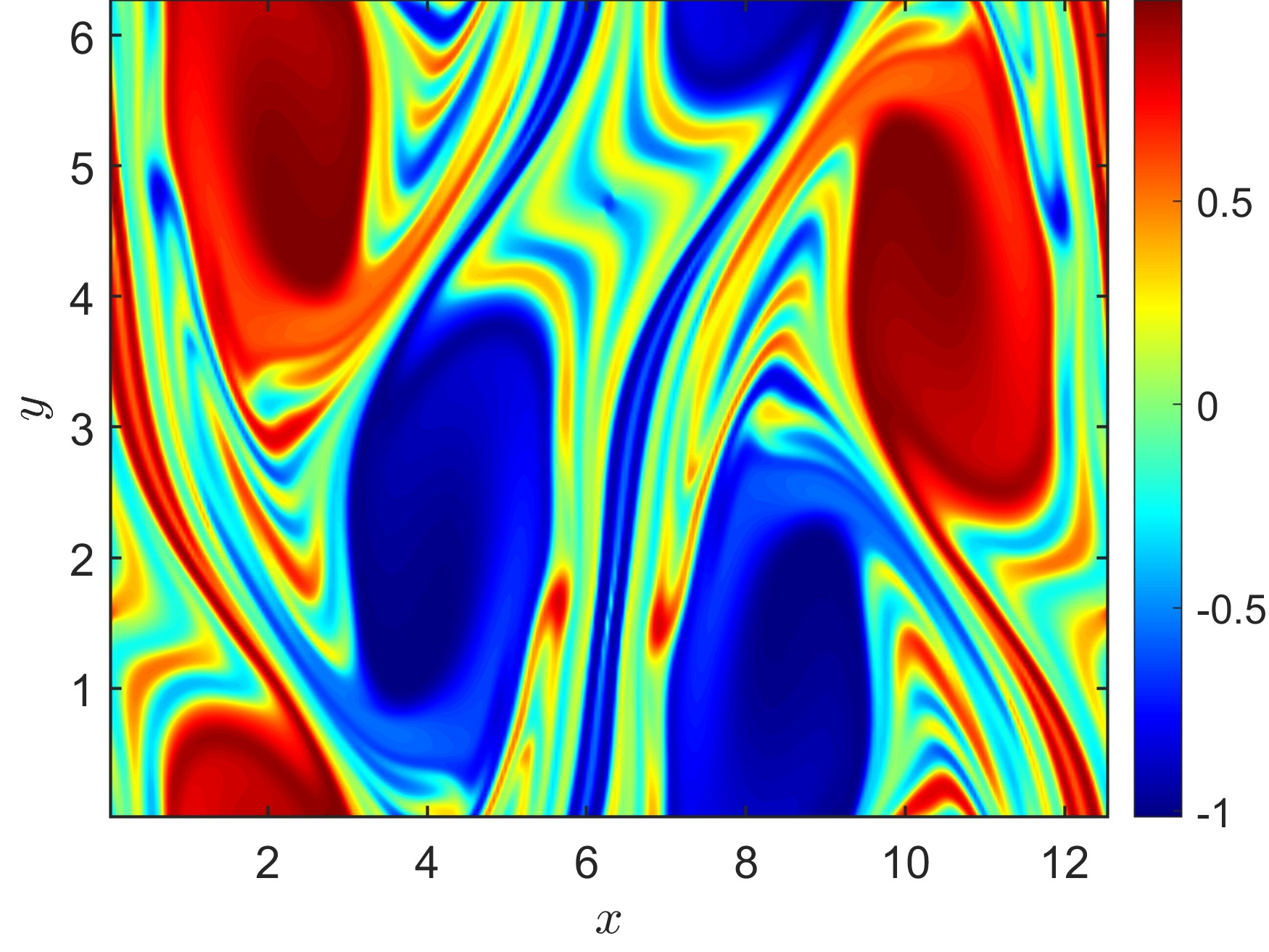}
	}
	\caption{(Kelvin-Helmholtz instability problem) Mesh plot and contour plot of the numerical solution of the EL-RK-FV-WENO scheme with CFL = 10.2 and with mesh $256\times256$ at $t=40$.}\label{fig:KHI_meshplot}
\end{figure}	

In \Cref{fig:KHI_conserve}, we show the deviation of mass, relative deviation of energy
and entropy of the proposed scheme from $t=0$ to $t=50$. As shown, the proposed scheme is mass conservative. The magnitudes of relative deviation of energy and entropy results are comparable with the ones in \cite{zheng_fourth-order_2022}.

\begin{figure}[!htbp]
	\centering
	\subfloat[Deviation of mass]{
		\includegraphics[width=0.3\textwidth]{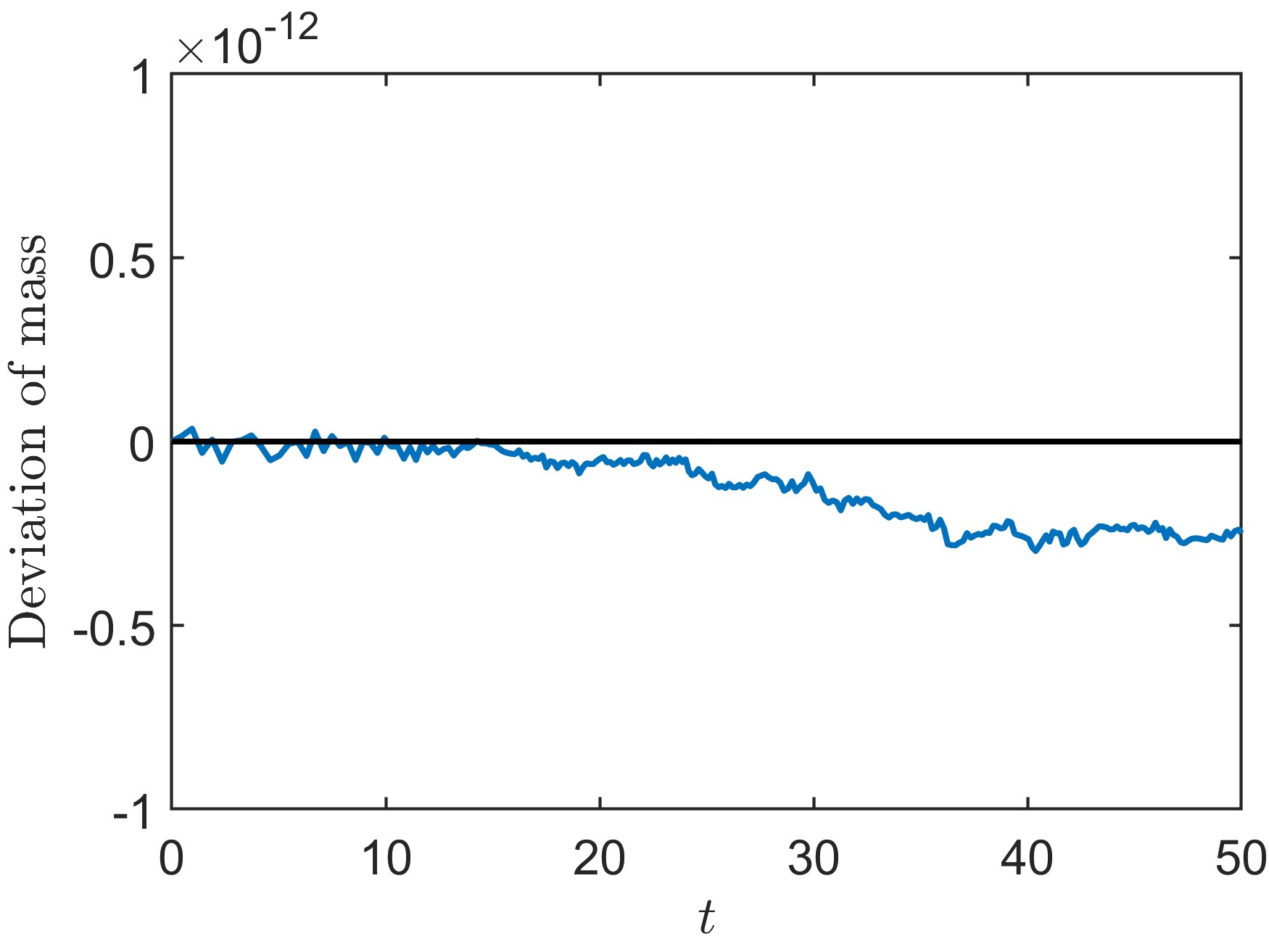}
	}
	\subfloat[Relative deviation of energy]{
		\includegraphics[width=0.3\textwidth]{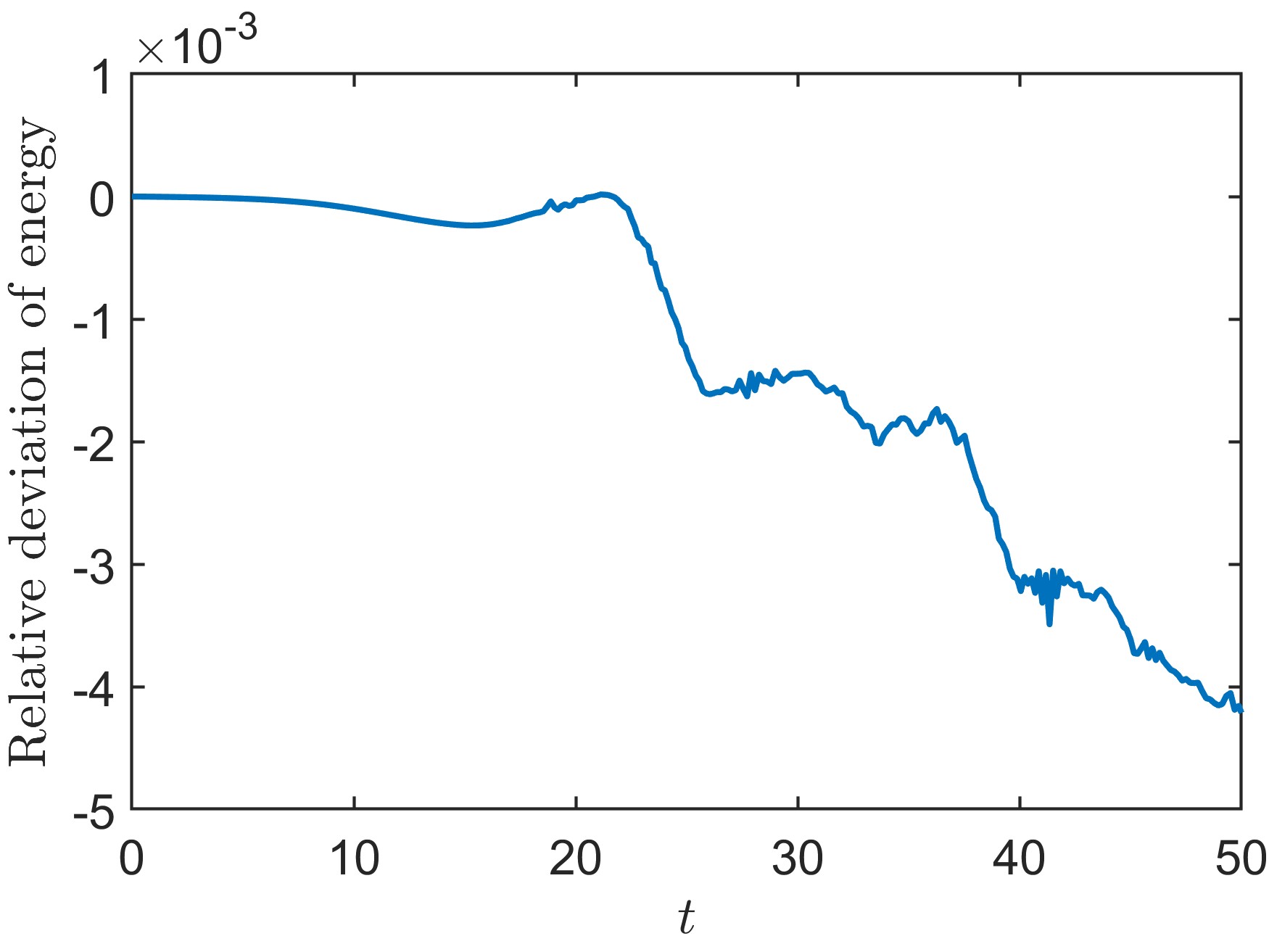}
	}
	\subfloat[Relative deviation of entropy]{
		\includegraphics[width=0.3\textwidth]{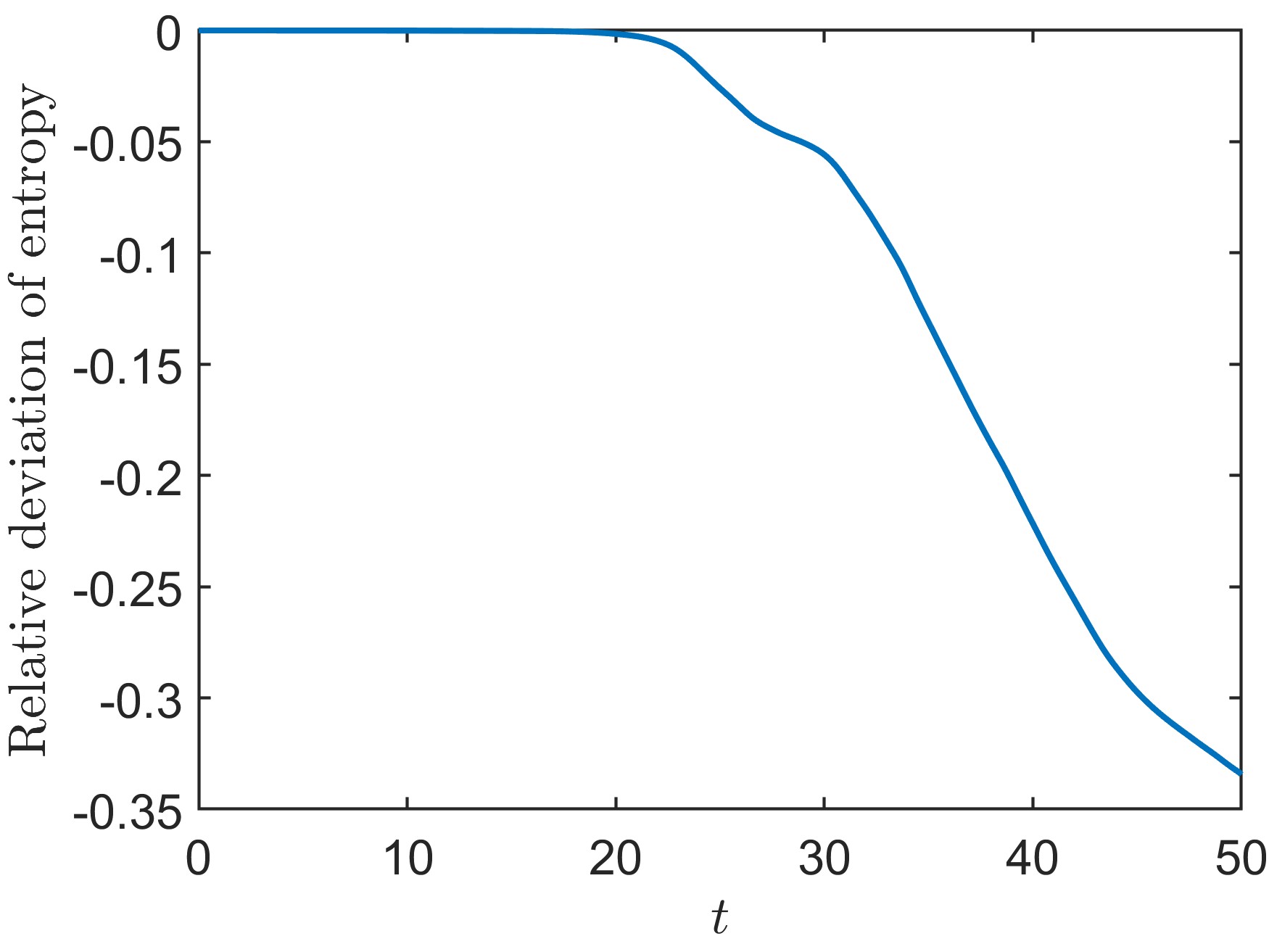}
	}
	\caption{(Kelvin-Helmholtz instability problem) Deviation of mass, relative deviation of energy and entropy for the EL-RK-FV-WENO scheme with
CFL = 10.2 and with mesh $256\times256$ from $t=0$ to $t=50$ .}\label{fig:KHI_conserve}
\end{figure}	

This nonlinear example effectively showcases the success of the generalization strategy for nonlinear model outlined in \Cref{sec:nonlinear}, which is achieved with minimal additional cost.

\end{example}

\begin{example}(Incompressible Navier-Stokes equations) The governing equations are as follows:
    \begin{equation}\label{eq:INS}
        \begin{split}
            \omega_t+(u\omega)_x+(v\omega)_y=\nu(\omega_{xx}+\omega_{yy}),\\
            \Delta\psi=\omega,~~\left(u,v\right)=\left(\psi_y,\psi_x\right),
        \end{split}
    \end{equation}
    where $\omega$ is the vorticity of the flow, $\left(u,v\right)$ is the velocity field, and $\nu$ is the kinematic viscosity, which is set to be $\frac{1}{100}$. We first consider an initial condition given by:
    \begin{equation}\label{eq:INS_init1}
    \omega(x,y,0)=-2\sin(x)\sin(y),\quad x\in [0,2\pi],~~y\in [0,2\pi]
    \end{equation}
    with the exact solution $\omega(x,y,t)=-2\sin(x)\sin(y)\exp(-2t\nu)$. Similarly, 3rd-order spatial and temporal order of accuracy of validated in \Cref{tab_2_D_IE} and \Cref{fig:IEIMEX_CFL_vs_L2error} respectively. For this problem, we observe that large time-steps are allowed for the proposed scheme up to CFL = 40.

	\begin{table}[!htbp]
	\centering
	\caption{(Incompressible Navier-Stokes equations) $L^1$, $L^2$, and $L^{\infty}$ errors and corresponding orders of accuracy of the EL-RK-FV-WENO scheme for \eqref{eq:INS} with initial condition \eqref{eq:INS_init1} at $t = 0.5$ with CFL = 1.}\label{tab_2_D_IE}
	\centering
	\begin{tabular}{|c|cc|cc|cc|}
		\hline
		mesh&$L^1$ error&order&$L^2$ error&order&$L^{\infty}$ error& order\\
		\hline
  		16$\times$  16&    3.90E-03&   ---&    4.65E-03&   ---&    1.14E-02&   ---\\
		32$\times$  32&    4.88E-04&   3.00&    5.82E-04&   3.00&    1.45E-03&   2.98\\
		64$\times$  64&    6.09E-05&   3.00&    7.27E-05&   3.00&    1.82E-04&   2.99\\
		128$\times$ 128&    7.61E-06&   3.00&    9.08E-06&   3.00&    2.29E-05&   2.99\\
		256$\times$ 256&    9.51E-07&   3.00&    1.13E-06&   3.00&    2.86E-06&   3.00\\
		\hline											
	\end{tabular}
\end{table}

	\begin{figure}[!htbp]
	\centering
	\includegraphics[width=0.45\textwidth]{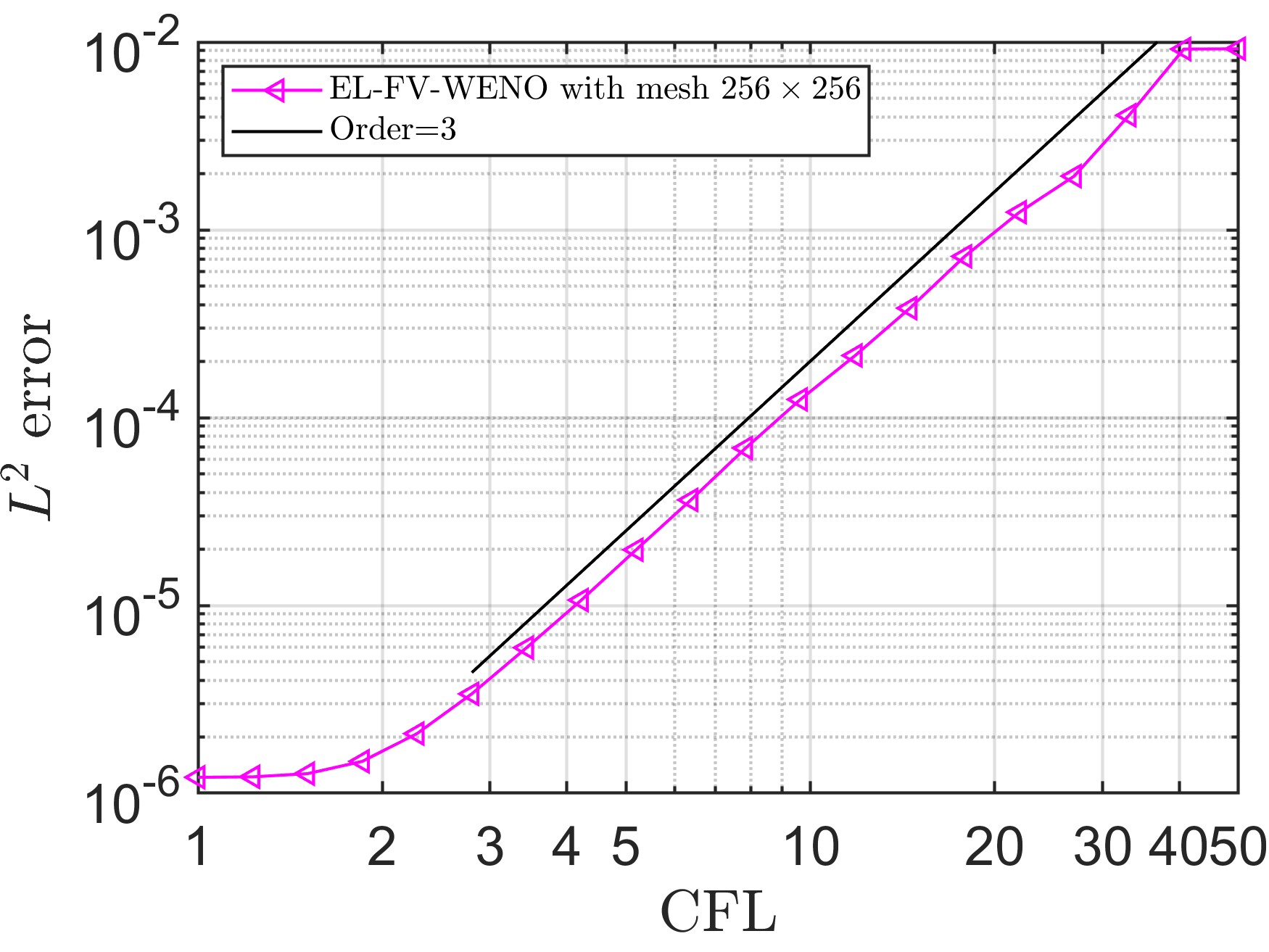}
	\caption{(Incompressible Navier-Stokes equations) Log-log plot of CFL numbers versus $L^2$ errors with fixed mesh $256\times256$ at $t = 0.5$ of the EL-RK-FV-WENO scheme.}\label{fig:IEIMEX_CFL_vs_L2error}
\end{figure}	

For a more complex scenario, we consider the incompressible Navier-Stokes equation \eqref{eq:INS} on $[0, 2\pi]^2$ with the following initial condition (the vortex patch problem)
\begin{equation}\label{eq:ini_vpp}
    \omega(x,y,0)=\begin{cases}
        -1,\quad\frac{\pi}{2}\leq x \leq\frac{3\pi}{2},~~\frac{\pi}{4}\leq y \leq\frac{3\pi}{4};\\
        1,\quad~~\frac{\pi}{2}\leq x \leq\frac{3\pi}{2},~~\frac{5\pi}{4}\leq y \leq\frac{7\pi}{4};\\
        0\quad~~~~ \text{otherwise}
    \end{cases}
\end{equation}
with zero boundary condition. We provide the mesh plot and contour plot of the numerical solution of the proposed scheme at $t = 5$ in \Cref{fig:VPP_meshplot}. The numerical result in \Cref{fig:VPP_meshplot} is comparable with the one in \cite{xiong2015high}.

\begin{figure}[!htbp]
	\centering
	\subfloat[Mesh plot]{
		\includegraphics[width=0.45\textwidth]{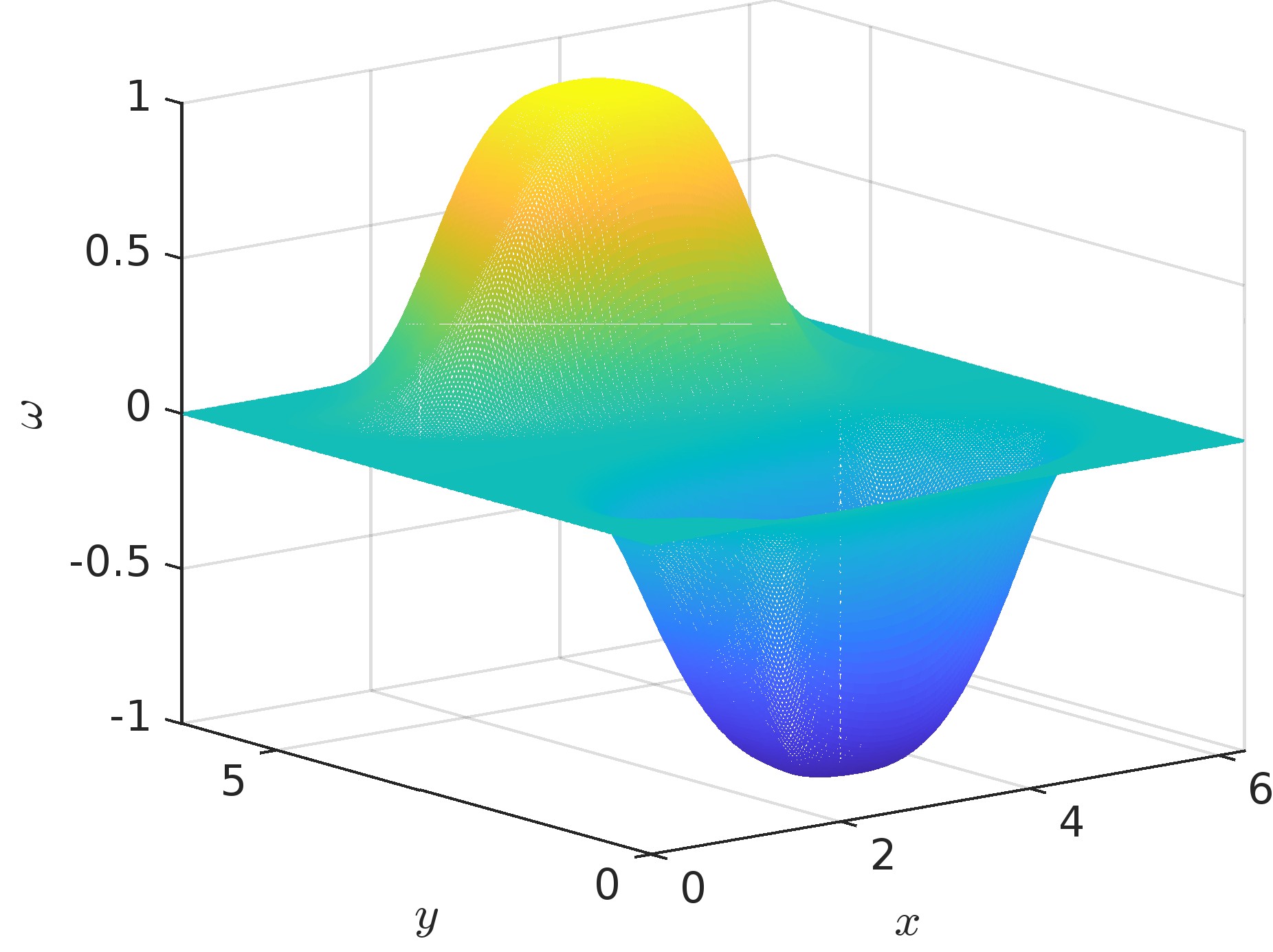}
	}
 	\subfloat[Contour plot]{
		\includegraphics[width=0.45\textwidth]{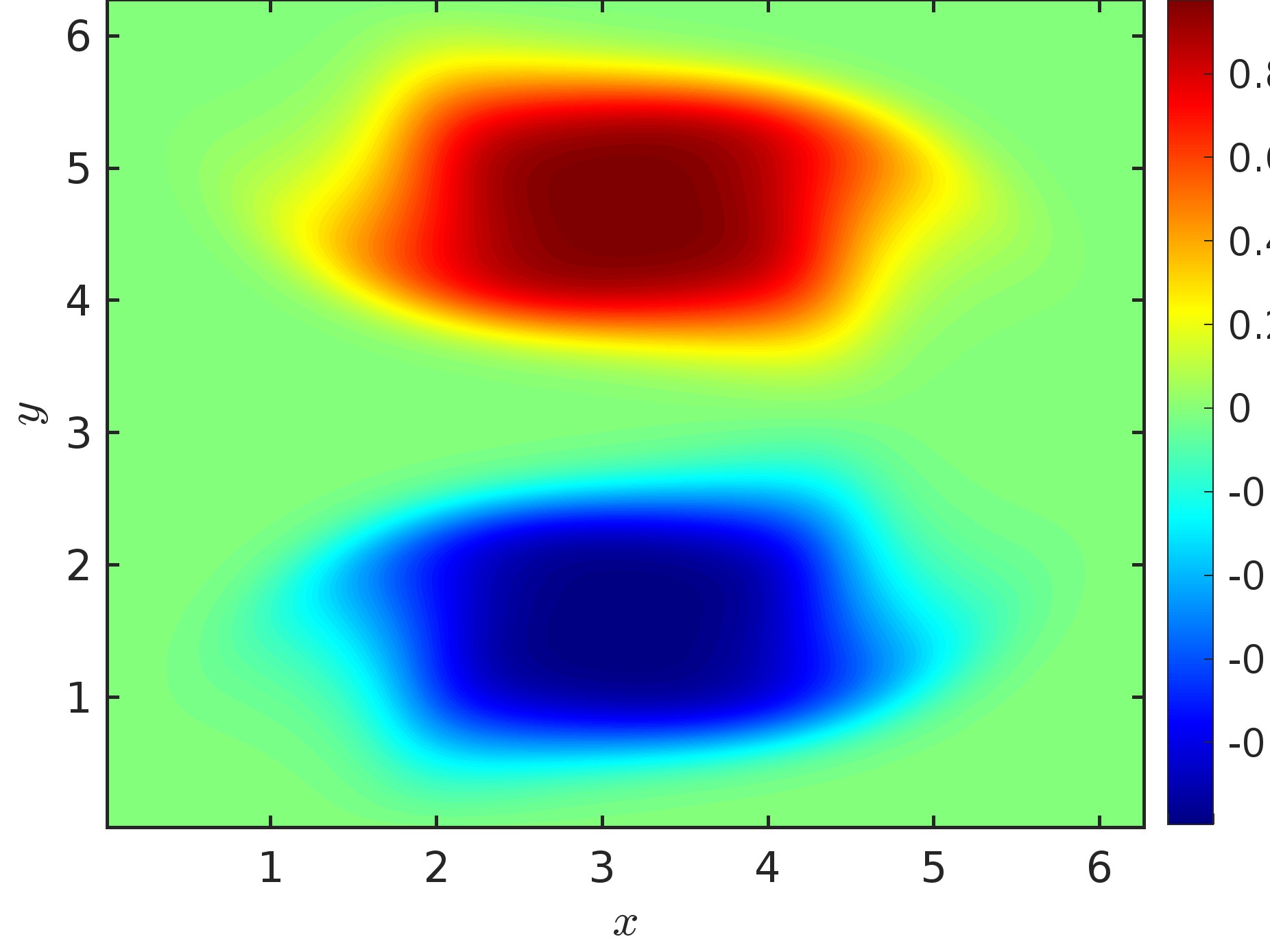}
	}
	\caption{(Vortex patch problem) Mesh plot and contour plot of the numerical solution of the EL-RK-FV-WENO scheme with CFL = 10.2 and with mesh $256\times256$ at $t=5$.}\label{fig:VPP_meshplot}
\end{figure}	

Reflecting on this example, we note that the proposed EL-RK-FV-WENO scheme is able
to simulate a nonlinear convection-diffusion equation with all the designed good properties. This
is one of the major reasons why this EL-RK-FV-WENO scheme is attractive compared with our
previous SL-FV-WENO scheme \cite{zheng_fourth-order_2022}.
\end{example}

\section{Conclusion}\label{sec:conclusion}
In this paper, we introduce a third-order EL-RK-FV-WENO scheme for convection-diffusion equations. By defining a modified velocity field and corresponding flux-form semi-discretization, we relax the time-step constraint. Spatial discretization is carefully designed to fit the EL formulation and overcome the challenges brought by the modified velocity field. Compared with the SL-FV scheme in \cite{zheng_fourth-order_2022}, the proposed EL-RK-FV-WENO scheme is capable of simulating nonlinear convection-diffusion equations while inheriting the ability to apply large time-steps. Extensive numerical tests are conducted, verifying the effectiveness of the proposed scheme.

\appendix
%\appendixpage
%\addappheadtotoc
%\setcounter{section}{0}
%\renewcommand\thesection{\Alph{section}}

\section{Third-order WENO-ZQ reconstruction method for Eulerian mesh}\label{sec:WENO_ZQ_Eulerian}
In this appendix, we introduce the 3rd-order WENO-ZQ reconstruction for the Eulerian mesh. For convenience, we assume $\Delta x_i \equiv \Delta x$ and $\Delta y_j \equiv \Delta y$ for all $i,~j$. We define $\mu_i(x):=\frac{x-x_i}{\Delta x}$, $\nu_j(y):=\frac{y-y_j}{\Delta y}$ and introduce a set of local orthogonal polynomials as $\{P^{(i,j)}_l(x,y)\}$ for a given cell $I_{i,j}$:

\begin{equation}
    \begin{aligned}
        &P^{(i,j)}_1 := 1, \quad P^{(i,j)}_2 := \mu_i(x), \quad P^{(i,j)}_3 := \nu_j(y), \\
        &P^{(i,j)}_4 := \mu_i^2(x) - \frac{1}{12}, \quad P^{(i,j)}_5 := \mu_i(x)\nu_j(y), \quad P^{(i,j)}_6 := \nu_j^2(y) - \frac{1}{12}.
    \end{aligned}
\end{equation}

We assume that $\overline{u}_5 := \overline{u}_{i,j}$ and $I_5 := I_{i,j}$, while $\{u_s\}$ and $\{I_s\}$ represent corresponding cell averages and Eulerian cells based on the serial number in Figure \ref{fig:stencil}. The reconstruction procedure is performed as follows:
\begin{figure}[htb]
	\centering
	\begin{tikzpicture}
		%%%%%%%%%%%%%%%%%%%%%%%%%%%%%%%%%%%%%%%%%%%%%%%%%%%%%%%%%%mesh
		
		\draw[black,thin] (0.5,0.5) node[left] {} -- (5,0.5) node[right]{};
		\draw[black,thin] (0.5,2) node[left] {} -- (5,2) node[right]{};
		\draw[black,thin] (0.5,3.5) node[left] {} -- (5,3.5) node[right]{};
		\draw[black,thin] (0.5,5) node[left] {} -- (5,5) node[right]{};
		
		\draw[black,thin] (0.5,0.5) node[left] {} -- (0.5,5) node[right]{};
		\draw[black,thin] (2,0.5) node[left] {} -- (2,5) node[right]{};
		\draw[black,thin] (3.5,0.5) node[left] {} -- (3.5,5) node[right]{};
		\draw[black,thin] (5,0.5) node[left] {} -- (5,5) node[right]{};
		
		\draw(0.5,0.5) node[above right=14pt] {$1$} -- (2,0.5) node[above right=14pt] {$2$};
		\draw(3.5,0.5) node[above right=14pt] {$3$} -- (3.5,0.5) node[above right=14pt] {};
		
		\draw(0.5,2) node[above right=14pt] {} -- (0.5,2) node[above right=14pt] {$4$};
		\draw(2,2) node[above right=14pt] {$5$} -- (3.5,2) node[above right=14pt] {$6$};
		
		\draw(0.5,3.5) node[above right=14pt] {$7$} -- (2,3.5) node[above right=14pt] {$8$};
		\draw(3.5,3.5) node[above right=14pt] {$9$} -- (3.5,3.5) node[above right=14pt] {};
		
		\draw(1.25,0) node[below] {$i-1$} -- (1.25,0) node[above right=14pt] {};
		\draw(1.25+1.5,0) node[below] {$i$} -- (1.25+1.5,0) node[above right=14pt] {};
		\draw(1.25+3,0) node[below] {$i+1$} -- (1.25+3,0) node[above right=14pt] {};
		
		\draw(0,-0.25+1.5) node[left] {$j-1$} -- (0,-0.25+1.5) node[above right=14pt] {};
		\draw(0,-0.25+3) node[left] {$j$} -- (0,-0.25+3) node[above right=14pt] {};
		\draw(0,-0.25+4.5) node[left] {$j+1$} -- (0,-0.25+4.5) node[above right=14pt] {};
		
	\end{tikzpicture}
\caption{Stencil for the 3rd-order WENO-ZQ reconstruction.}
\label{fig:stencil}
\end{figure}

\begin{enumerate}
	\item[\textbf{Step 1}] Construct a quadratic polynomial $q_0(x,y) = \sum\limits_{l=1}^{6}a_l^{q_0}P_l^{(i,j)}(x,y)$ using a special least-squares procedure. We define $$V:=\{p(x,y)\in P^2(I_{i,j})|\frac{1}{\Delta x\Delta y}\int_{I_s}p(x,y)dxdy = \overline{u}_s,~s=2,4,5,6,8\},$$
$$E(p(x,y)):=\left[\sum\limits_{s=1,3,7,9}\left(\frac{1}{\Delta x\Delta y}\int_{I_s}p(x,y)dxdy - \overline{u}_s\right)^2\right]^{\frac{1}{2}}.$$
 Then, we determine that $q_0(x,y)$ is the unique polynomial satisfying:

	\begin{equation}
		E(q_0(x,y))=\min\limits_{p\in V}{E(p(x,y))}.
	\end{equation}
	%We provide the explicit expressions of $\{a_l^{q_0}\}$ in .

	\item[\textbf{Step 2}] Construct eight linear polynomials $\{q_k(x,y)\}_{k=1}^8=\{\sum\limits_{l=1}^{3}a_l^{q_k}P^{(i,j)}_l(x,y)\}$ satisfying:

	\begin{equation}
		\frac{1}{\Delta x\Delta y}\iint_{I_5}q_k(x,y)dxdy=\overline{u}_5\quad \text{for}~~k=1,2,\ldots,8,
	\end{equation}
	and
	\begin{equation}
		\frac{1}{\Delta x\Delta y}\iint_{I_s}q_k(x,y)dxdy=\overline{u}_s,
	\end{equation}
	where
	\begin{equation*}
		\begin{split}
			 &s=1,2\quad\text{for}~~k=1;~s=2,3\quad\text{for}~~k=2; \\ &s=3,6\quad\text{for}~~k=3;~s=6,9\quad\text{for}~~k=4;\\
			&s=8,9\quad\text{for}~~k=5;~s=7,8\quad\text{for}~~k=6;  \\ &s=4,7\quad\text{for}~~k=7;~s=1,4\quad\text{for}~~k=8.\\
		\end{split}
	\end{equation*}
	%We provide the explicit expressions of the coefficients of $\{q_k(x,y)\}_{k=1}^8$ in .
	
    \item[\textbf{Step 3}] Rewrite $q_0(x,y)$ as in \cite{levy_central_1999,levy2000compact,zhu_new_2016}:
    \begin{equation}\label{eq:rewrite_q0}
        q_0(x,y) = \gamma_0\left(\frac{1}{\gamma_0} q_0(x,y) - \sum\limits_{k=1}^8\frac{\gamma_k{\gamma_0}}q_k(x,y)\right)+\sum\limits_{k=1}^8\gamma_k q_k(x,y),
    \end{equation}
    where $\{\gamma_k\}_{k=0}^8$ is a set of positive linear weights with their sum being $1$. The linear weights control the balance between  optimal reconstruction accuracy and avoiding numerical oscillation. In our numerical tests, we set $\gamma_0=0.6$ and $\gamma_1=\ldots=\gamma_8=0.05$ for such balance.

    \item[\textbf{Step 4}] Compute the smoothness indicators of $\{q_k(x,y)\}_{k=0}^8$ \cite{jiang_efficient_1996}:
    \begin{equation*}\label{eq:smoothness_indicator_def}
        \begin{split}
            &\beta_0 = \frac{1}{\Delta x\Delta y}\sum\limits_{l_1+l_2\leq 2}\iint_{I_5}\left(\Delta x^{l_1}\Delta y^{l_2}\frac{\partial^{|l_1+l_2|}}{\partial_{l_1}\partial_{l_2}}q_0(x,y)\right)^2dxdy,\\
            &\beta_k = \frac{1}{\Delta x\Delta y}\sum\limits_{l_1+l_2\leq 1}\iint_{I_5}\left(\Delta x^{l_1}\Delta y^{l_2}\frac{\partial^{|l_1+l_2|}}{\partial_{l_1}\partial_{l_2}}q_k(x,y)\right)^2dxdy,   \text{for}~~k=1,\ldots,8.
        \end{split}
    \end{equation*}
    The explicit expressions of $\{\beta_k\}_{k=0}^8$ are given by
    \begin{equation}\label{eq:smoothness_indicator_expre}
        \begin{split}
            &\beta_0 = \left(a_2^{q_0}\right)^2+\left(a_3^{q_0}\right)^2+\frac{13}{3}\left(a_4^{q_0}\right)^2+\frac{7}{6}\left(a_5^{q_0}\right)^2+\frac{13}{3}\left(a_6^{q_0}\right)^2,\\
            &\beta_k = \left(a_2^{q_k}\right)^2+\left(a_3^{q_k}\right)^2\quad \text{for}~~k=1,2,\ldots,8.
        \end{split}
    \end{equation}

    \item[\textbf{Step 5}] Compute the nonlinear weights $\{\omega_{k}\}_{k=0}^8$ \cite{borges_improved_2008,zhu_new_2016}:
    \begin{equation}\label{eq:nonlinear_weights}
        \omega_k=\frac{\widetilde{\omega}_k}{\sum\limits_{l=0}^8\widetilde{\omega}_l},
    \end{equation}
    where
    \begin{equation}\label{eq:nonlinear_weights_tilde}
        \widetilde{\omega}_k=\gamma_k\left(1+\frac{\tau^{\frac{5}{4}}}{\beta_k+\epsilon}\right)\quad \text{for}~~k=0,1,\ldots,8
    \end{equation}
    with
    \begin{equation}\label{eq:tau}
        \tau=\frac{\sum\limits_{k=1}^8|\beta_0-\beta_k|}{8}.
    \end{equation}
    When the exact solution is smooth over the entire large stencil $\bigcup\limits_{s=1}^{9}I_s$, we can prove
    \begin{equation*}\label{eq:estimate_omega}
        \omega_k=\begin{cases}
            \gamma_k\left(1+O\left(\Delta^{\frac{3}{2}}\right)\right),\quad \text{if}~~Du|_{(x_i,y_j)}\neq 0~~ \text{and} ~~D^2u|_{(x_i,y_j)}\neq 0,\\
            \gamma_k\left(1+O\left(\Delta x\right)\right),\quad~~ \text{if}~~Du|_{(x_i,y_j)}= 0~~ \text{and} ~~D^2u|_{(x_i,y_j)}\neq 0,
        \end{cases}
    \end{equation*}
    by Taylor expansion.

    \item[\textbf{Step 6}] Construct the final reconstruction polynomial as follows:
    \begin{equation}\label{eq:final_construction_poly}
        u^{\text{WENO}}_{i,j}(x,y) = \omega_0\left(\frac{1}{\gamma_0} q_0(x,y)-\sum\limits_{k=1}^8\frac{\gamma_k}{\gamma_0}q_k(x,y)\right) + \sum\limits_{k=1}^8\omega_k q_k(x,y).
    \end{equation}
    \end{enumerate}

    Eventually, we define that $u^{\text{WENO}}(x,y)$ is the piecewise polynomial satisfying:
    \begin{equation}
        u^{\text{WENO}}(x,y) = u^{\text{WENO}}_{i,j}(x,y)\quad (x,y)\in I_{i,j}\quad\text{for all}~i,j.
    \end{equation}

\begin{remark}\label{remark:WENO_ZQ_accuracy}
    $u^{\text{WENO}}(x, y)$ offers a 3rd-order approximation to $u(x, y, t)$, provided $\{\overline{u}_{i,j}\}$ is sufficiently accurate. Assuming $\Delta x \sim \Delta y$, we can prove this as follows:
    \begin{align*}\label{eq:WENO_ZQ_accuracy}
        &u^{\text{WENO}}_{i,j}(x, y) - u(x, y, t) \nonumber\\
        &= \left(\gamma_0 + \omega_0 - \gamma_0\right) \left(\frac{1}{\gamma_0} q_0(x, y) - \sum_{k=1}^8 \frac{\gamma_k}{\gamma_0} q_k(x, y)\right) + \sum_{k=1}^8 \left(\gamma_k + \omega_k - \gamma_k\right) q_k(x, y) \nonumber\\
        &\quad - \left(\sum_{k=0}^{8} \gamma_k + \sum_{k=0}^{8} \left(\omega_k - \gamma_k\right)\right) u(x, y, t) \nonumber\\
        &=   \left(\omega_0 - \gamma_0\right) \Bigg(\frac{1}{\gamma_0} \left(q_0(x, y) - u(x, y, t)\right) - \sum_{k=1}^8 \frac{\gamma_k}{\gamma_0} \left(q_k(x, y) - u(x, y, t)\right)\Bigg) \nonumber\\
        &\quad +q_0(x, y) - u(x, y, t) + \sum_{k=1}^8 \left(\omega_k - \gamma_k\right) \left(q_k(x, y) - u(x, y, t)\right) \nonumber\\
        &= O(\Delta x^3) + O(\Delta x)\left(O(\Delta x^3) + O(\Delta x^2)\right) + O(\Delta x)O(\Delta x^2) \nonumber\\
        &= O(\Delta x^3), \quad \quad \text{for } (x,y)\in I_{i,j}.
    \end{align*}
\end{remark}

\bibliographystyle{siamplain}
\bibliography{Reference}
\end{document}